%% file: ResearchProject.tex
\documentclass{article}
\usepackage{graphicx} 
\usepackage{amsmath}
\usepackage{amsfonts}
\usepackage{amsthm}
\usepackage{amssymb}
\usepackage{xcolor}
\usepackage{mathrsfs}
\usepackage{breqn}
\usepackage{comment}
\usepackage[inline]{asymptote}
 
\title{Understanding Fermat's Last Theorem's Proofs}
\author{Alex Qiu, Tanish Sarathy, Spencer Nicklin, Michael Sun}
\date{22nd August 2025}

\newtheorem{theorem}{Theorem}[section]

\newtheorem{thm}{Theorem}[section]
\newtheorem{conj}{Conjecture}[section]
\newtheorem{cor}{Corollary}[section]
\newtheorem{lemma}[theorem]{Lemma}
\newtheorem{lem}[theorem]{Lemma}
\newtheorem{proposition}{Proposition}[section]
\newtheorem{prop}{Proposition}[section]
\theoremstyle{definition}
\newtheorem{definition}{Definition}[section]
\theoremstyle{definition}
\newtheorem{defn}{Definition}[section]

\theoremstyle{definition}
\newtheorem{eg}{Example}[section]

\newcommand{\SLZ}{\operatorname{SL}_2(\mathbb{Z})}

\newcommand{\mat}{\operatorname{Mat}}
\newcommand{\gl}{\operatorname{GL}}
\newcommand{\GL}{\operatorname{GL}}
\newcommand{\SL}{\operatorname{SL}}
\newcommand{\gal}{\operatorname{Gal}}
\newcommand{\Gal}{\operatorname{Gal}}

\newcommand{\tr}{\operatorname{tr}}

\newcommand{\Aut}{\operatorname{Aut}}
\newcommand{\tate}{\operatorname{Tate_{\ell}}}
\newcommand{\galq}{\operatorname{Gal(\overline{\Q},\Q)}}
\newcommand{\rad}{\operatorname{rad}}

\newcommand{\Frp}{\operatorname{Frob_{p}}}
\newcommand{\Frob}{\operatorname{Frob}}

\newcommand{\C}{\mathbb{C}}
\newcommand{\Z}{\mathbb{Z}}
\newcommand{\Q}{\mathbb{Q}}
\newcommand{\F}{\mathbb{F}}
\newcommand{\N}{\mathbb{N}}
\newcommand{\A}{\mathbb{A}}
\newcommand{\R}{\mathbb{R}}

\newcommand{\Mabcd}{\begin{pmatrix} a&b\\c&d \end{pmatrix}}

\newcommand{\Tr}{\operatorname{Tr}}

\begin{document}

\maketitle

\begin{abstract}
We take the perspective of an advanced high school student trying to understand the proof of Fermat's Last Theorem for the first time. We collect definitions and statements needed to summarise how Fermat's Last Theorem was first proved in 1995 as well as to see how the argument has simplified since then. We include a current timeline, outlines of proofs, background material and recent developments, as well as to organise the content in a way that is beginner friendly, offering a preview of what students may expect to learn more deeply in the future.
\end{abstract}

\tableofcontents

\include{intro}

\include{timelines}

\include{outlines}

\include{structure}

\include{ack}

\include{prelimElliptic}

\include{elliptic}

\include{prelimModular}

\include{modular}

\include{prelimGalois}

\include{galois}

\include{algnt}

\include{n234ABC}

\include{appendElliptic}

\include{appendModular}

\include{appendices}

\include{fourier}

\include{appendGalois}
















Playlist of useful videos

\begin{verbatim}
    https://www.youtube.com/playlist?list=PLC6DvfglS0p9EB6QyCtueJv3MOdgwWE_B
\end{verbatim}

Fermat conference.playlist
\begin{verbatim}
https://youtube.com/playlist?list=PL_jwwOG0kPgPNUaODd32I59XGZZAXJ7oj&si=mXYMEzDFjHy8lS1D    
\end{verbatim}

\end{document}

%% file: intro.tex
\section{Introduction}





The statement of Fermat's Last Theorem (FLT) is that for any integer $n>2$, there are no integers $x,y,z\neq0$ such that
$$x^n+y^n=z^n.$$
Despite its simplicity, FLT remained unproved for over 350 years, becoming one of the longest-standing and most infamous unsolved problems in the history of mathematics. It was finally resolved in 1995 by Andrew Wiles \cite{Wiles}, building on the work of generations of leading mathematicians.
\newline\newline
One is naturally curious about how such an infamous problem was finally resolved. Yet, to the frustration of any amateur attempting to read Wiles’s paper, they are confronted with an almost alien language, where every clarification seems to introduce ten times as many unfamiliar terms, each appearing more technical and contrived than the particular terms which are supposedly being clarified. Thus, one is quickly overwhelmed by a sense that there are infinitely many unknowns and obstacles to understanding what has happened. 
\newline\newline
In this article, we convert our despair into a summary of the key definitions and statements needed to get an initial handle into what is going on with the proof of FLT. We do not claim to understand the proofs of the main statements and in fact we replace the proof environment that usually follows them by further clarifications of the statement. We attempt to present all of the necessary definitions as well and try to understand them, offering a chronological order of where a student might first encounter them in their mathematical journey.
\newline\newline
We do not shy away from unravelling the definitions associated with the three main, and enormous, topics surrounding FLT: Elliptic Curves, Modular Forms and Galois representations. We are certainly not experts in any of these topics and will not try to hide our ignorance. When there is heavy detail, we will often selectively explain the proofs and theorems to match our own comprehension.
\newline\newline
We often hear that one needs to know all of these things and so much more before even beginning on such a quest to understand FLT, which while in principle sounds great, we know from experience it is idealistic and impractical. To begin to understand FLT one mustn't be a hostage to being logically prepared but also needs the right mix of motivation, hope and encouragement throughout their preparation.
\newline\newline
Concrete achievements for this article include being able to engage with the introduction of Wiles' paper where previously unfathomable, as well as to be able to understand the main statements that went into the first proof of FLT and updated proofs since then.

In the next Sections we present a timeline and summarise a logical outline of proofs of FLT before diving into the details in the Sections that follow.

%% file: timelines.tex
\section{Timeline for proving Fermat's Last Theorem}

There are three main components to our timeline:

Classical contributions from 1667, the first proof by Wiles in 1995, and subsequent developments. The rest of this article after the timeline will largely ignore the classical contributions.

\subsection{Classical Contributions}

These might be considered the initial direct attacks on the problem that have led to many algebraic foundations that have since exploded in depth and abstraction.

\begin{itemize}
    \item 
    1667: $n=4$ proved by Fermat. (Proof in Appendix \ref{n4})
    
\item   1753-1770: $n=3$ was proved by Euler with a gap that was filled by Legendre. (Another proof is in Appendix \ref{n3})

\item 1820 Germain: If $n=p$ is a prime such that $2p+1$ is also a prime there are no solutions for $x,y,z$ where $n$ does not divide any of the three numbers. Note that to this day it is still not known if there are infinitely many such primes. (We explain later on why it suffices to consider $n$ prime.)

\item 1825: Legendre and Dirichlet proved $n=5$.

\item 1838: Lam\'{e} proved $n=7$ and makes instructive errors in further proof attempts. Many developments in algebraic number theory at this time. 

\item 1850: Kummer proves that if $n$ is part of a specific set of primes, then no solution exists. This set includes all  primes up to 100 except 5, yet the size of this set is not known to be finite or infinite.

\item 1952: Vandiver proves up to $n=2000$; some review of Kummer's work played a role.

\item 1983: Faltings proves Mordell's conjecture, which when applied to FLT shows that for all $n$ there are only finitely many solutions such that $\gcd(x,y,z)=1$.

\item 1985: The ABC conjecture is formulated by Oesterlé–Masser which would imply FLT to be true for sufficiently large $n$. The statement can be found in the Appendix \ref{abc}.

\item 1985: Fermat's Last Theorem is still not known to be true for infinitely many primes $n$.

\item 1990 Buhler, Crandall, Ernvall, Metsanklya verified up to $n=4000000$ with computer assistance.

\end{itemize}


References for some of the above (and below) can be found in \cite{DDT}.
\newline\newline
We will need the $n=3,4$ cases for Wiles' proof and their proofs are in our Appendix (\ref{n4} and \ref{n3}). The $n=4$ proof is shorter, while $n=3$ is longer and takes more investment. We also state the ABC conjecture in the Appendix \ref{abc} and examine how it can solve FLT for all but finitely many cases of $n$.
\newline\newline
Aside from these, we ignore the classical contributions for the remainder of the article. Some higher cases of $n$ may be needed to start the updated version of Wiles' proof.


\subsection{Wiles' Proof in 1995}

Here we intertwine a timeline for Wiles' proof with another timeline for an updated proof of FLT in blue but due to technical requirements we have separated them into two. We do not assume the reader is familiar with the following key terms and treat them as black boxes for clarity:
\newline\newline
\textbf{Key Terms:} Taniyama–Shimura–Weil Conjecture (TSW), Elliptic Curve, (Epsilon) $\epsilon$-conjecture, Ribet’s Theorem.

\begin{itemize}

\item \textbf{1955}: Taniyama proposes a surprising connection between elliptic curves and modular forms. 

\item \textbf{1957}: Shimura refines the conjecture with Taniyama, giving it a more rigorous mathematical formulation.

\item \textbf{1958}: Taniyama commits suicide at age 31.

\item \textbf{1963}: Andrew Wiles aged 10 reads about FLT in Bell's \emph{The Last Problem} and decides he would be the first to solve it!

\item \textbf{1967}: Weil provides supporting evidence and reformulations, solidifying the conjecture’s credibility and making more people aware of it. It becomes known as the \emph{Taniyama–Shimura–Weil} conjecture (TSW).

\item \textbf{1969}: Hellegouarch uses hypothetical solution to Fermat's equation to define an Elliptic curve. Serre also mentions Hurwitz 1886 as a pioneer of this idea.

\item \textbf{1982}: Frey suggests that TSW implies FLT using a curve like that of Hellegouarch.


\item \textbf{1985}: Serre rigorously formalises Frey's suggestion and proves it up to a statement he called $\epsilon$ or the (epsilon) $\epsilon$-conjecture \cite{Serre}. 

Serre showed that if the $\epsilon$-conjecture is true, then TSW implies FLT.

\item \textbf{1986}: Ribet proves the (epsilon) $\epsilon$-conjecture \cite{Ribet}. He recounts that he proved a special case and consulted Barry Mazur about the general case, to which Mazur observed that Ribet's argument for the special case did not lose generality and also works for the general case. This became known as Ribet's Theorem. Therefore, TSW implies FLT.



\item \textbf{1986}:
Wiles upon learning this, immediately begins work on TSW.

\item \textbf{1993}: Wiles announces proof of FLT via a part of TSW.

\item \textbf{1994}: A gap was found in Wiles' proof by Nick Katz.

\item \textbf{1995}:
Wiles plugs the gap with his former student Richard Taylor there to assist \cite{Wiles}. Fermat's Last Theorem is proved!

\item Foundational work and helpful contributors that are mentioned by Wiles' introduction \cite{Wiles} include:  Mazur, Hida, Langlands-Tunnel, Eichler, Poitou, Tate, Iwasawa, Kunz, Flack, Ramakrishna, Rubin, Kolyvagin, and we acknowledge we may have missed some key people due to our ignorance.


\end{itemize}

\subsection{Updates to Proof via Serre’s Modularity Conjecture}

This component mainly concerns the developments related to Serre's modularity conjecture, which played a key role in both the original proof and updated proofs. It is quite likely to be incomplete.
\begin{itemize}

\item \textbf{1916}: Ramanujan publishes several conjectures regarding the so called $\tau$-function, some involving the prime 691. Dies in 1920 at age 33. 

\item \textbf{1970s}: Deligne, building on Serre and Swinnerton-Dyer, proves some of Ramanujan’s more difficult $\tau$ conjectures using Galois representations.

\item \textbf{1973-1975}: Serre introduces his modularity conjecture, relating Galois representations to modular forms, in a letter to Tate. Tate replies with the proof of the first cases. Serre makes conjecture public.

 The $\epsilon$-conjecture is a part of Serre's modularity conjecture in same paper.

\item \textbf{1986}: Serre also notes that Colmez observed that Serre's modularity conjecture implies the Taniyama-Shimura-Weil conjecture.

   \item \textbf{1999}: Taniyama-Shimura-Weil conjecture proof completed by Breuil, Conrad, Diamond and Taylor building on work of Wiles, Wiles-Taylor. Now known as the Modularity Theorem \cite{BCDT}. FLT proof updated.

\item \textbf{2004-2008}: Serre's modularity conjecture was proved by Khare and Wintenberger \cite{SMC}. FLT proof updated.

\item \textbf{2019}: Wintenberger dies at age 65. Shimura dies at age 89. Tate dies aged 94.

\item \textbf{2024-2025-}: On-going project led by Buzzard to formalise a modern proof of FLT in Lean with a 5 year EPSRC grant.

\end{itemize}

%% file: outlines.tex
\section{Outline of 1st Proof of Fermat's Last Theorem}
\subsection{Initial Summary}

We do not assume the reader is familiar with the following key terms, which will be introduced here: Frey's curve, Elliptic curve, Modular.
\newline\newline
We will first give an outline of the structure of the first proof of FLT  without elaborating on the term \emph{Modular} nor mentioning \emph{Galois Representations}.

\begin{itemize}
\item Cases $n=3,4$ are done using elementary methods
\item Assume FLT is false and define Frey's curve.

\item Ribet proved that Frey's curve is not modular.

\item Wiles proved that Frey's curve is modular.
\item Contradiction, hence Fermat's Last Theorem is true.
\end{itemize}

\subsubsection{Cases $n=3,4$}

The proofs of these are in the Appendix (\ref{n3},\ref{n4}). 


We explain why it now suffices to consider the case when $n$ is an odd prime $P\geq5$:
\begin{theorem}
    If Fermat's Last Theorem is proved true when $n$ is an odd prime $P\geq5$ or $4$ and $a,b,c$ are pairwise coprime, it must also be true in general.
\end{theorem}
\begin{proof}
    We show a solution for $n>2$ gives a solution for $n=4$ or $n$ and odd prime.
    
    We will do two cases: $n$ is a power of $2$ and $n$ is not a power of $2$: 
    \newline
   \begin{enumerate}
       \item Suppose that $n=2^k$ satisfies Fermat's equation for some $k\in \Z_{\geq2}$. We have $$a^{2^k}+b^{2^k}=c^{2^k}$$ $$\iff (a^{2^{k-2}})^4+(b^{2^{k-2}})^4=(c^{2^{k-2}})^4.$$ But the case $n=4$ has no non-zero solutions.
       \item Now suppose that $n=pr$ satisfies Fermat's equation for some $r\in \Z^+$ and odd prime $p$. Then $$a^{pr}+b^{pr}=c^{pr}$$ $$\iff (a^r)^p+(b^r)^p=(c^r)^p$$ which is only true if $n=p$ satisfies Fermat's equation. 
   \end{enumerate}
   Since we know the $n=3$ case, we can assume $P\geq 5$.
   \newline\newline
   Now if any pair of $a,b,c$ satisfying the equation share a prime factor, then so does the third and we can just divide the $n$-th power of this prime from the equation. We can do this until no such common factors exist.
   \end{proof}

\subsubsection{Frey's Curve and its Properties}
Keywords: Discriminant, Non-Singular, Good/Bad/Additive/Multiplicative Reduction, Semistable. 
\newline\newline
Assume FLT is false:

That is, there exists $a,b,c$ pairwise coprime, and a prime $P\geq 5$ prime, such that
$$a^P+b^P=c^P.$$ 

Frey's curve \cite{Frey} is then defined from the equation
$$y^2=x(x-a^P)(x+b^P),$$
and is an example of an \emph{Elliptic Curve}.
\newline\newline
We can temporarily take as a definition of an elliptic curve those equations of the form $y^2$ equals a cubic in $x$ with integer coefficients.
\newline\newline
The reader may be familiar with the concept of a discriminant, especially in regard to quadratic equations. If we have a quadratic expression: $$(x-r_1)(x-r_2)=x^2-(r_1+r_2)x+r_1r_2$$ then the discriminant is $$(r_1+r_2)^2-4r_1r_2=r_1^2-2r_1r_2+r_2^2=(r_1-r_2)^2$$ which is the square of the difference of the roots. This is defined for cubics by:
\begin{definition}[Discriminant]\label{discdef}
Suppose a cubic polynomial has roots $r_1,r_2, r_3$, with leading coefficient 1, then its discriminant is given by 

$$\Delta=(r_1-r_2)^2(r_1-r_3)^2(r_2-r_3)^2.$$
\end{definition}

   We also take this to be the \emph{discriminant} of the elliptic curve 
   $$y^2=(x-r_1)(x-r_2)(x-r_3).$$
\newline\newline
$\Delta$ detects whether or not the roots coincide. Part of the definition of an elliptic curve is that it is \emph{non-singular}, which means that none of these roots coincide or equivalently 
$$\Delta\neq0.$$
\begin{prop}
    Frey's curve is non-singular. Moreover, 
    $$\Delta=(abc)^{2p}.$$
\end{prop}
\begin{proof}
  For Frey's Curve, $r_1=0,r_2=a^P,r_3=-b^P$. Hence
  $$\Delta=(a^P-0)^2(-b^P-0)^2(-b^P-a^P)^2=(abc)^{2P}\neq0.$$  
\end{proof}

In analysing such equations it can be helpful to reduce it modulo a prime number $p$, i.e. setting $p=0$ which turns this into a finite problem that can be understood explicitly.
Given a prime $p$, one can consider the \emph{reduction at $p$} of the Elliptic Curve by reducing all the coefficients of the curve modulo $p$. 
One issue that arises is that the $\Delta$ might be divisible by $p$ and the new equation becomes singular.

\begin{definition}\label{reduction}
    
If $p$ does not divide $\Delta$,  the discriminant is non-zero modulo $p$, which we call \emph{good reduction at $p$}. If $p \mid \Delta$,  the discriminant is $0$ modulo $p$, which we call \emph{bad reduction at $p$}. There are two types of bad reduction: \emph{multiplicative} reduction is if exactly two roots coincide and \emph{additive} reduction is if all three roots coincide.

\end{definition}
\begin{definition} [Semi-stable]
We call an elliptic curve \emph{semi-stable} if it has only good or multiplicative reduction at every prime, or equivalently it never has additive reduction at a prime.
\end{definition}

\begin{prop}
    Frey's curve is semistable.
\end{prop}
\begin{proof}
    Since $a,b,c$ are coprime, if a prime divides $a^P$, it does not divide $b^P$. This means when reducing it any prime the roots $a^P$ and $b^P$ never coincide, hence the reduction at any prime is not additive.
\end{proof}

\subsubsection{Ribet and Wiles}
\begin{thm}[Ribet \cite{Ribet} Corollary 1.2]
     If every semistable elliptic curve is modular then Fermat's Last Theorem is true.
\end{thm} 

The corollary in \cite{Ribet} states the assumption as all elliptic curves are modular but it suffices to just have semistable elliptic curves be modular.

\begin{thm}[Wiles \cite{Wiles}]
     Every semistable elliptic curve over $\Q$ is modular.
\end{thm}

Note that Wiles' theorem is independent of the Frey curve, whereas the statement given for Ribet is another general theorem applied to the Frey curve. This summary is about as clean as you can get without expanding further on key concepts like modular forms and Galois representations. In particular, the major contributions of Ribet and Wiles are just short statements for now.
\newline\newline
Modular is mentioned but not elaborated in any detail for the time being, one word of warning is that the first definition of modular form one finds is actually not the definition of modular but usually just of level 1. So this initial summary is a proof outline up to a definition of the word modular.


\subsection{Role of Modular Forms}

Here we expand on what's behind the word modular.

Key Terms: Conductor, Weight, Level.

\subsubsection{Conductor and Point Counts of Semistable Elliptic Curves}
To keep track of the curve’s reduction behavior across all primes, we have:
\begin{definition}The \emph{conductor} $N$ is defined for semistable elliptic curves as the product of all primes where the elliptic curve has bad reduction.
\end{definition}
A more general definition exists when the curve is not semistable (which can be found in the appendix), from which, $N$ isn't necessarily square-free and it being square-free (i.e. $N$ being the product of distinct primes) is equivalent to the curve being semistable. 
\newline\newline
Let $\rad(n)$ denote the product of distinct prime factors of $n$. Note $\rad(n)$ is square-free for all $n$, this notation allows the nice expression:

\begin{prop} The conductor of the Frey curve is given by
$$N=\rad(abc).$$    
\end{prop}
\begin{proof}
 This is immediate from the definition of the conductor above since Frey's curve is semistable.   
\end{proof}
So, the conductor of the Frey curve is then the product of all primes dividing $abc$. Note that $N$ is even as one of $a,b,c$ is even. 
\newline \newline Now if we let $E_p$ be the curve defined by the reduction of the Elliptic Curve modulo $p\nmid\Delta$, i.e. where the curve has good reduction, define $$A_p=p+1-|E_p|$$ where $|E_p|$ denotes the number of solutions $(x,y)$ which satisfy the equation of the Elliptic Curve modulo $p$. If the curve has bad reduction at $p$ there are 2 cases:
\begin{itemize}
    \item Additive reduction: $A_p=0$,
    \item Multiplicative reduction: $A_p=\pm 1$, sign known in each case.
\end{itemize}
We further extend this to the prime powers. Define $A_1=1$, and for the prime powers $A_{p^k}$ can be defined by the recurrence relation for $k\geq2$
$$A_{p^k}=\begin{cases}
A_pA_{p^{k-1}}&\text{if $p$ is of bad reduction}\\
A_pA_{p^{k-1}}-pA_{p^{k-2}}&\text{if $p$ is of good reduction}
\end{cases}$$

We can extend this definition to a general $n$ multiplicatively, e.g. for coprime $p,q$, we define $A_{pq}=A_pA_q$.

\subsubsection{Preliminary Definition of Modular}

The word modular refers to the ability to associate a modular form.

This amounts to taking the sequence $A_n$ and defining the following complex valued function
$$f(z)=\sum_{n=0}^{\infty} A_n e^{2\pi i nz},$$

for all $z\in\C$ with strictly positive imaginary part. Let us write $\mathcal{H}$ for the set of all such $z$. This definition necessitates the series to converge in $\mathcal{H}$.
\begin{itemize}
    \item It is called a \emph{cusp form} if $A_0=0$.
\end{itemize}

We also see that such a function is necessarily periodic, that is $\forall z$ replacing $z$ with $z+1$ changes every exponent by a multiple of $2\pi i$ and hence:
$$f(z+1)=f(z).$$
This is a special case of the more general requirement that there exist integers $k,N$ such that
$$f\left(\frac{az+b}{cz+d}\right)=(cz+d)^kf(z)$$
for all $z\in \mathcal{H}$ where $a,b,c,d\in\Z$, $c=0\mod N$, and  $ad-bc=1$.
\newline \newline
We see for example, if we take $a=1,b=1,c=0,d=1$, we recover the condition of being periodic.
\newline \newline
The $N$ is called the \emph{level} and $k$ the \emph{weight}. Other parts of the definition will be introduced in the modular forms section.

\subsubsection{Ribet's Proof Revisited}
You might notice that the letter $N$ is already taken, but this is explained by the following elaboration of Ribet's contribution:
\begin{itemize}
    \item It was known that if an elliptic curve is modular, then the level of its modular form, is of the same value, $N$, as the conductor of the elliptic curve.

    \item Ribet showed that for odd primes $p\neq P$, if $p\mid N$ and $p^2 \nmid N$ then we can replace our modular form of level $N$ by another modular form of level $N/p$ and weight 2. 

    \item Since for the Frey curve, $N$ is square free and even, this process can be iterated until $N=2$ or $N=2P$ with weight 2. If $P\mid N$, Ribet removes it using a result of Mazur. So we are left with a level two weight two cusp form.

    \item It was previously known that there are no non-zero cusp forms of level 2 and weight 2. 

\end{itemize}

We see that revealing the role of modular forms expands the elliptic curve story a little bit more and unlocks a chunk of disproving the Frey curve is modular, while the part about elliptic curves being modular is still basically locked. You might also be wondering how the $P$ is used. To really unlock these parts of the story, we will need to talk about Galois representations.

\subsection{Role of Galois Representations}

Keywords (we don't assume you already know): Galois Representation, $l$-adic, Point at Infinity, Abelian Group, Torsion Points
\newline\newline
We now start mentioning Galois representations somewhat like how modular forms were mentioned in the previous or first subsection, with proper definitions deferred to another section.
\newline \newline
The word Galois refers to Galois theory, which in turn is named after the mathematician Galois who famously died in a duel. We will use Galois representations as a black box in this subsection and try to introduce its origins, the two main types we will consider, as well as their place in the proof.




\subsubsection{Every Elliptic Curve has a Galois Representation}

We try to introduce the origins of Galois representations from elliptic curves.
\newline\newline
Given the points of an elliptic curve, there is a meaningful way in which we can add any two points $A$ and $B$ among them:
\begin{enumerate}
    \item Take the line $AB$ and intersect it with the elliptic curve at a third point $C$.
    \item Then take the vertical line through $C$ and intersect it again with the elliptic curve at $C'$.
\end{enumerate}
 This final intersection point is defined to be $A+B=C'.$ This can be seen in the diagram below with the elliptic curve in green, line $AB$ in black, and $A+B$ as the second intersection of the vertical line through C with the elliptic curve.
\newline\newline
  \includegraphics[scale=0.3]{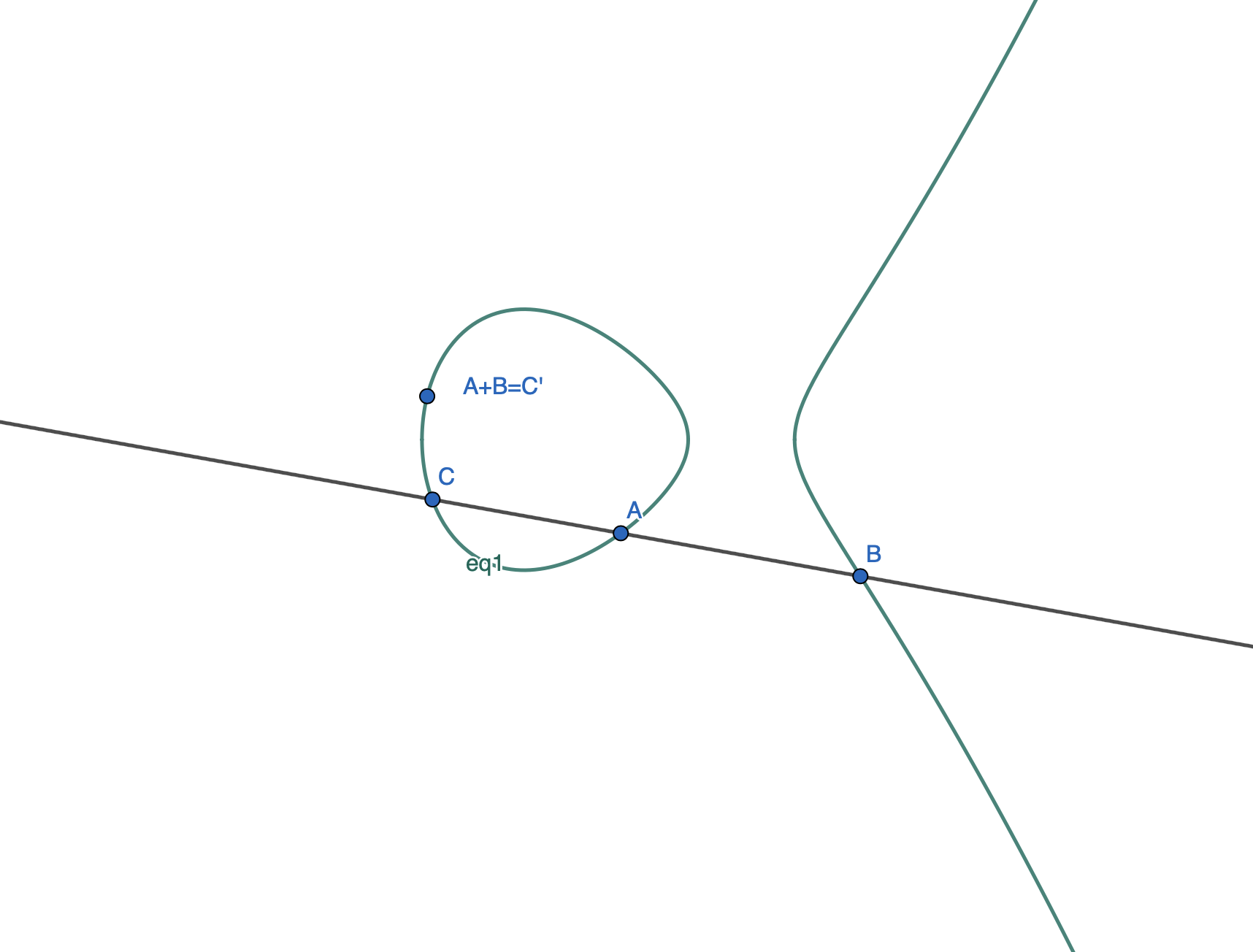}

You may notice that when both points $A$ and $B$ line on a vertical line, there is no third point of intersection with the elliptic curve, disrupting the definition of their sum. 
\newline\newline
One way to account for this is to add an extra point $\mathcal{O}$ called \emph{a point at infinity} and define the sum of any two vertical points to equal this point. 
\newline\newline
In this way we see that the point $\mathcal{O}$ behaves like zero in that for all points on the curve
$$Q+\mathcal{O}=Q.$$

We also see that two points $Q,R$ on a vertical line are negatives of each other in the sense:

$$Q+R=\mathcal{O}$$

We are now able to give a working definition for elliptic curve:

\begin{defn}\label{y2cubic}
    An elliptic curve $E$ over $\Q$ are all the rational pairs $(x,y)$ satisfying some equation $y^2=x^3+a_2x^2+a_4x+a_6$ with rational constants $a_2,a_4,a_6$, together with an extra point $\mathcal{O}$, called the point at infinity, such that $\Delta\neq0$.
\end{defn}


Any time an addition is defined, we get a notion of scaling by an integer as repeated addition or repeated subtraction:
$$mQ=\underbrace{Q+Q+\dots+Q}_{m\text{ times}}.$$

We are now ready to define the source of Galois representations:

\begin{defn}[$m$-Torsion Points]\label{torsionpt}
$$E[m]=\{Q\in E: m Q=\mathcal{O}\}.$$
\end{defn}

This torsion is the source of our Galois representations in a similar sense to the sequence of point counts being the raw data for our modular form. This is also where the $P$ from our initial assumption comes in as we specifically look at the set of $P$ torsion points, denoted by $E[P]$.

\subsubsection{The Two Types of Galois Representations}

Looking at all the powers of a prime $l$ at once leads to $l$-adic Galois representations, while just looking at $l$ torsion leads to mod $l$ Galois representation.

\subsubsection{Ribet and Wiles revisited}
 Serre's full conjecture was that Galois representations under certain conditions were modular and further specified how to get a minimal level and weight.

 The part about the minimal level and weight is called the epsilon conjecture.

\begin{itemize}
\item Ribet proves the epsilon conjecture and applies it to the mod $P$ Galois representation of Frey's curve.

\item Wiles proves that the $3$-adic Galois representation is modular for semi stable elliptic curves and uses known results to conclude that the elliptic curve is modular in general. Part of this was already having the mod $3$ Galois representation modular within reach from known results.
\end{itemize}

\subsubsection{Collection of Key Statements}

Here are some examples of the theorems in the actual papers whose statements we hope to understand. We put some here upfront to set our goals for which definitions to include.

\begin{eg}[Ribet's theorem]\label{Ribet}
Ribet typo in actual paper, no one read to end, though at the start there is a typo $\rho$ that's supposed to be a $p$

    \begin{center}
        \includegraphics[width=0.8\linewidth]{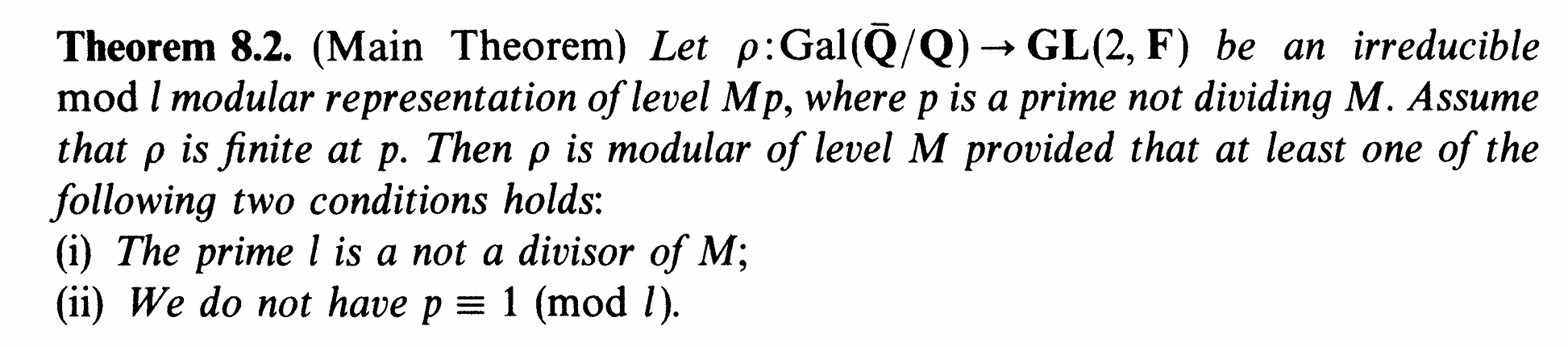}
    \end{center}

        We will be able to define everything in this statement by the end except for finite, which is defined in Serre \cite{Serre} and also showed to be satisfied when needed.
\end{eg}

\begin{eg}[Theorems in Wiles \cite{Wiles}]
Highlighted in the introduction of \cite{Wiles}
\begin{center}
    \includegraphics[width=0.8\linewidth]{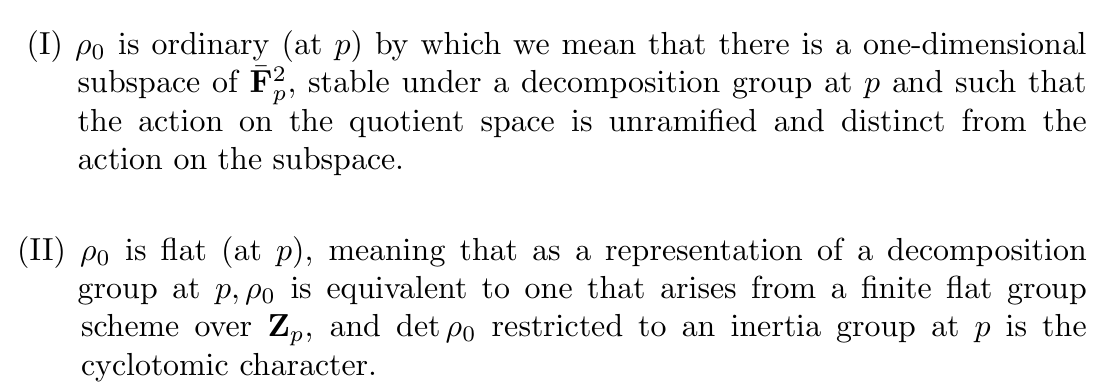}

    \includegraphics[width=0.5\linewidth]{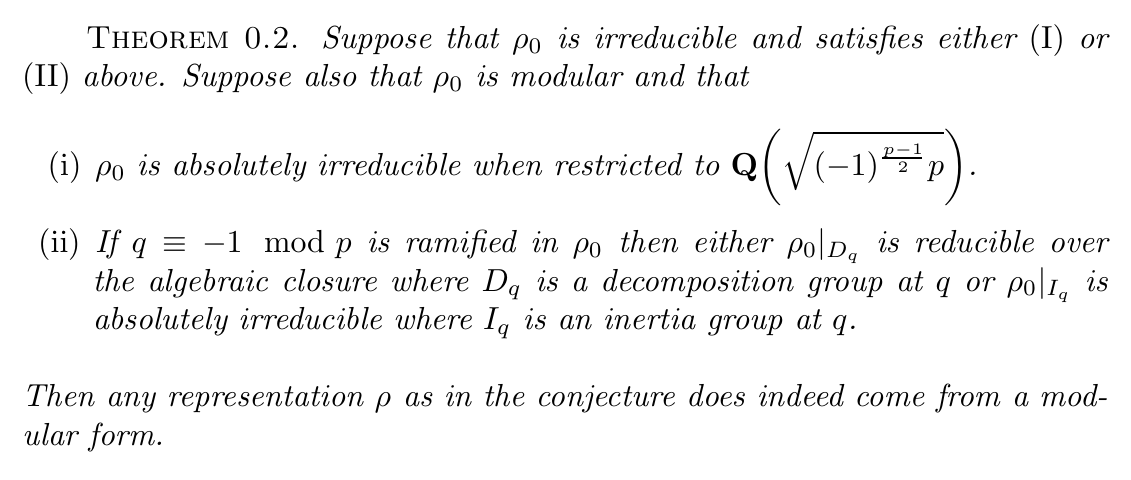}\includegraphics[width=0.5\linewidth]{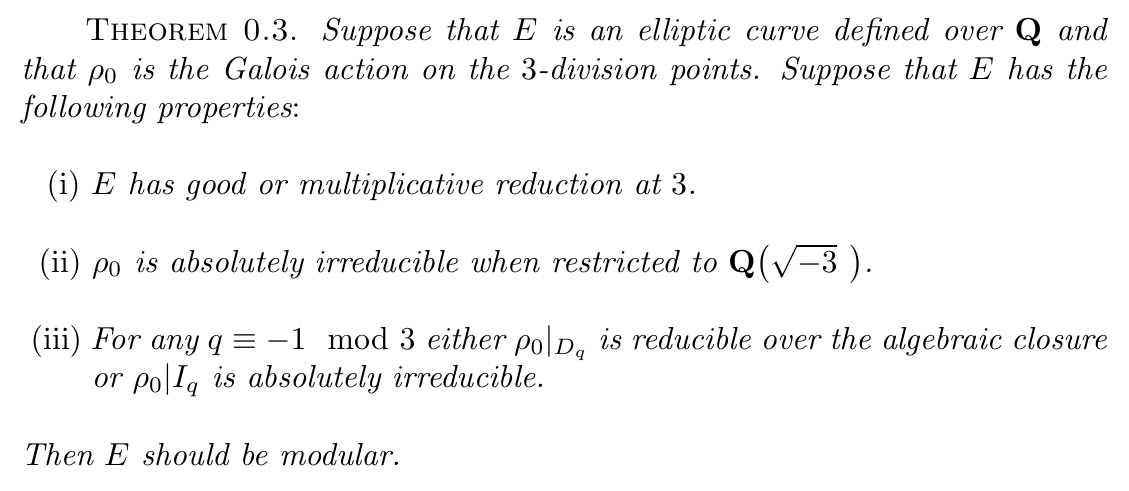}
    \end{center}

   We can explain almost all of the definitions needed to state these, except currently for ordinary, finite flat group scheme, we have not yet discussed characters. We are far from a deeper understanding of these statements.
\end{eg}

\section{Outline of a modern Proof of FLT}




\begin{itemize}
    \item Everything up to the Frey curve and its mod $P$ Galois representation is the same, possibly needing a few more base cases than just $n=3.4$.

    \item Serre shows in \cite{Serre} that Serre's modularity conjecture imples FLT.
    
    He does this by showing the Galois representation from the Frey curve satisfies the Serre conjecture.

     This then shows that the Galois representation from the Frey curve leads to a cusp form of level 2 and weight 2, which is the same contradiction as before.
    
    \item Khare-Wintenberger prove Serre's modularity conjecture \cite{SMC}.

\end{itemize}

\subsection{Serre's modularity Conjecture} Serre's Conjecture implies both the $\epsilon$-conjecture and the TSW Conjecture but instead of using those two statements, the Serre conjecture gives a contradiction to prove FLT directly. The contradiction is ultimately the same one with no level 2 weight 2 cusp forms. 

\subsubsection{Collection of Statements}

\begin{eg}[Serre's Conjecture and Translation \cite{Serre}]

\begin{center}

    \includegraphics[width=0.5\linewidth]{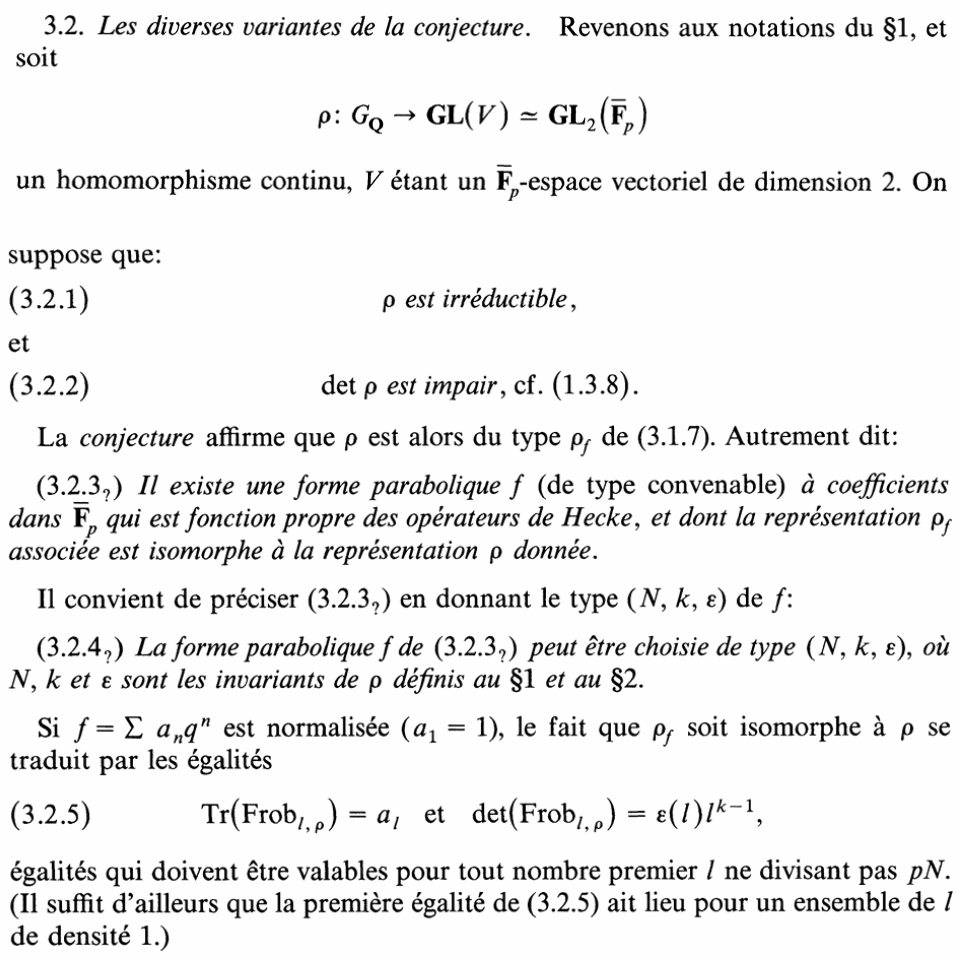}\includegraphics[width=0.5\linewidth]{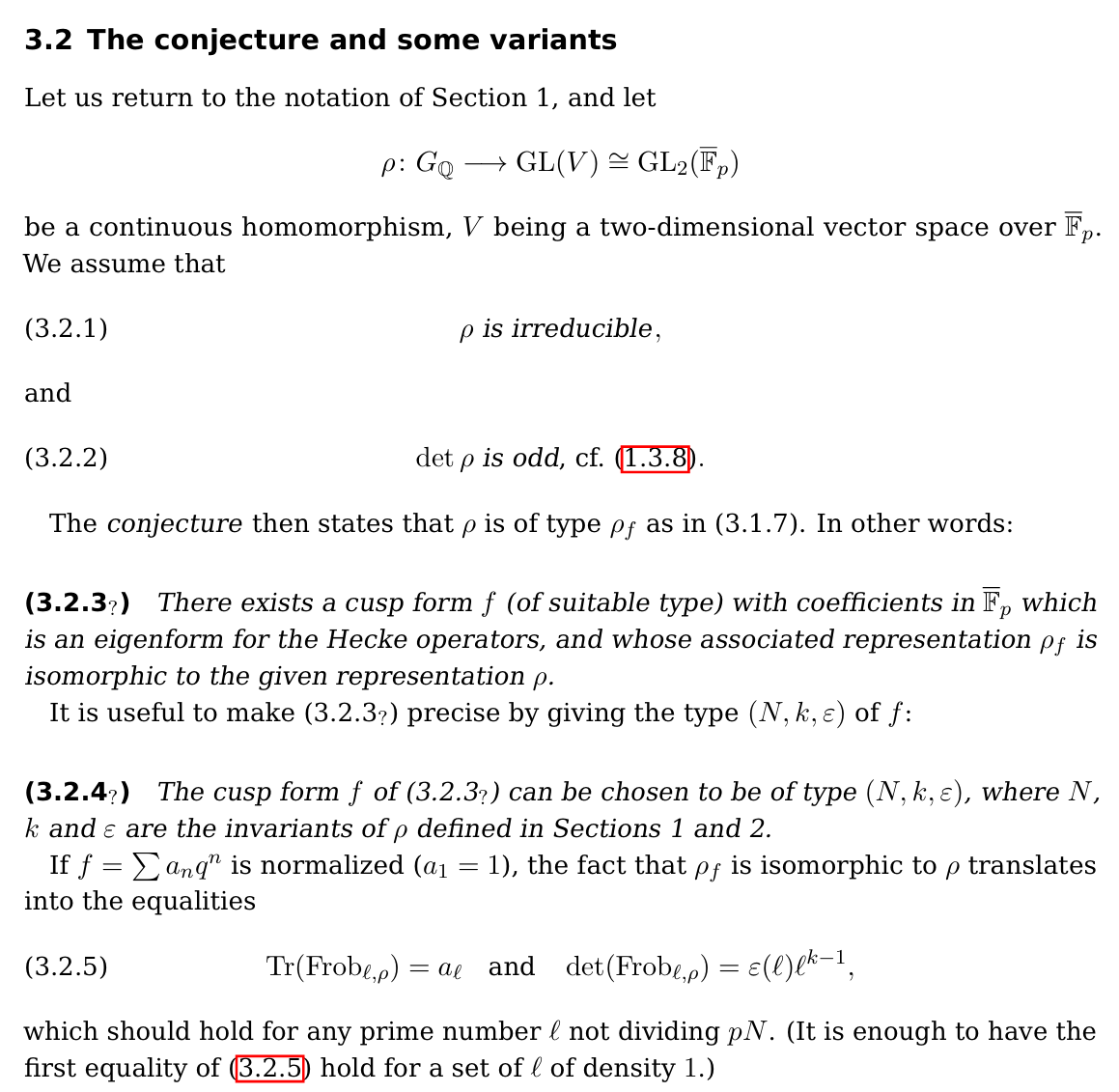}
\end{center}
\end{eg} 


%% file: structure.tex
\section{Structure of Article}
The goal of this section is to outline the structure of the rest of the article, especially with regard to the preliminary sections and appendices.

\subsection{Preliminaries}

The reader is assumed to be familiar with basic notions and common notation for functions, modular arithmetic and complex numbers.
\newline\newline
The preliminary sections introduce standard content that you might find in undergraduate courses that we do not assume the reader has completed. It is content we believe every mathematician will encounter at some point in their education. We try to spread these out as needed rather than all in one place, so that one comes across them as they require the information for the following section. The later preliminaries will build on the earlier preliminaries.
\newline\newline
 While much of preliminary content is not central to our discussion, it is fundamental mathematics that is needed to proceed. Students should get an idea of what is ahead in their education and will also likely be able to master some components as needed, just as some of the authors have. 





\subsection{Main Sections}

The main sections are Elliptic curves, Modular Forms, and Galois Representations in that order. Their difficulty and dependencies increase dramatically with each section. Some content that might usually go into preliminaries, we put into the main sections if it is more specialised, advanced or central to the discussion.

\subsection{Appendices}

Information that take us too far afield from the main discussion in terms of length or prerequisites are explored in the Appendices. They are not necessarily a systematic treatment of the topics, but focus on the specific explorations that the authors embarked on.

%% file: ack.tex
\subsection{Acknowledgements}

This project was initiated by Tanish at the start of the school year with a deadline set for completion by the July school holidays. At the outset, none of us fully anticipated the scale or depth of what lay ahead. Spencer joined midway through the project after completing his UK Team Selection Tests in Cambridge. Both Spencer and Tanish are currently in Year 12, and it is particularly admirable that they chose to undertake this work in the midst of a demanding HSC graduation from high school year.
\newline\newline
Alex, a member of last year’s Australian IMO team, is now entering undergraduate studies and is set to begin at Cambridge at the end of the year. All of us shared a strong desire to understand the mathematics behind Fermat’s Last Theorem, and I was especially motivated to explore whether the proof could be made accessible to high school students. Our goal was to distil the essence of the argument and consider how much of it could be regarded as “standard,” noting that the boundary between elementary and advanced mathematics has shifted significantly—even in the last ten years.
\newline\newline
The project grew rapidly beyond our expectations. In a matter of months, we produced over 100 pages of written material, alongside a collection of video recordings and supporting references. It has been a deeply rewarding and exciting experience for all involved.
\newline\newline
We are especially grateful to our families for their support throughout this journey, in particular, Shikha, Laura, Kathy, Pauline and Lisa. Finally, we sincerely thank Dr. Michael Sun’s School of Maths for fostering a strong supportive community environment—one in which students can engage with mathematics not in competition, but in collaboration and shared discovery.

\subsubsection{Email addresses}

\begin{itemize}
\item 
alexqiu2007@outlook.com
    \item tanish.s@icloud.com
    \item 
spencerhuxtablenicklin@gmail.com
    \item drmichaelsun@outlook.com
\end{itemize}

%% file: prelimElliptic.tex
\section{Preliminaries I}
In this section we introduce the following terms for the first time so that we can use them in later sections, especially for elliptic curves.

\subsection{Abstract Algebra}
Here are some terms that one might come across traditionally in second or third year of undergraduate courses but could learn earlier without much needed prior knowledge. We do not go into much theory but even just packaging the words can be very useful for later.
\newline\newline
Keywords: Abelian Group, Subgroup, Torsion Subgroup, Commutative Ring, Field, Characteristic, Vector Space, Module, Direct Sum  

\subsubsection{Groups, Abelian Groups}
We first define a group without going into group theory

\begin{definition} [Group]
    A group is a set $G$ with a binary operation, $*$, such that the operation is \emph{associative}:
    $$(a*b)*c=a*(b*c)$$
    for all $a,b,c\in G$, it has an \emph{identity element} $e$:
    $$a*e=a=e*a$$ 
    for all $a\in G$, and finally, every $a\in G$ has an \emph{inverse element} $b$, such that
    
    $$a*b=e=b*a.$$ 
\end{definition}
The main property of emphasis in a group is that everything is invertible. For example this means that if we can add, we can also subtract, and if we can multiply, we can also divide. Associativity is in some sense present in everything, while even if something doesn't have an identity, one can usually just add one artificially.

\begin{definition}[Subgroup]
    A \emph{subgroup} $H$ of a group $G$ is a subset that is closed under the operation of that whole group:

    $$h_1h_2\in H$$
for all $h_1,h_2\in H$ and is itself a group under this restricted operation.

\end{definition}

\begin{eg}
    Given a set $X$, the set of invertible functions $f:X\to X$ with operation as composition of functions is a key example of a group.
\end{eg}
We see that in this example there is exactly one identity function and every function has a unique inverse. This is true in general:
\begin{lem}
    The identity and inverses are uniquely defined.
\end{lem}
\begin{proof}
\textbf{Identity is Unique:}

Suppose $e$ and $e'$ are identities.

\[
e = e * e' = e'.
\]

\textbf{Inverse is Unique:}

Suppose $b$ has inverse $c$, and another inverse $d$.

\[
c = c * e = c * (b * d) = (c * b) * d = e * d = d.
\]

\end{proof}
If the operation can be carried out in either order then we have an \emph{Abelian} group:

\begin{definition}[Abelian Group]
    An \emph{Abelian} group is a group $G,*$ whose operation $*$ is \emph{commutative}:

$$a*b=b*a$$
    for all $a,b\in G$. It is common for Abelian groups to have $*$ written as $+$ and the operation called addition with identity $0$.
\end{definition}

\begin{eg}
   The set of integers   $\Z$ together with addition is a key example of an Abelian group.
\end{eg}

\begin{eg}  The set integers modulo $n$, $\Z/n\Z$, together with addition is another key example of an Abelian group. It is an example of a quotient group of $\Z$ by $n\Z$, which can be thought of as a group formed by setting the elements of the subgroup $n\Z$ to zero in $\Z$ along with its consequences (for example, if two numbers differ by $n$, then they are now equal). A similar construction can be made for any abelian group and one of its subgroups.
\end{eg}

\begin{eg}
The set of symmetries of an equilateral triangle (ways of rotating and reflecting it which preserve its appearance) form a group under the operation of composition, i.e. doing one symmetry after another, but it is not an abelian group. For instance, horizonally reflecting the triangle and then rotating it $120^\circ$ clockwise is not the same symmetry as rotating it $120^\circ$ clockwise and then horizontally reflecting it, so the operation is not commutative.
\end{eg}

\begin{definition}[Cyclic Group]
    A finite cyclic group is a group with a finite number of elements which can all be reached by repeatedly applying the group operation on one element, known as the \emph{generator}. For example, the set of integers modulo n: $\{0,1,2,...,n-1\}$ denoted $\mathbb{Z}/n\mathbb{Z}$ is a finite cyclic group under addition with generator $1$, while $\Z,+$ is an example of an infinite cyclic group.
\end{definition}

\begin{itemize}
    \item One of our key checkpoints is to show that the sum of points defined in the outline makes an elliptic curve into an Abelian group.
     \item Notice $\Z/n\Z$ has the property that adding $1$ $n$ times gets back to $0$, this is an example of \emph{torsion}, which can be defined for any Abelian group: 
\end{itemize}

\begin{definition} [Scalar Multiplication by Integers]
     Let $G$ be an Abelian group. For $n \in\Z$ and $g\in G$, define 
     $ng$ by
     $$ng = \underbrace{g + g + \cdots + g}_{n\ \text{times}}$$
     when $n\geq0$ and $(-n)(-g)$ when $n<0$.
\end{definition}%

Now we have the following definition:
\begin{definition}[Torsion Subgroup] Given an Abelian group $G$, $+$ and an integer $n$, the set 
$$G[n]=\{g\in G:ng=0\}$$
is a group under the same addition, called the $n$-torsion subgroup. The \emph{torsion subgroup} of $G$ is 
$$G=\bigcup_n G[n],$$
which consists of all the torsion elements of the group.
\newline\newline
Our key milestone for encountering a torsion subgroup is when we get the torsion points of an elliptic curve.
    
\end{definition}

Here is a simple construction to build bigger Abelian groups:

\begin{definition}
    Give two Abelian groups $A$ and $B$, define an Abelian group $A\oplus B$ by taking the set $A\times B$ of ordered pairs $(a,b)$ with $a\in A$ and $b\in B$ and define addition \emph{pointwise}:
    $$(a_1,b_1)+(a_2,b_2)=(a_1+a_2,b_1+b_2)$$
    for all $a_1,a_2\in A$ and $b_1,b_2\in B$.
\end{definition}

We will use this to show many key groups associated with elliptic curves are just two copies of the same group, for example $\Z/n\oplus \Z/n$.

\subsubsection{Rings and Fields}
We can also get two operations on the same set:

\begin{definition} [Ring]
    A ring is a set $R$ with two operations, usually called addition $+$ and multiplication $\times$, satisfying several properties:

    $R, +$ is an Abelian group.

    $R, \times$ is associative, and together \emph{distribute}:
    $$a\times (b+c)=a\times b+a\times c,$$
    $$(a+b)\times c=a\times c+b\times c$$
    for all $a,b,c\in R.$ If the multiplication is commutative then the ring is called a \emph{commutative ring}. Often rings will also have an identity element for multiplication included as part of the definition.
\end{definition}

In other words, rings have a notion of addition, subtraction and multiplication but not necessarily division. When there is a notion of division, they are called division rings and when the multiplication is commutative, we get a \emph{field}:

\begin{definition} [Field]
    A field $F$ is a commutative ring with the added property that every non-zero element has a multiplicative inverse.
\end{definition}

In other words, $F,+$ is an Abelian group and $F\setminus\{0\}, \times$ is an Abelian group, with distributivity holding them together.

\begin{definition}
    [Characteristic of a Field] The \emph{characteristic} of a field is the smallest $n$ for which scaling $1$ by $n$ equals zero, that is $n(1)=0$. The characteristic is defined to be zero if no $n$ exists.
\begin{proposition}
    If the characteristic is non-zero, then it is prime. 
\end{proposition}
    
    \begin{proof}
        Suppose for the sake of contradiction that we have a field $F$ of characteristic $ab$ where $a,b\in \N_{\geq2}$. We know that every element in $F$ has a multiplicative inverse. Suppose $a^{-1}$ is the inverse of $a$. Therefore, $1= aa^{-1}$ from which multiplying by $b$ implies that $b= baa^{-1}=0$. This is impossible because $0=b<ab$ which is assumed to be the smallest value that can scale $1$ to become $0$. This contradiction necessitates that the characteristic cannot be composite if it is nonzero.
    \end{proof}

\end{definition}

\begin{eg}
    $\Z$ with addition and multiplication is a commutative ring and so is $\Z/n\Z$ with addition and multiplication modulo $n$.
\end{eg}

\begin{eg} $\Q$ the set of rational numbers together with addition and multiplication is a field of characteristic zero.
\end{eg}

\begin{eg}  $\C$ the set of complex  numbers together with addition and multiplication is a field of characteristic zero.
\end{eg} 

\begin{eg} $\Z/p\Z$ the set of integers modulo a prime $p$ together with addition and multiplication modulo $p$ is a field. In this context, we will rewrite this as $\mathbb{F}_p$ to emphasise it is a finite field of size $p$. It is a field of characteristic $p$.
\end{eg} 
\subsubsection{Vector Spaces and Modules}

The torsion subgroup of an Abelian group seems to come from more than just an Abelian group since we defined a notion of scalar multiplication.
\newline\newline
This subsection explores the interaction between scalar multiplication and addition:

\begin{definition}[R-module] Let $R$ be a ring with a multiplicative identity of $1$. A (left) $R$-module $V$ is an Abelian group $V,+$ together with a pairing
$$R\times V\to V$$
denoted
$$(\lambda, v)\mapsto \lambda v$$
called scaling by $R$ that satisfies several natural compatibility properties: 
\begin{itemize}
    \item $\lambda\cdot(v_1+v_2)=\lambda\cdot v_1+\lambda\cdot v_2$
    \item $(r_1+r_2)\cdot \lambda=r_1\lambda+r_2\lambda$
    \item $(r_1r_2)\cdot v=r_1\cdot(r_2\cdot v)$
    \item $1\cdot v= v$
\end{itemize}
\end{definition}

\begin{definition} [$F$-Vector Space]
    A vector space over a field $F$ is an $F$-module.
    
\end{definition}

\begin{eg} Every abelian group is a $\Z$-module where  scaling by $\Z$ is defined as above  as repeated addition or subtraction.
\end{eg}
\begin{eg} Every commutative ring $R$ is an $R$-module, in particular,  $\Z/n\Z$ is a $\Z/n$-module and $\F_p$ is a $\F_p$-vector space.
\end{eg}

\begin{defn}
    Given two $R$-modules $U$ and $V$, define an $R$-module $U\oplus V$ by taking the Abelian group $U\oplus V$ and define scaling by
    $$\lambda(u,v)=(\lambda u,\lambda v)$$
    for all $\lambda\in R,u\in U,v\in V$.
\end{defn}
\begin{eg} $\F_p\oplus \F_p$ is a $\F_p$ vector space. 
\end{eg}
\begin{eg} $\Z/n\Z\oplus\Z/n\Z$ is a $\Z/n\Z$-module
\end{eg}

We start by looking at the $P$-torsion points as an Abelian group but gradually change our perspective to first see a $\Z$-module and then a $\Z/p$-module, which is the same as an $\Z/p$-vector space.

\subsection{Morphisms in Abstract Algebra}

So far, we've introduced several key \textit{objects} in abstract algebra: \textbf{groups}, \textbf{rings}, \textbf{fields}, and \textbf{modules}. Each of these structures is defined by a set equipped with operations satisfying certain axioms. However, we’ve yet to discuss the concept of \textbf{morphisms} — the functions between these structures that preserve their defining properties.
\newline\newline
Understanding morphisms is important because they allow us to compare algebraic objects meaningfully. They also provide powerful shortcuts: instead of verifying all properties directly, we can often reason about structures via the morphisms between them.
\newline\newline
Before diving into the details, getting the big idea can save a lot of effort, because to write everything in every detail can be difficult to get right completely, and it can be overwhelming sometimes when you're just starting out.
\newline\newline
Morphisms that we discuss are functions between objects of the same type that preserve their defining properties. Even if it appears they are defined to not preserve some properties, they usually are able to be shown as a consequence of the preserved ones that are included.

\subsubsection{Group Homomorphisms}

\begin{definition}[Group Homomorphisms]

Given two groups \( G \) and \( H \), a \emph{group homomorphism} is a function
\[
\varphi : G \to H
\]
that satisfies
\[
\varphi(ab) = \varphi(a)\varphi(b) \quad \text{for all } a, b \in G.
\]
\end{definition}
\begin{definition}[Group Isomorphism]
    A \emph{group isomorphism} is a bijective homomorphism whose inverse is also a homomorphism. If such a map exists, we say that \( G \) and \( H \) are \emph{isomorphic} as groups — essentially, they are the same group structure up to renaming of elements. \( G \) and \( H \) being \emph{isomorphic} is denoted by  \( G \) $\cong$ \( H \).
\end{definition}
 
 \begin{definition}[Automorphism]
     An \emph{automorphism} is an isomorphism of a group to itself.
 \end{definition}

You may wonder if $\varphi$ should send the identity of $G$ to the identity of $H$ but this is a consequence of preserving the operation.

\begin{eg}Group homomorphisms $f$ between Abelian groups might be written
$$f(a+b)=f(a)+f(b),$$
which is just an additive function. We know that $f(0)=0$ and $f(-a)=-a$ are consequences of additivity.
    
\end{eg}




\subsubsection{Homomorphisms of Rings and Modules}

\begin{itemize}
  \item For \textbf{rings}, a ring homomorphism \( \varphi : R \to S \) must preserve both addition and multiplication:
  \[
  \varphi(a + b) = \varphi(a) + \varphi(b), \quad \varphi(ab) = \varphi(a)\varphi(b),
  \]
  and often (but not always), \( \varphi(1_R) = 1_S \).

\begin{eg}
  An illustrative example is the ring homomorphism $\phi:\mathbb{Z}\to \mathbb{Z}/n\mathbb{Z}$. This holds as addition and multiplication are preserved in modular arithmetic. 
\end{eg}

\begin{eg}
    The most familiar example of a field automorphism on $\C$ might be that of complex conjugation:

    $$z=x+yi\mapsto \bar{z}=x-yi$$
\end{eg}  
\begin{definition} Given a field $F$,
    the set of all field automorphisms is denoted $\Aut(F)$ and is itself a group under function composition.
\end{definition}
  \item For \textbf{modules}, a module homomorphism between two \( R \)-modules \( M \) and \( N \) is a function \( f : M \to N \) that preserves addition and scalar multiplication:
  \[
  f(x + y) = f(x) + f(y), \quad f(rx) = rf(x) \quad \text{for all } x, y \in M,\, r \in R.
  \]

  \begin{definition} This is also an example but is used later too: an $F$-vector space $V$ is \emph{two dimensional} if
$$V\cong F\oplus F.$$ 
  \end{definition}


    
\end{itemize}

%% file: elliptic.tex
\section{Elliptic Curves}

The goal for this section is to define the Tate module from the elliptic curve, and show that it is two dimensional, while gaining a deeper understanding of elliptic curves along the way.
\newline\newline
Keywords: Short Weierstrass Equation, Conductor, $j$-invariant, Singularities, Cusp, Node, Tate Module, Two Dimensional, Inverse Limit, $\ell$-adic. 
\newline\newline
The study of elliptic curves is closer to core topics in graduate school such as algebraic geometry but not always singled out on its own. It has now made its way as a chapter in the classic book \cite{HW} 6th Edition after FLT was proved.

\subsection{Defining Equations}

While the definition of an elliptic curve given in the outline is sufficient for much of what we will discuss, there are further variations in definitions that we explore both for better understanding and so that various definitions can be applied in different proofs to simplify them.
\newline\newline
Here is one definition of an elliptic curve; as solutions to more general equations in $x$ and $y$:
\begin{definition}[Weierstrass Equation] Given a field $F$ and constants $a_1,a_2,a_3,a_4,a_6\in F$, the Weierstrass equation (of an elliptic curve) is given by
$$y^2+a_1xy+a_3y=x^3+a_2x^2+a_4x+a_6$$

\end{definition}
However for simplicity consider the following transformation into the \emph{Short Weierstrass Equation}:

\begin{prop}[Short Weierstrass Equation]The transformation $x\mapsto x$, and  $y\mapsto y-a_1x/2-a_3/2$ takes the Weierstrass equation 

$$y^2+a_1xy+a_3y=x^3+a_2x^2+a_4x+a_6$$
to
    $$y^2=x^3+(a_1^2/4+a_2)x^2+(a_1a_3/2+a_4)x+a_3^2/4+a_6.$$
\end{prop}
\begin{proof}
   
$$(y+a_3/2)^2+ a_1x(y+a_3/2)=x^3+a_2x^2+(a_4+a_1a_3/2)x+a_3^2/4+a_6$$
Replace $y+a_3/2$ with $y$:
$$y^2+a_1xy=x^3+a_2x^2+(a_4+a_1a_3/2)x+a_3^2/4+a_6$$
$$y(y+a_1x)=x^3+a_2x^2+(a_4+a_1a_3/2)x+a_3^2/4+a_6$$
$$((y+a_1x/2)-a_1x/2)((y+a_1x/2)+a_1x/2)=x^3+a_2x^2+(a_4+a_1a_3/2)x+a_3^2/4+a_6$$
Replace $y+a_1x/2$ with $y$:
$$(y-a_1x/2)(y+a_1x/2)=x^3+a_2x^2+(a_4+a_1a_3/2)x+a_3^2/4+a_6$$
$$y^2-(a_1x/2)^2=x^3+a_2x^2+(a_4+a_1a_3/2)x+a_3^2/4+a_6$$
$$y^2=x^3+(a_1^2/4+a_2)x^2+(a_1a_3/2+a_4)x+a_3^2/4+a_6.$$
 
\end{proof}

We see that $F$ must not be characteristic $2$ since we are dividing by $2$.
The above equation can then be put in the following form, (by replacing $y$ with $y/2$ and multiplying both sides by 4):
 $$y^2=4x^3+b_2x^2+2b_4x+b_6$$ 
 where $b_2=a_1^2+4a_2,b_4=a_1a_3+2a_4, b_6=a_3^2+4a_6$.
\newline\newline
For all the above calculations we see that $F$ must not be characteristic $2$ since we are dividing by $2$.
\newline\newline
The $x^2$ term can also be removed: 

\begin{theorem}If the characteristic of $F$ is not $2$ or $3$, then using a change of variables one can transform the equation into the form
$$y^2=4x^3-c_4x-c_6$$
where $c_4=\frac{b_2^2-24b_4}{12}$ and $c_6=\frac{-b_2^3+36b_2b_4-216b_6}{216}$. 
\end{theorem}

\begin{proof}

    In the following elliptic curve consider replacing $x$ with $x-\frac{b_2}{12}$:
    
    $$y^2=4x^3+b_2x^2+2b_4x+b_6$$
    $$y^2=4(x-\frac{b_2}{12})^3+b_2(x-\frac{b_2}{12})^2+2b_4(x-\frac{b_2}{12})+b_6$$
    $$y^2=4(x^3-\frac{b_2x^2}{4}+\frac{b_2^2x}{48}-\frac{b_2^3}{1728})+b_2(x^2-\frac{b_2x}{6}+\frac{b_2^2}{144})+2b_4(x-\frac{b_2}{12})+b_6$$
    $$y^2=4x^3-b_2x^2+\frac{b_2^2x}{12}-\frac{b_2^3}{432}+b_2x^2-\frac{b_2^2x}{6}+\frac{b_2^3}{144}+2b_4x-\frac{b_2b_4}{6}+b_6$$
    $$y^2=4x^3-(\frac{b_2^2}{12}-2b_4)x-(-\frac{b_2^3}{216}+\frac{b_2b_4}{6}-b_6)$$
    $$y^2=4x^3-c_4x-c_6$$
\end{proof}
This is a form more familiar to some people.

\begin{theorem}If the characteristic of F is not 2 or 3, then
using a change of variables one can transform the above equation into the form: $$y^2=x^3+Ax+B$$ where $A=-4c_4$ and $B=-16c_6$.   
\end{theorem}
\begin{proof}
    In the following equation consider replacing $x$ with $\frac{x}{4}$ and $y$ with $\frac{y}{4}$:
    $$y^2=4x^3-c_4x-c_6$$
    $$\frac{y^2}{16}=\frac{x^3}{16}-\frac{c_4x}{4}-c_6$$
    $$y^2=x^3-4c_4x-16c_6$$
    $$y^2=x^3+Ax+B$$
\end{proof}

Recall the discriminant defined in Definition \ref{discdef}.

 This can be represented purely in terms of the coefficients:

\begin{lemma}
   The discriminant of $y^2=x^3+Ax+B$ is $$\Delta=-16(4A^3+27B^2).$$
\end{lemma}
 

If we have an elliptic curve over $\mathbb{Q}$ (here $F=\Q$), we can always put it in the form of the Weierstrass equation with integer coefficients. Such an equation would have an integer discriminant. If we attempt to take the coefficients of this equation mod $p$ for some prime $p$, then a possible issue that may arise it that many of the coefficients might be divisible by $p$ and become $0$. Thus we consider  the equation for an elliptic curve that has minimal power of $p$ dividing $\Delta$, and call this the \emph{minimal Weierstrass equation}.
\newline\newline
In general, there is an issue with different equations defining what we might consider to be the same elliptic curve. As such the following definition is supposed to distinguish between elliptic curves that are the same or not.






\begin{definition}
    The $j$-invariant of an elliptic curve is defined by
    $$j=\frac{1728c_4^3}{\Delta}$$
    where $c_4$ is as defined above and $\Delta$ is the discriminant.
\end{definition}

We do not explore nor define the notion of isomorphism between two elliptic curves here but it is said that two elliptic curves are isomorphic if and only if they have the same $j$-invariant. 






\subsection{Elliptic Curve as an Abelian group}

We prove that the addition of points on an elliptic defined in the outline forms an Abelian group:

\begin{thm} Let $E$ be an elliptic curve and $O\in E$ the point at infinity. For $R\in E$, define $-R\in E$ to be $O$ if $R=O$ and otherwise a point with the same $x$ value but negative $y$ value to $R$. The operation $+$, the addition of points, is defined for $P,Q,R\in E$ by
$$P+Q=R$$
if and only $P, Q, -R$ are collinear with some clarifications: in the case that $P=Q=O$, $R=O$, a line through $O$ and another point $P$ is a vertical line through $P$, if $P=Q\neq O$ then the line is tangent at $P$. This addition is well-defined and makes $E,+$ an Abelian group with zero element $O$ and $-R$ the negative of $R$.

\end{thm}

\begin{proof}
    First, we must show that we can obtain $-R$ with rational coordinates given $P$ and $Q$ with rational coordinates. As the elliptic curve is a rational cubic with the equation $$y^2=x^3+Ax+B$$ when we intersect with the straight line $y=mx+c$, the result is the cubic equation $$x^3+Ax+B=(mx+c)^2$$ which becomes $$x^3 - m^2x^2 + (A - 2mc)x + (B - c^2)=0.$$ If rational points $P$ and $Q$ satisfy the above equation, the line $mx+c$ passing through them must have rational slope and intercepts, i.e. $m$ and $c$ are rational. So, as $A,B,m,c\in\mathbb{Q}$, i.e. the coefficients of the cubic are rational, by Vieta's formulae it follows that the sum of the three roots is rational. The $x$-coordinates of $P$ and $Q$ are already rational, so it follows that the third root is rational. Using $y=mx+b$, we get that the $y$-value is also rational.
    \newline\newline
    Satisfying Group Axioms:
    \begin{itemize}
        \item The addition of points on the curve satisfies commutativity $P+Q=Q+P$ as the line through $P$ and $Q$ is the same as the line through $Q$ and $P$. 
        \item We now show the point at infinity is the identity. The line through $P$ and $\mathcal{O}$ is the vertical line passing through $P$. As the elliptic curve is symmetric about the $x$-axis, then our $-R$ is the reflection of $P$ about the $x$-axis. Reflecting returns the point $P$ as desired. 
        \item We must also show that every point has an inverse. Given a point $P(x, y)$ on the elliptic curve, its inverse is $-P(x, -y)$ which lies on the same vertical line as $P$. As elliptic curves are symmetric about the $x$-axis, the reflection of a point about the $x$-axis always exists on the curve. Thus, $-P$ is a valid point on the curve, and satisfies $P+(-P)=\mathcal{O}$, completing the requirement for inverses in the group law. 
        \item Finally, the operation must satisfy the associativity axiom $(P+Q)+R=P+(Q+R)$. This is a lengthy proof which can be found in the appendix for those interested. 
 \end{itemize}
\end{proof}
We can also write a formula for the $x$-value of this addition:

\begin{prop}
The $x$ value of $P+Q=R$ on $y^2 = x^3 + Ax + B$ is given by

\[
x_R = \left( \frac{y_Q - y_P}{x_Q - x_P} \right)^2 - x_P - x_Q
\]

\end{prop}

\begin{proof} 

Let the equation of line through $P,Q,-R$ be $y=mx+b$. It suffices to solve for $x$ as $x_R=x_{-R}$.

\[
y^2 = m^2 x^2 + 2mbx + b^2 = x^3 + Ax + B
\]
\[
m = \frac{y_Q - y_P}{x_Q - x_P}, \quad m^2 = \left( \frac{y_Q - y_P}{x_Q - x_P} \right)^2.
\]

\[
m^2 = x_P + x_Q + x_R
\]
$$x_R=m^2-x_P-x_Q.$$
    
\end{proof}

Recall Definition \ref{torsionpt} of $m$-torsion points $E[m]$. They can also be described as the $m$-torsion subgroup of $E$.
\newline\newline
The main values of $m$ we will be interested in are powers of a prime, in particular the prime itself in the power of the equation in the statement of FLT.
\begin{prop}\label{divpts} Let $p$ be a prime, then
    $E[p]$ with scalar multiplication modulo $p$ is a $\F_p$-vector space. In this case, Call $E[p]$ the $p$-division points.
\end{prop}
    

Fixing a prime $\ell$ we have homomorphisms

$$\varphi_n:E[\ell^{n+1}]\to E[\ell^n]$$
defined by
$$P\mapsto \ell P.$$

Noticing that the $E[\ell^n]$ are related by these homomorphisms, it is natural to want to somehow combine the information from each of them into a single object. The way that this is done is via the following idea:
\newline\newline
Define

$$\lim_{\longleftarrow} E[\ell^n]=\{(e_n)_{n=1}^\infty: e_n=\varphi_n^m(e_{n+m}),\,\, \forall m,n\in\N \},$$
where $\varphi^m$ means applying $\varphi$ `$m$ times'.
\newline\newline
We expect to produce an Abelian group because our inputs are Abelian groups, but we could have produced something else if the inputs were other objects.
\newline\newline
If we try to consider the torsion points as modules, we see that $E[\ell^n]$ is an $\Z/\ell^n\Z$-module but not an $\Z/\ell^{n-1}\Z$ module and so they are not same type of modules, but there is a way to combine these too, using a similar idea of construction

\begin{definition}[$\ell$-adic Integers] \label{l-adic Z} Define
    $$\lim_{\longleftarrow} (\Z/\ell^n\Z)=\{(a_n)\,:\, a_n\in\Z/\ell^n\Z, \,\, a_n=a_{n+m}\mod \ell^n \quad\forall m\in\N\},$$
which is also known as the $\ell$-adic integers, written $\Z_{\ell}$, that is,
$$\Z_\ell=\lim_{\longleftarrow} (\Z/\ell^n\Z).$$
\end{definition}

We now combine the above into one nice definition/theorem:

\begin{definition} [$\ell$-adic Tate Module]
    For an elliptic curve $E$ over $\mathbb{Q}$ and a prime $\ell$, the Tate module of $E$ at $\ell$ is defined to be
    $$\tate(E)=\lim_{\longleftarrow} E[\ell^n]=\{(P_n)_{n\geq1}\mid P_n\in E[\ell^n],\quad \varphi_n(P_{n+1})=P_n\quad
    \forall n\geq1\}.$$

\end{definition}

\begin{prop}$\tate(E)$ is a $\Z_\ell$-module.  
\end{prop}

\begin{proof}
    Let \( a = (a_n) \in \mathbb{Z}_\ell  \). We define scalar multiplication on the Tate module $\tate(E)$ by:
\[
a \cdot (P_n) := (a_n \cdot P_n),
\]
where \( a_n \cdot P_n \) is just scalar multiplication in the \( \mathbb{Z}/\ell^n\mathbb{Z} \)-module \( E[\ell^n] \).
\newline\newline
We must check that this new sequence \( (a_n \cdot P_n) \) still satisfies the compatibility condition:
\[
\varphi_n(a_{n+1} \cdot P_{n+1}) = a_n \cdot \varphi_n(P_{n+1}) = a_n \cdot P_n,
\]
because:
\begin{itemize}
    \item \( \varphi_n \) is a group homomorphism, so it commutes with scalar multiplication. 
    
    \item  \( a = (a_n) \in \mathbb{Z}_\ell \), so \( a_{n+1} \equiv a_n \pmod{\ell^n} \).
\end{itemize}

Therefore, the result is still a valid element of $\tate(E)$, and hence the Tate module is a $\Z_\ell$-module.

\end{proof}

\subsection{Uniformisation Theorem}

As we transition to the next discussion on modular forms, we introduce a theorem that we cannot prove here called the Uniformisation Theorem.

This allows us to reduce the study of elliptic curves to that of lattices:

\begin{definition}[Lattice $\mathcal{L}$]
   Consider the complex plane and $\omega_1,\omega_2\in\C\setminus\{0\}$ such that they are not multiples of each other. Then the lattice $\mathcal{L}$, is defined by
   $$\mathcal{L}=\mathcal{L}(\omega_1,\omega_2)=\Z\omega_1\oplus\Z\omega_2=\{a\omega_1+b\omega_2\,:\,a,b\in\Z\}$$
\end{definition}
\begin{lem}
    $\mathcal{L}(\omega_1,\omega_2)$ is an Abelian group under addition.
\end{lem}
\begin{proof}
    Consider two elements $z_1=a_1\omega_1+b_1\omega_2$ and $z_2=a_2\omega_1+b_2\omega_2$ which are in $\mathcal{L}$. Therefore, $$z_1+z_2=(a_1+a_2)\omega_1+(b_1+b_2)\omega_2 \in \mathcal{L}.$$ The identity element is $0$, or more visually $0\omega_1+0\omega_2$. For any element $a\omega_1+b\omega_2$, the inverse is $-a\omega-b\omega_2$ which is in $\mathcal{L}$. Since addition of complex numbers is associative and commutative, all the group axioms hold. Thus, $\mathcal{L}(\omega_1,\omega_2)$ is an abelian group under addition.

\end{proof}

If we view the complex numbers $\C$ as an Abelian group under addition, then any lattice $\mathcal{L}(\omega_1,\omega_2)$ is a subgroup.

We can construct the quotient group 
$$\C/\mathcal{L}(\omega_1,\omega_2)$$
as the set where complex numbers are identified via setting elements of $\mathcal{L}$ to zero. That is for $z_1,z_2\in\C$
$$z_1\cong z_2\mod \mathcal{L}\iff z_1-z_2\in\mathcal{L}.$$

\begin{prop}
    Two lattices give isomorphic quotients if and only if one scales to the other. That is 

$$\C/\mathcal{L}\cong\C/\mathcal{L}'\iff \lambda\mathcal{L}=\mathcal{L}' \qquad\exists\lambda\in\C\setminus0$$
\end{prop}
$$\C/\mathcal{L}'\cong\C/\lambda\mathcal{L}\cong \lambda\C/\lambda\mathcal{L}\cong \C/\mathcal{L}$$

The Uniformisation theorem says that elliptic curves over $\C$ can be associated with the plane mod a lattice.

$$E_\C\cong \C/\mathcal{L}$$

\begin{itemize}
    \item What is the nature of this isomorphism? 
    \item What types of objects are these between?
    \item We see that the right hand side is very naturally an Abelian group and so is $E$, so can we hope that this is an isomorphism of Abelian groups?
\end{itemize}

We first note that if $\lambda\in\C\setminus 0$, then

As a consequence of this we prove here that


\begin{thm}\label{EC Iso}
    The set of $n$-torsion points on the Elliptic Curve is isomorphic to the direct sum of two finite cyclic groups, i.e. 
    $$E[n] \cong (\mathbb{Z}/n\mathbb{Z}) \oplus (\mathbb{Z}/n\mathbb{Z}).$$ 
\end{thm}

\begin{proof}
  We know that there exists $\omega_1,\omega_2\in\C$ such that $E(\C)$ is isomorphic to $\C/\mathcal{L}$ with $\mathcal{L}=\Z\omega_1+\Z\omega_2$. Any element in $\C/\mathcal{L}$ can be written uniquely as $r_1\omega_1+r_2\omega_2$ where $r_1,r_2\in[0,1)$.  Call this $(r_1,r_2)$. For an $n$-torsion point,  $n(r_1,r_2)=(0,0)$, then $(\{nr_1\},\{nr_2\})=(0,0)$, where $\{x\}$ denotes the fractional part of $x$. But this implies that $nr_i$ is an integer, and hence we may  write: $$r_1\omega_1+r_2\omega_2=\frac{a}{n}\omega_1+\frac{b}{n}\omega_2$$ with $a,b\in \{0,1, \cdots, n-1\}$. Define a function
  $$f:E[n]\to (\mathbb{Z}/n\mathbb{Z}) \oplus (\mathbb{Z}/n\mathbb{Z})$$ by $f(r_1, r_2)=(a,b)$. Then, $f$ is surjective by construction. We want to show that $f$ is injective. Suppose $$f(r_1,r_2)=f(r_1', r_2') \implies(a,b)=(a',b').$$ This means the points $$\frac{a}{n}\omega_1+\frac{b}{n}\omega_2 \equiv \frac{a'}{n}\omega_1+\frac{b'}{n}\omega_2 \quad (\text{mod}\ \mathcal{L})$$
  Then: $$\frac{a-a'}{n}\omega_1+\frac{b-b'}{n}\omega_2\in \mathcal{L}.$$ This implies: $$(a-a')\omega_1+(b-b')\omega_2\in n\mathcal{L}.$$ But $\omega_1, \omega_2$ are linearly independent over $\mathbb{R}$, so the only way this combination lies in $n\mathcal{L}$ is if $a=a'$ and $b=b'$. Thus, $f$ is injective. \newline \newline Therefore, $f$ is a bijection and we have isomorphism of the abelian groups: $$ E[n] \cong (\mathbb{Z}/n\mathbb{Z}) \oplus (\mathbb{Z}/n\mathbb{Z})$$. 
\end{proof}
\begin{cor}
   If $l$ is prime, then $E[l]$ is a two dimensional $\Z/l\Z$-vector space. 
\end{cor}

\begin{proof} We have already explained that this is a vector space known as the $l$-division points (Proposition \ref{divpts}). We see from the theorem above that it is two-dimensional with the same two generators.

\end{proof}

Finally, we get that the Tate module is also two-dimensional:

\begin{thm}
    $$\tate(E)\cong \Z_\ell\oplus\Z_\ell$$
\end{thm}
\begin{proof}
We see that from Theorem \ref{EC Iso}, the multiplication by $\ell$ maps are surjective. One way to see this, is to see the image is 
$$\varphi_n(E[\ell^{n+1}])\cong\ell((\Z/\ell^{n+1}\Z)\oplus(\Z/\ell^{n+1}\Z))\cong ((\Z/\ell^n\Z)\oplus(\Z/\ell^n\Z)).$$

Now pick generators for the direct sum so that $\varphi_n$ preserves the direct sum.
\newline\newline
Define $(P_n)_{n\geq1}$ and $(Q_n)_{n\geq1}$ as follows:

Let $P_1$ and $Q_1$ generate $E[\ell]$. For $P_{n-1}$ and $Q_{n-1}$ already defined, let $P_n,Q_n\in E[\ell^n]$ be any elements such that $\varphi_n(P_n)=P_{n-1}$ and $\varphi_n(Q_n)=Q_{n-1}$, which exist by surjectivity of $\varphi_n$. Each pair $P_n,Q_n$ must also generate $E[\ell^n]$ since it suffices to show that they are not multiples of each other. They are not multiples of each other because their image is not cyclic.
\newline\newline
Now that the generators align with $\varphi_n$, we get the desired result. As a bonus We prove a more general result below after introducing some more general notions.
\end{proof}

The constructions of the Tate module and $\ell$-adic integers follow a general construction known as inverse limits.

\begin{definition}[Inverse Limit]
Let $I$ be an infinite set of with a notion of $\leq$ such that for every $i,j\in I$, there exists $k\in I$ with $k\geq i,j$. Consider the family of groups $A_i$ and a family of homomorphisms, $f_{ij}$, where for each $i\leq j$  $f_{ij}$ is a homomorphism from $A_j$ to $A_i$, $f_{ii}$ is the identity, and $f_{ij}=f_{ik}\circ f_{kj}$. The inverse limit is defined as follows:
$$\varprojlim_{i \in I} A_i = \left\{ (a_i)_{i\in I} \;\middle|\; a_i\in A_i,\quad a_i = f_{ij}(a_j),\quad\forall j>i \right\}$$
The homomorphisms $f_{ij}$ and the family of groups together are called an \emph{inverse system}.
\end{definition}

We show that when each component of the inverse limit decomposes as a direct sum, and the connecting maps align with this decomposition, then the inverse limit preserves the decomposition as well.



\begin{prop}

  Suppose we have two inverse systems of abelian groups:
  \begin{itemize}
     \item $\{A_n\}_{n\in \N}$ with transition maps $\varphi^A_n:A_{n+1}\to A_n,$
    \item $\{B_n\}_{n\in \N}$ with transition maps $\varphi^B_n:B_{n+1}\to B_n.$
  \end{itemize}
  Then the product system $\{ A_n \oplus B_n\}$ is an inverse system with transition maps $$
\varphi_n : A_{n+1} \oplus B_{n+1} \longrightarrow A_n \oplus B_n$$ defined component-wise by: $$\varphi_n(a_{n+1}, b_{n+1}) = \left( \varphi_n^A(a_{n+1}), \varphi_n^B(b_{n+1}) \right).$$
Thus we have that

$$\varprojlim (A_n \oplus B_n)=\left( \varprojlim A_n \right) \oplus \left( \varprojlim B_n \right)$$

\end{prop}
\begin{proof}
    
\begin{align*}
\varprojlim (A_n \oplus B_n) 
&= \left\{ (a_n, b_n) \in \prod_n (A_n \oplus B_n) \;\middle|\; \varphi_n((a_{n+1}, b_{n+1})) = (a_n, b_n) \text{ for all } n \right\} \\
&= \left\{ (a_n, b_n) \in \prod_n (A_n \times B_n )\;\middle|\; \varphi_n^A(a_{n+1}) = a_n \text{ and } \varphi_n^B(b_{n+1}) = b_n \text{ for all } n \right\} \\
&= \left( \varprojlim A_n \right) \oplus \left( \varprojlim B_n \right).
\end{align*}

\end{proof}



%% file: prelimModular.tex
\section{Preliminaries II}

\subsection{Linear Algebra}

We avoid much of linear algebra by introducing two by two matrices explicity. In general, linear algebra is a huge topic that many countries start in high school or first year university.

\subsubsection{Concrete Linear Algebra}

Explicit calculations with matrices might be considered concrete linear algebra and these computations are easily taught to and carried out by high school students.
\newline\newline
We focus here on $2\times 2$ matrices:


\begin{definition}[$2\times2$ Matrices]
    A $2\times2$ ('two by two') matrix over a commutative ring $R$ with $1\in R$ is an array of elements $a,b,c,d\in R$ in 2 rows and columns, written: $$\Mabcd$$
    The set of all such matrices is denoted $\mat_2(R)$. There is also a matrix addition and multiplication that make $\mat_2(R)$ into a ring:
    \begin{itemize}
        \item Addition of two 2 by 2 matrices is defined as summing the values in the corresponding entries of the 2 matrices, evaluating to a matrix with entries in $R$. An example of this for $\mat_2(\Z)$ is as follows: $$\begin{pmatrix} 1&4\\7&8 \end{pmatrix}+\begin{pmatrix} -1&2\\0&1 \end{pmatrix}=\begin{pmatrix} {1-1}&{4+2}\\{7+0}&{8+1}\end{pmatrix}=\begin{pmatrix} 0&6\\7&9\end{pmatrix}$$
        \item Multiplication of two 2 by 2 matrices, say $X$ and $Y$, results in a matrix with entries in $R$. Each entry of the resulting matrix is defined as the dot product of the row in $X$ containing that entry and the column in $Y$ containing that entry, strictly speaking for:$$X=\begin{pmatrix} {a_1}&{b_1}\\{c_1}&{d_1} \end{pmatrix},Y=\begin{pmatrix} {a_2}&{b_2}\\{c_2}&{d_2} \end{pmatrix}$$ $$X Y=\begin{pmatrix} {a_1}&{b_1}\\{c_1}&{d_1} \end{pmatrix}\begin{pmatrix} {a_2}&{b_2}\\{c_2}&{d_2} \end{pmatrix}=\begin{pmatrix} {a_1a_2+b_1c_2}&{a_1b_2+b_1d_2}\\{c_1a_2+d_1c_2}&{c_1b_2+d_1d_2} \end{pmatrix}$$
   \end{itemize}
   Matrix multiplication is usually not commutative.
   The identity matrix for matrix multiplication is: 
   $$\begin{pmatrix} {1}&{0}\\{0}&{1} \end{pmatrix}$$
   and can explicitly be checked, in the above matrix multiplication formula, to satisfy $$X\cdot\begin{pmatrix} {1}&{0}\\{0}&{1} \end{pmatrix}=X=\begin{pmatrix} {1}&{0}\\{0}&{1} \end{pmatrix}\cdot X$$
   A matrix $X$, with entries in the ring $R$, is invertible if there exists another matrix, $Y$, with entries in $R$ such that $XY=\begin{pmatrix} {1}&{0}\\{0}&{1} \end{pmatrix}=YX$.
\end{definition}

\begin{definition}[Trace]

The trace of a 2 by 2 Matrix $X$, denoted as $\Tr(X)$, is the sum of the elements in the diagonal from the top left to bottom right, specifically the trace of X is:
$$\Tr(X)= \Tr\Mabcd=a+d.$$

\end{definition}
\begin{thm}
    For all $X,Y\in\mat_2$,
    $$\Tr(XY)=\Tr(YX).$$
    In particular, if $X^{-1}$ exists, then
    $$\Tr(XYX^{-1})=\Tr(Y).$$
\end{thm}
\begin{proof}
    
$$XY=\begin{pmatrix} {a_1a_2+b_1c_2}&{a_1b_2+b_1d_2}\\{c_1a_2+d_1c_2}&{c_1b_2+d_1d_2} \end{pmatrix}, YX=\begin{pmatrix} {a_2a_1+b_2c_1}&{a_2b_1+b_2d_1}\\{c_2a_1+d_2c_1}&{c_2b_1+d_2d_1} \end{pmatrix}$$

$\Tr(XY)=a_1a_2+b_1c_2+c_1b_2+d_1d_2=a_2a_1+b_2c_1+c_2b_1+d_2d_1=\Tr(YX)$
Thus,
$$ \Tr((XY)X^{-1})=\Tr(X^{-1}(XY))=\Tr(Y).$$
\end{proof}
\begin{definition}[Determinant]
    
The \emph{determinant}, denoted by $\det$, is given by
$$\det\Mabcd=ad-bc.$$

\end{definition}

\begin{thm} 
   For all $X,Y\in\mat_2$,
   $$\det(XY)=\det(X)\det(Y).$$
\end{thm}
\begin{proof}
    
From the above definition of X and Y we have:
$$XY=\begin{pmatrix} {a_1a_2+b_1c_2}&{a_1b_2+b_1d_2}\\{c_1a_2+d_1c_2}&{c_1b_2+d_1d_2} \end{pmatrix}$$
Thus, now we explicitly show $\det(XY)=\det(X)\det(Y)$:
$$\det(XY)=(a_1a_2+b_1c_2)(c_1b_2+d_1d_2)-(a_1b_2+b_1d_2)(c_1a_2+d_1c_2)$$
$$=a_1a_2b_2c_1+a_1a_2d_1d_2+b_1b_2c_1c_2+b_1c_2d_1d_2-a_1b_2c_1a_2-a_1b_2c_2d_1-b_1d_2c_1a_2-b_1d_2d_1c_2$$
$$=a_1a_2d_1d_2+b_1b_2c_1c_2-a_1b_2c_2d_1-b_1d_2c_1a_2$$
$$=(a_1d_1-b_1c_1)(a_2d_2-b_2c_2)=\det(X)\det(Y).$$
\end{proof}

\begin{definition}[General and Special Linear Groups]
   \leavevmode
    \begin{itemize}
        \item The General Linear Group of degree 2 is denoted by $\GL_2(R)$ and is the set of invertible $2 \times 2$ matrices with entries in the ring $R$. It is a group under matrix multiplication. 
        \item The Special Linear Group of degree 2 is denoted by $\SL_2(R)$ and is the set of $2 \times 2$ matrices with entries in $R$ and determinant 1. It is also a group under matrix multiplication.
    \end{itemize}
\end{definition}

Specific Examples of $\GL_2(R)$ and  $\SL_2(R)$:
    In this project, the Special/General Linear Groups of degree 2 will be considered over different rings, e.g. the integers as well as the complex numbers.
    \newline\newline
    For example, the Special Linear Group of 2 by 2 matrices over the integers, defined by $\SLZ$, is the set of all matrices $\Mabcd\in \mat_2(\mathbb{Z})$ such that the determinant $ad-bc=1$, paired with the operation of matrix multiplication.

\subsubsection{Abstract Linear Algebra}

Abstract linear algebra goes from matrices to linear functions on vector spaces. Matrices are essentially linear functions where a choice of 'coordinates' or a choice of 'basis' has been made. It is therefore of interest to know what properties of the matrices are intrinsic to the linear function, instead of depending of the choice of basis. 
\newline\newline
Those introduced in the previous subsubsection such as trace and determinant do not depend on the choice of basis.

The following is a classic problem in linear algebra that is key to establishing invariants that do not depend on a choice of basis.

\begin{definition}[Eigenvectors and Eigenvalues] Given a $F$-vector space $V$ and $T:V\to V$ an $F$-linear function, then for $v\in V\setminus\{0\}$ and $\lambda\in F$, the \emph{eigenvalue problem} amounts to solving the equation 
$$Tv=\lambda v$$
for $v$ and $\lambda$. $\lambda$ is called an \emph{eigenvalue} and $v$ is called a $\lambda$-eigenvector.

\end{definition}

This won't appear until the end of the subsequent Section where Hecke operators are introduced but will also resurface in the Sections after. It is tempting to note now that the trace gives the sum of the eigenvalues and the determinant gives the product, where for a two by two matrix, because these have at most 2 eigenvalues, this gives us the complete picture. In other words, for a two dimensional $V$, we get that the eigenvalues are the roots of the so called characteristic polynomial:
$$\lambda^2-\tr(T)\lambda+\det(T).$$


\subsection{Group Theory}

\subsubsection{Group Actions}
\begin{definition}[Group Action on Sets]
    A group action of a group $G$ on a set $X$ is a rule that lets the elements of $G$ act on the elements of $X$. More formally, an action is a function $$G \times X\to X$$ such that for all $g,h\in G$ and $x\in X$: $$(gh)\cdot x=g\cdot (h\cdot x)$$ and there exists an identity element $e\in G$ such that $$e\cdot x=x.$$
    
\end{definition}

We call a set with a group $G$ acting on it a $G$-set.

\begin{definition}[Orbit]
 Given $x\in X$, the orbit of $x$ under the action of $G$ is the set  to which $x$ can be mapped by the elements of the group $G$: $$Gx=\{g\cdot x \mid g\in G\}$$
\end{definition}
We can see then that orbits partition the set $X$. We say $G$ acts \emph{transitively} on $X$, if there is only one orbit (every element of $X$ can be reached from any other via the group action).
\newline\newline
In the following Section, $G$ being a group of invertible $2\time 2$ matrices will be of particular interest to us. 
\newline\newline
It is hard to talk about actions without talking about the \emph{stabiliser} subgroups and orbit-stabiliser theorem, but we will settle for just defining the stabiliser for now to be used later.
\newline\newline
\begin{definition}[Stabiliser]
    The stabiliser of an element $x\in X$ is a subgroup of $G$ consisting of all the elements of $G$ that fix $x$: $$G_x=\{g\in G \mid g\cdot x=x\}$$
\end{definition}

If we fix an element $g$ of $G$, we get an invertible function $X\to X$ defined by $x\mapsto g\cdot x.$ This way of thinking of elements of $g$ leads to a group homomorphism
from $G$ to the group of invertible functions of $X$.
\newline\newline
Group actions are a way of representing an abstract group as a concrete group of invertible functions. Later, we will try to represent the group as a group of linear functions, giving rise to what is known as a \emph{group representation}.

\subsubsection{Quotients}

Quotients are a construction that allows us to formalise the idea of setting things to the identity. People often use the phrase `modulo' to mimic the case when $n$ is set to zero when working in modular arithmetic modulo $n$.

\begin{definition}
    
If we have a set $X$ with a relation $\sim$, we can write $X/\sim$ to mean the set obtained from identifying every element with its related elements. That is, if $x\sim y$ in $X$, then $x$ and $y$ are represented by the same point in $X/\sim$ and by different points otherwise. For this to work properly the relation needs to be whats called an \emph{equivalence relation}.

\end{definition}

\begin{eg}
    
If a group $G$ acts on a set $X$, we write $X/G$ to denote the set of orbits under $G$, this amounts to setting the group action to the identity, forcing us to identify all the elements in one orbits as the same. In the language of the above we let $\sim$ be defined by
$$x\sim y\iff x\in Gy\iff y\in Gx$$

\end{eg}
\begin{eg}
If we have a vector space $V$ and a subspace $W$, we can set all the elements of $W$ to be zero, which amounts to constructing the quotient space $V/W$. Here
$$u\sim v\iff u-v\in W.$$
The construction is the same for Abelian groups or $R$-modules.
\end{eg}

A key example for us that we have encountered in the last section is the quotient of $\C$ by a lattice $\mathcal{L}$.

\subsection{Complex Analysis} 

Complex analysis will be a significant part of the Section that follows. Complex analysis is an undergraduate topic that is typically taught in the third year. It is the study of complex functions that are complex differentiable, called \emph{holomorphic functions} to emphasise their amazing properties. It is reviewed again in graduate school as a core topic in analysis. 
\newline\newline
The following content can get quite demanding in both its digestion and application. While an understanding of complex analysis is not assumed of the reader, having a more thorough understanding of complex analysis will definitely be a great help to understand the following Section.
\subsubsection{Keywords}
\begin{itemize}

\item {\bf Tier One Keywords:} limit point, uniform/absolute convergence, neighbourhood, open/closed, power series, bounded, zeros, discrete, Bolzano-Weierstrass theorem.

Some of these are known to high school students in Australia but here the emphasis is on complex numbers, whereas previously it was for real numbers.

\item {\bf Tier Two Keywords}: holomorphic, differentiable, analytic, poles, Laurent series, complex integral, simply connected, M\"{o}bius transformations.

\item {\bf Tier Three Keywords}: Cauchy's Theorem, Cauchy's integral formula, argument principle, Louiville's theorem, removable singularity, meromorphic.
\end{itemize}

\subsubsection{Convergence of Complex Functions}

\begin{definition} 
    
A bounded function  $f$ is one for which there exists $B\in\R$ such that $|f(z)|<B$ for all $z$ in the domain of $f$.

A bounded sequence is a sequence, bounded when considered as a function $\N\to\C$.
 \end{definition} 

\begin{theorem}[Bolzano-Weierstrass Theorem]
    A bounded infinite sequence of real numbers has a convergent subsequence.
\end{theorem}

\begin{cor}
A bounded infinite sequence of complex numbers has a convergent subsequence.
\end{cor}
\begin{proof}
    If the complex numbers are bounded, then the real and imaginary parts are bounded. Then you can use the Bolzano-Weierstrass Theorem to take a subsequence of this sequence with convergent real part, and then take a subsequence of that sequence with convergent imaginary part. Then both real parts and imaginary parts of this subsequence are convergent, so this subsequence is convergent overall.
\end{proof}

\begin{definition}[Uniform Convergence]
    For a set $S\subset \C$, a sequence of functions $f_i:S\rightarrow\C$ which converges at each point to the function $f:S\rightarrow\C$ uniformly convergence if for every $\varepsilon>0$, there exists $N$ such that for all $i>N$ and $z\in S$, $|f_i(z)-f(z)|<\varepsilon$. Then $f$ is called the uniform limit of $f_i$.
\end{definition}
\begin{thm}
The uniform limit of a sequence of holomorphic functions is also holomorphic.
\end{thm}
\begin{definition}[Absolute convergence]
    An infinite sum $\sum_{i=0}^\infty a_i$ of complex numbers is absolutely convergent if the sum $\sum_{i=0}^\infty |a_i|$ has a finite limit.
\end{definition}
\begin{thm}
    An absolutely convergent sum of complex numbers is convergent. Furthermore, changing the order of the complex numbers keeps the sum convergent, and you get the same answer.
\end{thm}

\begin{definition}[$\epsilon$-neighbourhood]
    A neighbourhood of the point $z_0$ is a set of points that are close to $z_0.$ Specifically, the $\epsilon$-neighbourhood of the point $z_0$ is the set of points z such that $|z-z_0|<\epsilon$.
\end{definition}

Two types of subsets of the complex plane are called open and closed subsets, we will define them in what follows: 

\begin{definition}[Open Subset of complex plane]
    An open subset of the complex plane,  $U \subset \mathbb{C}$, is a set of points where $\forall z\in U$, $\exists  r\in\mathbb{R}$ such if $w\in \mathbb{C}$ satisfies $|z-w|<r$ then $w\in U$. In other words, for any $z\in U$, there exists an $\epsilon$-neighbourhood which is entirely contained within $U$.
\end{definition}
\begin{definition}[Limit Point]
    Let $z_0\in\mathbb{C}$, $S\subset \mathbb{C}$. Then $z_0$ is a limit point of $S$ if for every $\epsilon >0$, there is a point in the $\epsilon$ - neighbourhood of $z_0$, excluding $z_0$, that is in $S$.
\end{definition}
\begin{definition}[Closed Subset of Complex Plane]
   $S$ is a closed subset of $\mathbb{C}$ if and only if every limit point of $S$ also lies in $S$.
\end{definition}

\begin{definition}[Complex Differentiable]
    Let a complex function $f$ be defined on an open subset of the complex number plane. Then the complex derivative at a point $z$ is denoted by $f'(z_0)$ and only exists if the following limit exists: $$f'(z_0)=\lim_{z\rightarrow z_0}\frac{f(z)-f(z_0)}{z-z_0}.$$ 
 
\end{definition}
\begin{definition}[Holomorphic]
    A function is holomorphic if its defined on an open subset of the complex plane, and is complex differentiable at every point in the open subset.
\end{definition}

\begin{definition}[Power Series]
    A power series centered at a point \( z_0 \in \mathbb{C} \) is a series of the form:
\[
f(z) = \sum_{n=0}^{\infty} a_n (z - z_0)^n,
\]
where:

\begin{itemize}
  \item \( a_n \in \mathbb{C} \) are coefficients,
  \item \( z \in \mathbb{C} \),
  \item the series converges for \( z \) in some open disc \( |z - z_0| < R \), where the maximum $R$ is called the \textit{radius of convergence}.
\end{itemize}

Within the disc \( |z - z_0| < R \), the power series defines a holomorphic function.

\end{definition}
Every function that is holomorphic on an open set \( U \subset \mathbb{C} \) can be represented by a power series around any point \( z_0 \in U \), converging to the function in some neighbourhood of \( z_0 \). This is one of the most important results in complex analysis. 

\begin{definition}
    A function defined on an open set $U$ that has a power series expansion about every point $z_0\in U$ is said to be \emph{analytic}.
\end{definition}
\begin{thm}
    Every holomorphic function is analytic.
\end{thm}
Before we list some classic theorems in complex analysis we must first define what it means to take the integral in a complex plane: 
\begin{definition}[Complex Integral]
   Define an integral along a path $\gamma$ parametrised as a (real) differentiable function $\gamma:[t_1,t_2]\rightarrow \mathbb{C}$ be a parametrisation of $\gamma$ that is consistent with its direction.
\newline\newline
   Then the integral of $f$ along $\gamma$ is defined by: $$\int_{\gamma}f(z)dz=\int_{t_1}^{t_2}f(\gamma(t))\gamma'(t)dt.$$
   
\end{definition}

There are other equivalent definitions.

\subsubsection{Classic Theorems for holomorphic functions}
Here are some complex analysis results for holomorphic functions.

\begin{thm}\label{discretezero}
    The zeros of a holomorphic function are discrete. In other words, for any zero of a holomorphic function there exists a neighbourhood of that zero that contains no zero other than itself.
\end{thm}

\begin{thm}[Cauchy's Theorem]
for a function $f$, holomorphic on a simply connected subset of $\C$ with boundary $C$, 
$$\int_C f(z)dz=0$$
    
\end{thm}

\begin{thm}
    The number of zeros of a holomorphic function on a simply connected subset of $\C$ with boundary $C$ is given by
$$\int_C\frac{f'(z)}{f(z)}dz$$
\end{thm}

\begin{thm}[Liouville's Theorem]

A bounded holomorphic function on $\C$ is constant.
    
\end{thm}

\begin{thm}[Riemann's Theorem on Removable Singularities]
    If a holomorphic function is not defined on a complex number but is defined on all the other points in a neighbourhood of that number, and is bounded in that neighbourhood, then the function can be filled in at that point and still be holomorphic.
\end{thm}

\begin{definition}[Laurent Series]
    A Laurent Series is a representation of a complex function as a (possibly infinite) series that includes both non-negative and negative powers of $z$. It generalises the power series by allowing terms of the form $\frac{1}{z^n}$ as well.
\end{definition}

\begin{definition}[Pole]
    $z_0$ is a pole of the function $f$ if it is a zero of the function $\frac{1}{f(z)}$, and $\frac{1}{f(z)}$ is holomorphic in some neighbourhood of $z_0$. The multiplicity of the pole is the same as the multiplicity of the zero created.
\end{definition}

\begin{definition}
    A meromorphic function is one that is holomorphic on all but some isolated points in its domain.
\end{definition}

Those isolated points are removable singularities as well as poles.

\begin{thm}[Argument Principle]\label{argprinciple}For a meromorphic function $f$ and $C$ is a closed curve, then the zeros and poles enclosed by $C$ have difference given by
$$\frac{1}{2\pi i}\int_C\frac{f'(z)}{f(z)}dz= zeros-poles$$
counting multiplicity.
\end{thm}
\subsubsection{M\"{o}bius Transformations}
\begin{definition} [Projective Space]$\boldsymbol{P}_1(\C)$ denotes $(\C^{2}\setminus\{0\})/\sim$ where $\sim$ is the equivalence relation defined by $v_1\sim v_2\iff v_1=\lambda v_2$. In other words, one dimensional complex projective space is the set of one dimensional subspaces of a two dimensional complex vector space.
\end{definition}
Generally people refer to a point in the projective space using the coordinates of a representative element of the equivalence class. Here we will write the coordinates $(z,w)\in\C^2$ as a column vector $\begin{pmatrix}z\\w\end{pmatrix}$ so that we may left multiply by two by two matrices.

\begin{definition} [Riemann Sphere]
The Riemann sphere, denoted $\hat{\mathbb{C}}$ is $\boldsymbol{P}_1(\mathbb{C})$
\end{definition}

An alternative representation of the Riemann sphere (which more people probably use) is $\mathbb{C}\cup \{\infty\}$, adding a "point at infinity" to the complex plane. The way this corresponds to our definition is that for a coordinate $\begin{pmatrix}z\\w\end{pmatrix}$ where $w\neq 0$, $\begin{pmatrix}z\\w\end{pmatrix} \mapsto \frac{z}{w}$, and otherwise $\begin{pmatrix}z\\0\end{pmatrix}\mapsto \infty$.

\begin{definition} [M\"{o}bius Transformation on the Riemann Sphere]
    For $\Mabcd\in\GL_2(\C)$, its associated M\"{o}bius transformation is defined by left-multiplying the representative vector of the point in  $\hat{\mathbb{C}}$ with the matrix. We notate this as $\Mabcd z$ where $z\in \hat{\mathbb{C}}$.
\end{definition}
It's not exactly obvious that this transformation is well-defined, but it's simple to verify that values in the same equivalence class gets sent to the same equivalence class.
\newline\newline
If we apply this definition of M\"{o}bius transformation to the other representation of the Riemann sphere, then we get that 
$$\infty\mapsto\frac{a}{c},\qquad -\frac{d}{c}\mapsto\infty,$$
and otherwise $$z\mapsto\frac{az+b}{cz+d}.$$

This is all explained in the Appendix (including when $c=0$) as well as a proof that it is a group action of $\GL_2(\C)$.
We also have that the M\"{o}bius transformation is holomorphic on all but at most one point.

%% file: modular.tex
\section{Modular Forms}
We attempt to motivate a link between elliptic curves and modular forms and try to go far enough to see that cusp forms of level two and weight two do not exist.
\newline\newline
This is usually not in undergrad courses, maybe in an Honours program or master's coursework. It can be taught in graduate school, but not part of the standard core curricula, and is considered more specialised.
\newline\newline

\subsection{Level 1 Modular Forms}

\subsubsection{Modular Forms of Level 1 Already Encountered}
Before completing the full definition of modular forms, we introduce some examples of familiar objects that turn out to be modular forms (of level 1):

\begin{itemize}
    \item The most striking example comes from our discriminant $\Delta$, which when thought of as a complex function that assigns a complex number to every elliptic curve over $\C$, will reveal itself as a weight 12 cusp form of level 1.
    \item This also applies to the $j$-invariant, which is an example of a \emph{modular function}, i.e. weight zero modular form, of level 1.
\end{itemize}

If you recall what has been introduced in the outline, you might ask the question: \textit{How a function whose domain is the set of elliptic curves be an example of something whose domain is the upper half-plane?} We look into this further as the discussion develops.

\subsubsection{Definition of Level 1 modular forms}
A large chunk of the definition already appears scattered in the outline but some of it is presented differently and some key properties were not discussed.
\newline\newline
The first key property that we incorporate here is that modular forms are \emph{holomorphic functions}, which means they are complex differentiable.
\newline\newline
We also package the $a,b,c,d$ together using matrices, in particular $\SLZ$.
\newline\newline
Here we focus on level 1 modular forms and defer the definition of \emph{congruence subgroups} that unlocks the level $N$ modular forms.

\begin{definition}[Upper Half Plane $\mathcal{H}$]
The upper half plane $\mathcal{H}$ is the set of complex numbers $z$ such that $\operatorname{Im}(z)>0$.
\end{definition}

\begin{proposition}
    The Möbius transformation defines a group action of $\SLZ$ on $\mathcal{H}$.
\end{proposition}
\begin{proof} We put a proof in the Appendix that $\SL_2(\C)$ acts and this is a subgroup so the restriction gives an action.


\end{proof}

We take $\Gamma=\SLZ$ in the following definition and change this to obtain different levels.
\begin{definition} [Weakly Modular Form of Level 1]

A holomorphic function $f:\mathcal{H}\rightarrow\mathbb{C}$ is a weakly modular form of weight $k$ with respect to group $\Gamma$ if
$$f\left(\frac{az + b}{cz + d}\right)=(cz+d)^kf(z)\text{ for all }z\in \mathbb{C}, \Mabcd\in\Gamma.$$
\end{definition}

There is also one more component: that of being holomorphic at $\infty$, which we also delay until after understanding the condition of weakly modular better.

\subsubsection{Why the Upper Half Plane?}
You might wonder: why do modular forms live on the upper half-plane 
\[
\mathcal{H} = \{ z \in \mathbb{C} \mid \operatorname{Im}(z) > 0 \},
\]
rather than on all of \( \mathbb{C} \)?
\newline\newline
We begin with the action of \( \mathrm{GL}_2(\mathbb{Z}) \), the group of \( 2 \times 2 \) integer matrices with determinant \( \pm 1 \). This group acts on the extended complex plane 
\[
\widehat{\mathbb{C}} = \mathbb{C} \cup \{\infty\}
\]
via Möbius transformations:
\[
z \mapsto \frac{az + b}{cz + d}.
\]

We can simplify this situation by restricting our action to $\mathcal{H}$ and at the same time restricting our group to $\SL_2(\Z)$. Every matrix in $\GL_2(\Z)$ is either in $\SL_2(\Z)$ or in it after multiplying by a matrix $a=-1,b=c=0, d=1$ and this matrix reflects the upper half plane to the lower half plane.
\newline\newline
But why not allow all Möbius transformations?
\newline\newline
Well, while \( \mathrm{SL}_2(\mathbb{Z}) \) represents only a subset of Möbius transformations, it plays a special arithmetic role: it acts linearly on the lattice generators \( \omega_1, \omega_2 \in \mathbb{C} \) via
\[
\begin{pmatrix}
a & b \\
c & d
\end{pmatrix}
\cdot (\omega_1, \omega_2) = (a\omega_1 + b\omega_2,\, c\omega_1 + d\omega_2),
\]
which results in a new pair of generators of the same lattice \( \Lambda = \mathbb{Z}\omega_1 + \mathbb{Z}\omega_2 \). That is, \( \mathrm{SL}_2(\mathbb{Z}) \) acts on the set of bases of a fixed lattice. This leads to the following:

\begin{definition}
    Let $\mathscr{L}$ be the set of lattices
\end{definition}

We give another description of $\mathcal{H}$ to add some understanding:
\newline\newline
Consider a lattice $\mathcal{L}$ with a chosen basis $\omega_1,\omega_2$ oriented so that $\omega_2/\omega_1$ has positive imaginary part. If we act on these pairs by $\SLZ$ as 
$$w_1\mapsto cw_1+dw_2$$
$$w_2\mapsto aw_1+bw_2$$

then the orbits are in one to one correspondence with the lattices in $\mathscr{L}$.
\newline\newline
If we now introduce scaling of the lattices then up to scaling every lattice corresponds to a unique pair with $w_1=1$ and $\text{Im}(w_2)>0$:

$$w_1,w_2\mapsto 1, w_2/w_1$$

we then get $w_2/w_1\in \mathcal{H}$. The group action gives
$$g(w_1,w_2)=g(1,\omega_2/\omega_1=z)\mapsto \frac{az+b}{cz+d}=gz$$

So we have a map 

$$\mathscr{L}/\C^*\to \mathcal{H}/\SLZ $$

\begin{lemma} The $\SLZ$ orbits of $\mathcal{H}$ are in one to one correspondence with the set of lattices $\mathcal{L}$ up to scaling 
\end{lemma}
\begin{proof}
First we show that two complex numbers in one orbit correspond to the same lattice. If we have $z=\frac{\omega_2}{\omega_1}$ and $\gamma=\Mabcd\in\SLZ$, then if we define $\begin{pmatrix} \omega_4\\\omega_3 \end{pmatrix}=\gamma\begin{pmatrix} \omega_2\\\omega_1 \end{pmatrix}$, i.e. $\omega_4=a\omega_2+b\omega_1$ and $\omega_3=c\omega_2+d\omega_1$, then using the projective interpretation of the Mobius transformation, $\frac{\omega_4}{\omega_3}=\gamma\left(\frac{\omega_2}{\omega_1}\right)$. To show the lattices are equal, consider a point on the $\Z\omega_3+\Z\omega_4$ lattice, so $k_3\omega_3+k_4\omega_4$ where $k_3,k_4\in\Z$. Then
$$k_3\omega_3+k_4\omega_4=k_3(c\omega_2+d\omega_1)+k_4(a\omega_2+b\omega_1)=(dk_3+bk_4)\omega_1+(ck_3+ak_4)\omega_2\in\Z\omega_1+\Z\omega_2$$
On the other hand, if we have a point $k_1\omega_1+k_2\omega_2$ on the $\Z\omega_1+\Z\omega_2$ lattice, then we can use a similar argument to show that this point is also on the $\Z\omega_3+\Z\omega_4$ lattice, except this time instead of using $\begin{pmatrix} \omega_4\\\omega_3 \end{pmatrix}=\gamma\begin{pmatrix} \omega_2\\\omega_1 \end{pmatrix}$ we use $\begin{pmatrix} \omega_2\\\omega_1 \end{pmatrix}=\gamma^{-1}\begin{pmatrix} \omega_4\\\omega_3 \end{pmatrix}$ where $\gamma^{-1}$ has integer entries. Therefore these two lattices are the same, and so $\Z+\Z\frac{\omega_2}{\omega_1}$ and $\Z+\Z\frac{\omega_4}{\omega_3}=\Z+\Z\gamma\left(\frac{\omega_2}{\omega_1}\right)$ are equivalent up to scaling.

Now we need to show conversely that if two lattices are equivalent up to scaling, then they correspond to the same orbit. If we have such equivalent lattices $\Z\omega_1+\Z\omega_2$ and $\Z\omega_3+\Z\omega_4$, then we can scale one of them, say the second one, to make them the same lattice, and that will keep $\frac{\omega_4}{\omega_3}$ the same, so up to swapping, they correspond to $\frac{\omega_2}{\omega_1},\frac{\omega_4}{\omega_3}\in\mathcal H$. Then since $\omega_2\in\Z\omega_3+\Z\omega_4$, there exist $a,b\in\Z$ such that $\omega_2=a\omega_4+b\omega_3$. We can do similar things to $\omega_1,\omega_3,\omega_4$ to get the equations
\begin{align*}
    \omega_2&=a\omega_4+b\omega_3\\
    \omega_1&=c\omega_4+d\omega_3\\
    \omega_4&=e\omega_2+f\omega_1\\
    \omega_3&=g\omega_2+h\omega_1\\
\end{align*}
where $a,b,c,d,e,f,g,h\in\Z$, so then substituting the third and fourth equations into the first two we get
\begin{align*}
    \omega_2&=(ae+bg)\omega_2+(af+bh)\omega_1\\
    \omega_1&=(ce+dg)\omega_2+(cf+dh)\omega_1
\end{align*}
The coefficients of these linear combinations of $\omega_1,\omega_2$ are uniquely determined so putting them into a matrix we get that
$$\begin{pmatrix}
    ae+bg&af+bh\\
    ce+dg&cf+dh
\end{pmatrix}=\Mabcd\begin{pmatrix} e&f\\g&h \end{pmatrix}=\begin{pmatrix} 1&0\\0&1 \end{pmatrix}$$
Therefore $\det\Mabcd\det\begin{pmatrix} e&f\\g&h \end{pmatrix}=1$, and since these matrices have integer coefficients, their determinants are integers which must then each be $\pm1$.

If $\det\Mabcd=-1$, then define $\gamma'=\begin{pmatrix}c&d\\a&b\end{pmatrix}$, so $\det\gamma'=1$ so it is in $\SLZ$, and since $\begin{pmatrix} \omega_1\\\omega_2 \end{pmatrix}=\gamma'\begin{pmatrix} \omega_4\\\omega_3 \end{pmatrix}$, then $\frac{\omega_1}{\omega_2}=\gamma'\left(\frac{\omega_4}{\omega_3}\right)$ so $\frac{\omega_1}{\omega_2}\in\mathcal H$ which is a contradiction since $\frac{\omega_2}{\omega_1}\in\mathcal H$. Therefore $\det\Mabcd=1$, and $\Mabcd\in\SLZ$. Then $\frac{\omega_2}{\omega_1}=\Mabcd\left(\frac{\omega_4}{\omega_3}\right)$, so the lattices indeed map to complex numbers in the same orbit.
\end{proof}

So we can also think of these functions as defined on a lattice with a chosen basis.
\newline\newline
If we believe that a complex elliptic curve corresponds to a lattice, then this also explains how the domains of $j$ and $\Delta$ end up being $\mathcal{H}$.
\newline\newline
Now we turn our attention to the action's affect on $f$.

\begin{prop}
    For $f:\mathcal{H}\rightarrow\mathbb{C}$ it satisfies
$$f\left(\frac{az+b}{cz+d}\right)=(cz+d)^kf(z)\text{ for all }z\in \mathbb{C}, \Mabcd\in\SL_2(\Z).$$
If and only if it is both periodic:
$$f(z+1)=f(z),$$
and satisfies
 $$f(-1/z)= z^kf(z).$$

\end{prop}

 \begin{proof}
      The two matrices that give rise to the last two conditions are $a=1,b=1,c=0,d=1$ and $a=0,b=1,c=-1,d=0$.
 \newline\newline     
      To get the converse, it suffices to show that these two matrices generate $\SLZ$, which means of every element of $\SL_2(\Z)$ can be expressed as a product (of possibly many matrices) of these two matrices and their inverses.
\newline\newline
      We omit that they generate $\SL_2(\Z)$ but explain why we also have the inverses of the two matrices. The second matrix is its own inverse, while the first condition with $z\mapsto z-1$ gives the same matrix with $b=-1$, which is the inverse.

 \end{proof}

 In general we get periodic but not the second one for higher than level one.

 \subsubsection{Periodic Gives Fourier Series in $q$}
 This actually means that, by making the substitution $q=e^{2\pi iz}$, and letting $D$ be the unit disc of the complex plane, there exists a unique function $g:D\setminus\{0\}\rightarrow\mathbb{C}$ such that $g(q)=f(z)$. $g$ is then also holomorphic on $D\setminus\{0\}$.

\subsubsection{Holomorphic at Infinity}

Because the disc $D$ is missing $0$, $g$ may not necessarily be nicely behaved around $0$, based on the behaviour of $f$ as the imaginary part approaches infinity. We say that a weakly modular form $f$ is "holomorphic at infinity" if its corresponding $g$ can be continued to $0$ such that $g$ ends up being holomorphic at $0$.
\newline\newline
Remark about holomorphicness of function of $q$: A function of $q$ can be expanded as a Laurent series, which is a power series where the powers of $q$ are allowed to be negative. There are general situations for functions having a Laurent series, while also, in this case, it amounts to a periodic function having a Fourier series. One practical way to defining holomorphic here is just to say that there are no negative powers of $q$ in this expansion, so that taking $q=0$ should be a valid move on this expression once the point of infinity is added.
\newline\newline
An interesting fact is that by Riemann's theorem on removable singularities, if $g$ is bounded on some neighbourhood of $0$, then it is indeed possible to holomorphically continue $g$ to $0$. The converse is true by continuity. Therefore, $f$ being holomorphic at infinity is equivalent to $f$ being bounded past some value for its imaginary part.

\subsubsection{Eisenstein Series}

We define an infinite family of level 1 modular forms.

\begin{definition}
    For a fixed lattice $\mathcal L$, define the values 
    $$G_k(\mathcal L)=\sum_{\omega\in\mathcal L\setminus\{0\}}\omega^{-k}$$
    where $k\geq3$.  Furthermore, if we define $G_k(z)$ for $z\in\mathcal H$ to equal the value of $G_k(\mathcal L)$ where $\mathcal L$ is the lattice generated by $1,z$, we claim that we get a modular form of level 1.
\end{definition} 

In order to prove that this sum converges, we need to establish a lemma about lattices.
\begin{lemma}
    For a lattice $\Z\omega_1+\Z\omega_2$, there exists $\lambda>0$ such that for all $k_1,k_2\in\Z$, $|k_1\omega_1+k_2\omega_2|\geq\lambda(|k_1|+|k_2|)$.
\end{lemma}
\begin{proof}
    If $k_1,k_2\geq0$, then since $\omega_1$ and $\omega_2$ are not real multiples of each other, we can choose a direction in the complex plane, say represented by $v\in\C$ such that $\omega_1$ and $\omega_2$ have positive component in that direction. An example of such a direction is the angle bisector of $\omega_1$ and $\omega_2$. Denote by $\operatorname{comp}_vw$ the component of $w$ in the $v$ direction. Then setting $\lambda=\min(\operatorname{comp}_v\omega_1,\operatorname{comp}_v\omega_2)>0$, then
    $$\left|k_1\omega_1+k_2\omega_2\right|\geq \operatorname{comp}_v(k_1\omega_1+k_2\omega_2)=k_1\operatorname{comp}_v\omega_1+k_2\operatorname{comp}_v\omega_2\geq\lambda(|k_1|+|k_2|)$$
    For the other three cases of the signs of $k_1,k_2$, we can do a similar argument where we pick a $v$ such that the component of $\pm\omega_1$ and $\pm\omega_2$ are positive, where the signs of each of the $\pm$ depend on the signs of $k_1,k_2$. Then taking the minimum of the $\lambda$'s in each case to be our final $\lambda$, the lemma is proved.
\end{proof}

\begin{proof}[Proof this sum converges absolutely]\label{minus3converge}

It suffices to consider $k=3$.

Since we are looking for absolute convergence, we can evaluate this sum in whatever order we want. If we let $\omega_1,\omega_2$ be a basis for the lattice, then each $\omega$ can be written uniquely as $k_1\omega_1+k_2\omega_2$ for $k_1,k_2\in\mathbb Z$ which are not both $0$. 

There exists $\lambda>0$ such that $|k_1\omega_1+k_2\omega_2|\geq \lambda(|k_1|+|k_2|)$ for all $k_1,k_2\in\mathbb Z$.

Therefore we can order $\mathcal{L}\setminus \{0\}$ by the values of $|k_1|+|k_2|=n>0$, where for each $n$ there are at most $4n$ $\omega$'s in $\mathcal L\setminus \{0\}$ with that particular value of $n$. 
$$\begin{aligned} 
\sum_{\omega\in\mathcal{L}\setminus \{0\}}|\omega^{-3}|
&=\sum_{n=1}^{\infty}\sum_{|k_1|+|k_2|=n}|k_1\omega_1+k_2\omega_2|^{-3}\\
&\leq\sum_{n=1}^{\infty}\sum_{|k_1|+|k_2|=n}((|k_1|+|k_2|)\lambda)^{-3}\\
&=\sum_{n=1}^{\infty}\sum_{|k_1|+|k_2|=n}(n\lambda)^{-3}\\
&=\lambda^{-3}\sum_{n=1}^{\infty}\sum_{|k_1|+|k_2|=n}n^{-3}\\
&\leq\lambda^{-3}\sum_{n=1}^{\infty}4nn^{-3}\\
&\leq4\lambda^{-3}\sum_{n=1}^{\infty}n^{-2}<\infty.
\end{aligned}$$
\end{proof} 

We can also prove that the sum used to form $G_k(z)$, when placed in some order, converges uniformly on compact sets in $\mathcal H$.
\begin{proof}[Proof this sum converges uniformly on compact sets]
   Subbing in the definition for $G_k(z)$, since the sum is absolutely convergent we can rearrange the absolute sum
   \begin{align*}
   & \sum_{(C, D) \in \Z^2 \setminus \{(0, 0)\}} \left| \frac{1}{(Cz +
  D)^k} \right|\\
  & = \sum_{D \in \Z \setminus \{0\}} \left| \frac{1}{D^k} \right| + \sum_{C
  \in \mathbb{Z} \setminus \{ 0 \}} \sum_{D \in \mathbb{Z}} \left|
  \frac{1}{(Cz + D)^k} \right|
\end{align*}
For each $C\in\Z\setminus\{0\}$ we can choose $w_C=Cz+m_C$ for some $m_C\in\mathbb Z$ such that $\operatorname{Re}(w_C)\in[0,1)$, so now this sum can be expressed as
\begin{align*}
  & \sum_{D \in \Z \setminus \{0\}} \left| \frac{1}{D^k} \right| + \sum_{C
  \in \mathbb{Z} \setminus \{0\}} \sum_{D \in \mathbb{Z}} \left| \frac{1}{(w_C
  + D)^k} \right|\\
  & = \sum_{D \in \Z \setminus \{0\}} \left| \frac{1}{D^k} \right| + \sum_{D
  \in \mathbb{Z}} \sum_{C \in \mathbb{Z} \setminus \{0\}} \left| \frac{1}{(w_C
  + D)^k} \right|\\
  & = \sum_{D \in \Z \setminus \{0\}} \left| \frac{1}{D^k} \right| + \sum_{C
  \in \mathbb{Z} \setminus \{0\}} \left| \frac{1}{(w_C)^k} \right| + \sum_{D
  \in \mathbb{Z}^+} \sum_{C \in \mathbb{Z} \setminus \{0\}} \left|
  \frac{1}{(w_C + D)^k} \right| + \sum_{D \in \mathbb{Z}^+} \sum_{C \in
  \mathbb{Z} \setminus \{0\}} \left| \frac{1}{(w_C - D)^k} \right|
\end{align*}
Let $I=\operatorname{Im }(z)>0$. Then $\operatorname{Im}(w_C)=\operatorname{Im}(Cz)=CI$. For $D\in\Z^+$, $\operatorname{Re}(w_C+D)\geq D>0$, and $\operatorname{Re}(w_C-D)\leq -D+1<0$. Therefore $|w_C+D|\geq|CIi+D|$ and $|w_C-D|\geq|CIi-D+1|$. So
\begin{align*}
  &\sum_{D \in \Z \setminus \{0\}} \left| \frac{1}{D^k} \right| + \sum_{C
  \in \mathbb{Z} \setminus \{0\}} \left| \frac{1}{(w_C)^k} \right| + \sum_{D
  \in \mathbb{Z}^+} \sum_{C \in \mathbb{Z} \setminus \{0\}} \left|
  \frac{1}{(w_C + D)^k} \right| + \sum_{D \in \mathbb{Z}^+} \sum_{C \in
  \mathbb{Z} \setminus \{0\}} \left| \frac{1}{(w_C - D)^k} \right|\\
  &\leq\sum_{D \in \Z \setminus \{0\}} \left| \frac{1}{D^k} \right| + \sum_{C
  \in \mathbb{Z} \setminus \{0\}} \left| \frac{1}{(w_C)^k} \right| + \sum_{D
  \in \mathbb{Z}^+} \sum_{C \in \mathbb{Z} \setminus \{0\}} \left|
  \frac{1}{(CIi + D)^k} \right| + \sum_{D \in \mathbb{Z}^+} \sum_{C \in
  \mathbb{Z} \setminus \{0\}} \left| \frac{1}{(CIi-D+1)^k} \right|\\
  &= \sum_{C \in \mathbb{Z} \setminus \{0\}} \left| \frac{1}{(w_C)^k} \right| +\sum_{D \in \Z \setminus \{0\}} \left| \frac{1}{D^k} \right| + \sum_{D
  \in \mathbb{Z}} \sum_{C \in \mathbb{Z} \setminus \{0\}} \left|
  \frac{1}{(CIi + D)^k} \right|\\
\end{align*}

We know that $\sum_{(C,D)\in\Z^2\setminus\{(0,0)\}}\frac1{(CIi+D)^k}=G_k(Ii)$ which absolutely converges, so we can collect the two sums on the right to get
\begin{align*}
  &\sum_{C \in \mathbb{Z} \setminus \{0\}} \left| \frac{1}{(w_C)^k} \right| +\sum_{(C,D) \in \mathbb{Z}^2 \setminus \{(0,0)\}} \left| \frac{1}{(CIi + D)^k} \right|\\
  &\leq  \sum_{C \in \mathbb{Z} \setminus \{0\}} \left| \frac{1}{(CIi)^k} \right| +\sum_{(C,D) \in \mathbb{Z}^2 \setminus \{(0,0)\}} \left| \frac{1}{(CIi + D)^k} \right|
\end{align*}
which is a sum only depending on $I$ and decreasing as $I$ becomes larger. The first sum in particular is just a sub-sum of the second one so it's finite. For any compact subset of $\mathcal H$, it has a minimum imaginary value (which must be positive), and hence we can find an upper bound on the absolute sum. This actually gives us an upper bound for $G_k$ given a lower bound on $\operatorname{Im}(z)$. Furthermore, for any $\varepsilon>0$, there exists a finite subset of the terms in the above absolute sum such that the difference between the sum of those elements and the actual sum is less than $\varepsilon$. In the above algebra we have always kept the terms of the original sum separate, so tracing this subset back to the original absolute sum for $G_k(z)$, that sum converges uniformly on compact sets.
\end{proof}

Then since the sum for $G_k(z)$ converges uniformly on compact sets, $G_k$ is holomorphic on $\mathcal H$.

\begin{prop}
    $G_k$ is a modular form of level 1 and weight $k$.
\end{prop}
\begin{proof}
    $$G_k(z)=\sum_{\omega\in\mathcal L\setminus\{0\}}\omega^{-k}=\sum_{(C,D)\in\mathbb Z^2\setminus(0,0)}(Cz+D)^{-k}$$

and then for any $\begin{pmatrix} a&b\\c&d \end{pmatrix}\in\SLZ$,

\begin{align*}
G_k\left(\frac{az+b}{cz+d}\right)&=\sum_{(C,D)\in\mathbb Z^2\setminus(0,0)}\left(C\frac{az+b}{cz+d}+D\right)^{-k}\\
&=(cz+d)^k\sum_{(C,D)\in\mathbb Z^2\setminus(0,0)}\left(C(az+b)+D(cz+d)\right)^{-k}\\
&=(cz+d)^k\sum_{(C,D)\in\mathbb Z^2\setminus(0,0)}\left((Ca + Dc)z+(Cb+Dd)\right)^{-k}
\end{align*}
and we notice that $\begin{pmatrix}Ca+Dc\\Cb+Dd\end{pmatrix}=\begin{pmatrix}a&c\\b&d\end{pmatrix}\begin{pmatrix}C\\D\end{pmatrix}$ where $\begin{pmatrix}a&c\\b&d \end{pmatrix}=\begin{pmatrix} a&b\\c&d \end{pmatrix}^\intercal\in\SLZ$. Therefore the $(Ca + Dc)z+(Cb+Dd)$s are just a permutation of the $Cz+D$s so this equals

$$(cz+d)^k\sum_{(C,D)\in\mathbb Z^2\setminus(0,0)}\left(Cz+D\right)^{-k}=(cz+d)^kG_k(z).$$

Therefore $G_k$ is actually weakly modular with weight $k$. Furthermore, in the proof section above, we showed that given a lower bound for the imaginary part of $z$, then $|G_k(z)|$ is bounded. Hence, $G_k(z)$ is bounded for sufficiently large imaginary part of $z$, so it is a modular form of weight $k$.

\end{proof}

Note for $k$ odd, $G_k=0$ since every non-zero lattice point has a negative.
\newline\newline
Bonus:

If we add, multiply or scale modular forms of level 1 we get a modular form of level 1. This makes the set of modular forms of level 1 a commutative $\C$-algebra. If we multiply weight $k$, with weight $l$, we get weight $k+l$, which defines whats called a \emph{grading}.
\newline\newline
One main example of a graded commutative $\C$-algebra is that of a polynomial ring graded by degree.
\newline\newline
Let $\C[G_4,G_6]$ be a polynomial ring in $G_4$ and $G_6$ graded so that the degree of $G_4$ is $4$ and the degree of $G_6$ is $6$.

\begin{thm} The modular forms of level 1 is equal to
$$\C[G_4,G_6]$$
    as a commutative graded $\C$-algebra.
\end{thm}

\subsection{The Weierstrass $\wp$ Function}

This section deals with a very natural and important complex function that both gives many interesting examples of level 1 modular forms but also allows us to further carry out uniformisation.

\subsubsection{Definition and Absolute Convergence for $\wp$}

\begin{definition} [Weierstrass $\wp$ function]
   Let $\mathcal{L}$ be a complex lattice. The Weierstrass $\wp$ function is a function defined on $\mathbb{C}\setminus\mathcal{L}$ such that
   $$\wp(z)=\frac{1}{z^2}+\sum_{\omega\in\mathcal{L}\setminus\{0\}}\left(\frac{1}{(z-\omega)^2}-\frac{1}{\omega^2}\right).$$
\end{definition}

Before we prove the absolute convergence of the sum, we note that the formula is not as strange as it might seem. If you want to define something obviously periodic you might start with:

$$\sum_{\omega\in\mathcal{L}}\frac{1}{(z-\omega)^2},$$
but this won't converge, so something $\frac{1}{\omega^2}$ is needed to make it smaller. Of course, we cannot have $\omega=0$ in the denominator so we put the term where $\omega=0$ separately and require the sum to not have $\omega=0$.

This will stay periodic once we check that the definition is absolutely convergent.
\newline\newline
Here it is not so obvious what $\sum_{\omega\in\mathcal{L}\setminus\{0\}}$ means, since we don't know what order we are adding each of the lattice points in. To solve this we can order the lattice points by their modulus from smallest to biggest, since there are only a finite number of lattice points with modulus less than anything. However soon we will see that it doesn't really matter in what order we arrange the sum.

\begin{lemma}
The sum 
$$\sum_{\omega\in\mathcal{L}\setminus\{0\}}\left(\frac{1}{(z-\omega)^2}-\frac{1}{\omega^2}\right)$$
absolutely converges for all $z\in\mathbb{C}\setminus\mathcal{L}$, and also uniformly converges on any compact (closed and bounded) subset of $\mathbb{C}\setminus\mathcal{L}$.
\end{lemma} 
\begin{proof}
To show this, we will show that the absolute sum uniformly converges on any compact subset $C\subset\mathbb{C}\setminus\mathcal{L}$. This would prove it absolutely converges anywhere since any point is contained in a compact set, and obviously this would also show that the sum uniformly converges in $C$.
\newline\newline
We assume that $C$ is nonempty because otherwise the statement is vacuously true. Since $C$ is compact, then $C$ is bounded, so let $m\in\mathbb R^+$ be an upper bound to the modulus of points in $C$, so then for all $z\in C$, we have $|z|\leq m$. Now let the lattice points with modulus $\leq 2m$ be $M$ (so $0\in M$). $M$ is a finite subset of $\mathcal L$. We can now write
$$\sum_{\omega\in\mathcal{L}\setminus\{0\}}\left|\frac{1}{(z-\omega)^2}-\frac{1}{\omega^2}\right|=\sum_{\omega\in M\setminus\{0\}}\left|\frac{1}{(z-\omega)^2}-\frac{1}{\omega^2}\right|+\sum_{\omega\in\mathcal{L}\setminus M}\left|\frac{1}{(z-\omega)^2}-\frac{1}{\omega^2}\right|$$
where the first sum is finite, so it suffices to prove that the second sum uniformly converges.
\newline\newline
Since we are looking for absolute convergence, we can evaluate this sum in whatever order we want. If we let $\omega_1,\omega_2$ be a basis for the lattice, then each $\omega$ can be written uniquely as $k_1\omega_1+k_2\omega_2$ for $k_1,k_2\in\mathbb Z$ which are not both $0$. There exists $\lambda>0$ such that $|k_1\omega_1+k_2\omega_2|\geq \lambda(|k_1|+|k_2|)$ for all $k_1,k_2\in\mathbb Z$. Therefore we can order $\mathcal{L}\setminus M$ by the values of $|k_1|+|k_2|=n>0$, where for each $n$ there are at most $4n$ $\omega$'s in $\mathcal L\setminus M$ with that particular value of $n$. 
\begin{align*}
\sum_{\omega\in\mathcal{L}\setminus M}\left|\frac{1}{(z-\omega)^2}-\frac{1}{\omega^2}\right|&=\sum_{\omega\in\mathcal{L}\setminus M}\left|\frac{z(z-2\omega)}{(z-\omega)^2\omega^2}\right|\\
&\leq\sum_{\omega\in\mathcal{L}\setminus M}\left|\frac{m(|\omega|+|2\omega|)}{(z-\omega)^2\omega^2}\right|\\
&=\sum_{\omega\in\mathcal{L}\setminus M}\left|\frac{m \times3\omega}{(z-\omega)^2\omega^2}\right|\\
&\leq\sum_{\omega\in\mathcal{L}\setminus M}\frac{m}{|z-\omega|^2|\omega|}\\
&\leq\sum_{\omega\in\mathcal{L}\setminus M}\frac{m}{\left||\omega|-|z|\right|^2|\omega|}
\end{align*}
Here, $|z|\leq m\leq \frac{|\omega|}{2}\leq|\omega|$, so $|\omega|-|z|$ is non-negative. Thus this is
\begin{align*}
&\leq\sum_{\omega\in\mathcal{L}\setminus M}\frac{m}{\left||\omega|-\frac{|\omega|}{2}\right|^2|\omega|}\\
&=\sum_{\omega\in\mathcal{L}\setminus M}\frac{m}{\left(\frac{|\omega|}{2}\right)^2|\omega|}\\
&=2m\sum_{\omega\in\mathcal{L}\setminus M}|\omega|^{-3}\\
&=2m\sum_{n=0}^{\infty}\sum_{|k_1|+|k_2|=n}|k_1\omega_1+k_2\omega_2|^{-3}\\
\end{align*}
which can be checked to converge uniformly using same calculation as Proposition \ref{minus3converge}.
\end{proof}

\subsubsection{$\wp'$ the derivative of $\wp$}

Holomorphic functions that uniformly converge on compact sets converge to a holomorphic function whose derivative is the limit of the derivative (using Cauchy's Integral Formula) so the derivative of $\wp$ is equal to the sum of the derivatives of the summands, which is
$$\wp'(z)=-2\times\frac{1}{z^3}-2\sum_{\omega\in\mathcal L\setminus\{0\}}\frac{1}{(z-\omega)^3}=-2\sum_{\omega\in\mathcal L}\frac{1}{(z-\omega)^3}$$
which also converges uniformly and is clearly periodic with respect to $\mathcal L$, i.e. $\wp'(z)=\wp'(z+\omega)$ for all $z\in\mathbb C\setminus\mathcal L,\omega\in\mathcal L$. This allows us to give an alternate proof of periodicity of $\wp$.

\subsubsection{Proof $\wp$ is periodic}

We get $\wp$ is periodic, essentially by design, but here is another proof:

\begin{prop}$\wp$ is periodic with respect to $\mathcal L$
\end{prop}
\begin{proof}Let $\omega_1,\omega_2$ be generators for $\mathcal L$, so it suffices to prove that $\wp(z)=\wp(z+\omega_1)$ and $\wp(z)=\wp(z+\omega_2)$ for all $z\in\mathbb C\setminus\mathcal L$. Looking at the definition, it's clear that $\wp$ is even. Therefore $\wp(\frac{\omega_1}{2})=\wp(-\frac{\omega_1}{2})$. Since $\mathbb C\setminus\mathcal L$ is connected, find a contour $\gamma_1$ from $-\frac{\omega_1}{2}$ to $z$ within $\mathbb C\setminus\mathcal L$. Then let $\gamma_2$ be a translation of $\gamma_1$ by $\omega_1$, so $\gamma_2$ goes from $\frac{\omega_1}{2}$ to $z+\omega_1$. Now since $\wp$ is holomorphic, $\wp'$ is also holomorphic and in particular continuous, so using the Fundamental Theorem of Calculus,
$$\wp(z)-\wp\left(-\frac{\omega_1}2\right)=\int_{\gamma_1}\wp'(t)\mathrm{d}t=\int_{\gamma_2}\wp'(t)\mathrm{d}t=\wp(z+\omega_1)-\wp\left(\frac{\omega_1}2\right)$$
so therefore $\wp(z)=\wp(z+\omega_1)$. A similar argument can be used for $\omega_2$, proving that $\wp$ is $\mathcal L$-periodic.

\end{proof}

\subsubsection{Relationship with Eisenstein Series}

Now if we write
\begin{align*}
    \wp(z)&=\frac{1}{z^2}+\sum_{\omega\in\mathcal{L}\setminus\{0\}}\left(\frac{1}{(z-\omega)^2}-\frac{1}{\omega^2}\right)\\
    &=\frac{1}{z^2}+\sum_{\omega\in\mathcal{L}\setminus\{0\}}\frac{1}{\omega^2}\left(\frac{1}{(1-\frac{z}{\omega})^2}-1\right)\\
    &=\frac{1}{z^2}+\sum_{\omega\in\mathcal{L}\setminus\{0\}}\frac{1}{\omega^2}\left(\left(\frac{1}{1-\frac{z}{\omega}}\right)^2-1\right)\\
    &=\frac{1}{z^2}+\sum_{\omega\in\mathcal{L}\setminus\{0\}}\frac{1}{\omega^2}\left(\left(\sum_{n=0}^\infty\frac{z^n}{\omega^n}\right)^2-1\right)\\
    &=\frac{1}{z^2}+\sum_{\omega\in\mathcal{L}\setminus\{0\}}\frac{1}{\omega^2}\left(\sum_{n=0}^\infty(n+1)\frac{z^n}{\omega^n}-1\right)\\
    &=\frac{1}{z^2}+\sum_{\omega\in\mathcal{L}\setminus\{0\}}\frac{1}{\omega^2}\sum_{n=1}^\infty(n+1)\frac{z^n}{\omega^n}\\
    &=\frac{1}{z^2}+\sum_{\omega\in\mathcal{L}\setminus\{0\}}\sum_{n=1}^\infty(n+1)\frac{z^n}{\omega^{n+2}}.\\
\end{align*}
It can be checked that this last sum absolutely converges when $z$ has smaller modulus than all non-zero lattice points since then $\left|\frac{z}{\omega}\right|<1$, so we can swap the sum order
\begin{align*}
    &=\frac{1}{z^2}+\sum_{n=1}^\infty\sum_{\omega\in\mathcal{L}\setminus\{0\}}(n+1)\frac{z^n}{\omega^{n+2}}\\
    &=\frac{1}{z^2}+\sum_{n=1}^\infty(n+1)z^n\sum_{\omega\in\mathcal{L}\setminus\{0\}}\frac{1}{\omega^{n+2}}\\
    &=\frac{1}{z^2}+\sum_{n=1}^\infty(n+1)G_{n+2}(\mathcal L)z^n
\end{align*}
which gives us the Laurent series for $\wp$. Note that the odd $k$ terms are zero,  that is $G_{k}(\mathcal L)=0$ for odd $k$ and hence $\wp$ is even.

\subsection{Uniformisation via $\wp$}

We now revisit the uniformisation previewed in the last Section with renewed vigour.

\begin{theorem}
    The Weierstrass elliptic function satisfies the following differential equation:
$$\wp'(z)^2=4\wp(z)^3-60G_4(\mathcal L)\wp(z)-140G_6(\mathcal L)$$

\end{theorem}
\begin{proof}
    
To verify this, we use the Laurent series of the left and right hand sides, and take the non-positive exponent terms. We have that
$$\wp(z)=z^{-2}+3G_4(\mathcal L)z^2+5G_6(\mathcal L)z^4+\cdots$$
$$\wp'(z)=-2z^{-3}+6G_4(\mathcal L)z+20G_6(\mathcal L)z^3+\cdots$$
$$\wp(z)^3=z^{-6}+9G_4(\mathcal L)z^{-2}+15G_6(\mathcal L)+\cdots$$
$$\wp'(z)^2=4z^{-6}-24G_4(\mathcal L)z^{-2}-80G_6(\mathcal L)+\cdots$$
$$\wp'(z)^2=4(z^{-6}-6G_4(\mathcal L)z^{-2}-20G_6(\mathcal L))+\cdots$$
\begin{align*}
4\wp(z)^3-60G_4(\mathcal L)\wp(z)-140G_6(\mathcal L)&=4(\wp(z)^3-15G_4(\mathcal L)\wp(z)-35G_6(\mathcal L))\\
&=4(z^{-6}+(9-15)G_4(\mathcal L)z^{-2}+(15-35)G_6(\mathcal L))+\cdots\\
&=4(z^{-6}-6G_4(\mathcal L)z^{-2}-20G_6(\mathcal L))+\cdots.
\end{align*}
and indeed the Laurent series' terms for the non-positive exponents match. Thus, if you subtract them, it becomes a normal power series with constant term $0$, meaning that the difference is actually a holomorphic function on a neighbourhood of $0$ if you fill in the point at $0$. Due to $\mathcal L$-periodicity, this also means it's holomorphic in a neighbourhood around each lattice point and thus is holomorphic on the entirety of $\mathbb C$ after filling in the lattice points. Furthermore, on a single closed parallelogram of the lattice, the function is bounded due to continuity, and by periodicity it is bounded on the entirety of $\mathbb C$, so now by Liouville's theorem, the function is constant, and due to the $0$ constant term it is $0$. Therefore the differential equation holds.

\end{proof}

Do you notice anything interesting about this differential equation?
\newline\newline
It is an elliptic curve!
\newline\newline
That is $(\wp(z),\wp'(z))$ is a point on the elliptic curve $y^2=4x^3+60G_4(\mathcal L)x-140G_6(\mathcal L)$. However, at this point we do not yet know that the curve described by that equation is not singular, so we cannot technically call it an elliptic curve yet. This will be answered in the next theorem:

\begin{thm} Let $\mathcal{L}$ be a lattice and $\C/\mathcal{L}$ the quotient as Abelian groups.
The equation

$$y^2=4x^3-60G_4(\mathcal{L})x-140G_6(\mathcal{L})$$
defines an elliptic curve over $\C$ (that is, it is non-singular) written $E_\C$.

Furthermore, there is an isomorphism of Abelian groups
$$\C/\mathcal{L}\to E_{\C}$$
given by
$$z+\mathcal{L}\mapsto (\wp(z),\wp'(z))$$
for $z\notin \mathcal{L}$ and $0\mapsto O$, the point at infinity.
\end{thm}

\begin{proof} Proof checklist: Bijective, Group Homomorphism.
\newline\newline
First we prove that the map creates a \textbf{bijection} between the Abelian groups. 
\newline\newline
Let $x\in\C$ and let $f(z)=\wp(z)-x$. We aim to show that $x$ is in the range of $\wp$ by showing $0$ is in the range of $f$ and with further examination of the zeros of $f$, conclude there is a bijection $z\mapsto (x,y)$.
\newline\newline
Due to our Laurent series expansion of $\wp$, we know that $f$ has a pole of multiplicity $2$ at $0$, so we can define $g(z)=z^2f(z)$ so then $g$ can be extended to be holomorphic at $0$. Now take a closed parallelogram of the lattice and shift it so that it is centered at $0$, that is,  let $\omega_1,\omega_2$ generate $\mathcal{L}$ and take the parallelogram with vertices $\pm\frac{\omega_1+\omega_2}{2}$ and $\pm\frac{\omega_1-\omega_2}{2}$ (including the boundary and interior points). So now the only lattice point on the parallelogram is $0$. Therefore $g$ is defined on the whole of the parallelogram. Call this parallelogram $\mathcal{P}$.
\newline\newline
We show that $f$ has only finitely many zeros in $\mathcal{P}$.
Assume for the sake of contradiction that $f$ had infinitely many zeroes in $\mathcal{P}$. Then $g$ has infinitely many zeroes in $\mathcal{P}$, and we can form an infinite sequence of distinct zeroes in $\mathcal{P}$, so by the Bolzano-Weierstrass theorem, there is a convergent subsequence, so this sequence has a limit point inside the closed parallelogram. Then $g=0$ so $f=0$ in the parallelogram (Theorem \ref{discretezero}), so due to periodicity then $f=0$ everywhere, which is clearly not true, looking at the Laurent series of $f$ and we have reached a contradiction.
\newline\newline
Now since there are a finite number of zeroes on any period of the lattice, shift the parallelogram so that it has no zeroes nor lattice points on its border. Call this shifted parallelogram $\mathcal{P}_1$ and let this parallelogram's counterclockwise boundary be $\partial \mathcal{P}_1$. We show that
$$\int_{\partial \mathcal{P}_1} \frac{f'(z)}{f(z)} dz=0$$
and use the argument principle (Theorem \ref{argprinciple}), to count the number of zeros of $f$.
\newline\newline
Due to $\mathcal L$-periodicity of $f'/f$, the opposite sides of the parallelogram being in opposite directions cancel out in the integral, so this evaluates to $0$. Thus using the argument principle, $f$ has the same number of zeroes on the parallelogram as poles counting multiplicity, and there are $2$ poles (the double pole at the lattice point), so it has $2$ zeroes counting multiplicity. Therefore, $\wp(z)=x$ for either two distinct values of $z$ or a single value with multiplicity two, up to translations of $\mathcal{L}$ (we will say modulo $\mathcal{L}$).
\newline\newline
We now show that for all but three $z$ values modulo $\mathcal{L}$ (keep in mind that a cubic has three roots), there are exactly two distinct values of $z$ that produce distinct $y$ values (ie non-zero $y^2$) up to translation by $\mathcal{L}$, otherwise, $y^2=0$.
\newline\newline
For any $z\in\mathbb C\setminus\mathcal L$ not equal to $\frac{\omega_1}2$, $\frac{\omega_2}2$ or $\frac{\omega_1+\omega_2}2$ $\mod\mathcal{L}$, $z\not\equiv-z\mod\mathcal L$, so since $\wp$ is even, $z,-z$ are those values mod $\mathcal{L}$ on which $\wp(-z)$ equals $\wp(z)$. Hence in this case, there are exactly two distinct $z$ that output the same $x$. We show these output non-zero $y$-values which are negatives of each other. To see they are negatives of each other, we recall that $\wp'(-z)=-\wp'(z)$ is odd. To show these $\pm y$ are non-zero, it suffices to show $\wp(\frac{\omega_1}2)$, $\wp(\frac{\omega_2}2)$ and $\wp(\frac{\omega_1+\omega_2}2)$ are distinct zeros different from $\wp(z)$ since a cubic equation in $x$ can have at most 3 zeros.
\newline\newline
Since $\wp'$ is odd, that means that $$\wp'(\frac{\omega_1}2)=-\wp'(-\frac{\omega_1}2)=-\wp'(-\frac{\omega_1}2+\omega_1)=-\wp'(\frac{\omega_1}2),$$
so 
$$\wp'(\frac{\omega_1}2)=0.$$ Similarly, $$\wp'(\frac{\omega_2}2)=0\qquad \text{and}\qquad \wp'(\frac{\omega_1+\omega_2}2)=0.$$ 
We now explain why these are different to the other $\wp(z)$ using multiplicity of zeros of $f$ established above, namely, we show these each have multiplicity 2. If $\alpha\in\{\frac{\omega_1}2,\frac{\omega_2}2,\frac{\omega_1+\omega_2}2\}$, then $f(\alpha)=0$ for $f$ with $x=\wp(\alpha)$ and $\alpha$ has multiplicity at least two since $f'(\alpha)=\wp'(\alpha)=0$. This means that $\alpha$ has multiplicity exactly 2 using the multiplicity of zeros of $f$ already established.
\newline\newline



This also shows that $\wp(\frac{\omega_1}2),\wp(\frac{\omega_2}2),\wp(\frac{\omega_1+\omega_2}2)$ are distinct roots of the polynomial $4x^3-60G_4(\mathcal L)x-140G_6(\mathcal L)$. Hence the curve $y^2=4x^3-60G_4(\mathcal L)x-140G_6(\mathcal L)$ is non-singular. 
\newline\newline
In summary, $\frac{\omega_1}2,\frac{\omega_2}2,\frac{\omega_1+\omega_2}2$ are taken to the three distinct $x$-values which are roots of the cubic, at which the elliptic curve only assumes one $y$-value of zero, and all other points modulo $\mathcal{L}$ are taken to $x$-values at which the elliptic curve assumes two $y$-values which are negatives of each other, and each $y$-value is taken once due to $\wp'$ being odd. Finally if we extend the map to take the lattice points to the point at infinity on the elliptic curve, this map bijects $\mathbb C/\mathcal L$ with the elliptic curve.
\newline\newline
Now we need to show that this map is a \textbf{group homomorphism}, in other words, it is additive. Say that we have two points, $z_1$ and $z_2$, and we want to show that $z_1+z_2$ gets taken to the sum of the corresponding points on the elliptic curve. First we deal with the case where one of them is a lattice point. If $z_1\in\mathcal L$ then $z_1+z_2\equiv z_2\mod\mathcal L$, so $z_1+z_2$ gets taken to the same point as $z_2$ as expected. Otherwise, $z_1,z_2$ get taken to actual points in $\mathbb C^2$ by $\wp$ and $\wp'$. If $(\wp(z_1),\wp'(z_1))$ and $(\wp(z_2),\wp'(z_2))$ are inverses, then $\wp(z_1)=\wp(z_2)$ and $\wp'(z_1)=-\wp'(z_2)\neq\infty$, it can be checked from the bijection above that $z_1+z_2\equiv0\mod\mathcal L$, and hence $z_1+z_2$ gets sent to the point at infinity as expected. Otherwise, the line joining the points $(\wp(z_1),\wp'(z_1))$ and $(\wp(z_2),\wp'(z_2))$ isn't "vertical" since otherwise it would pass through the point at infinity (Note: this includes the tangent case as well), so it can be expressed as $y=\ell(x)$ where $\ell$ is a linear polynomial. Define
$$h(z)=\ell(\wp(z))-\wp'(z)$$
which then has roots corresponding exactly to the intersections of the line and the elliptic curve. Define 
$$p(x)=4x^3+60G_4(\mathcal L)x-140G_6(\mathcal L).$$ Given such a root $w$ of $h$, then we know that $\wp'(w)=\ell(\wp(w))$ since the point is on the line. The differential equation for $\wp$ can be expressed as $$\wp'(z)^2=p(\wp(z)),$$
so then differentiating this we get $2\wp'(z)\wp''(z)=\wp'(z)p'(\wp(z))$. Therefore for all $z$ where $\wp'(z)\neq0$, 
$$2\wp''(z)=p'(\wp(z)),$$
and since the zeroes of a non-zero holomorphic function ($\wp'$) are isolated, by continuity we get this identity for all $z\in\mathbb C\setminus\mathcal L$. Therefore, given that $w$ is a root of $h$, then $\wp'(w)^2=\ell(\wp(w))^2=p(\wp(w))$, so $\wp(w)$ is a root of $\ell^2-p$, then
\begin{align*}
&\text{$w$ is at least a double root of $h$}\\
   \iff&h'(w)&=0\\ 
   \iff& \wp'(w)\ell'(\wp(w))-\wp''(w)&=0\\ 
   \iff&2\ell(\wp(w))\ell'(\wp(w))-p'(\wp(w))&=0\\
   \iff&(\ell^2-p)'(\wp(w))&=0\\
   \iff&\text{$\wp(w)$ is at least a double root of $\ell^2-p$}
\end{align*}
Now assume further that the above equivalence is true for $w$. Then we can show that $\ell(\wp(w))\neq0$: assume for the sake of contradiction that $\ell(\wp(w))=0$. then $\wp'(w)=\ell(\wp(w))=0$ and $\ell(\wp(w))^2=p(\wp(w))=0$, and also from the above equivalence (second $\iff$), we have $\wp'(w)\ell'(\wp(w))-\wp''(w)=0$, then $\wp''(w)=\wp'(w)\ell'(\wp(w))=0$. Finally, $p'(\wp(w))=2\wp''(w)=0$. Now we have $p(\wp(w))=0$ and $p'(\wp(w))=0$, so this is a contradiction since that means $p$ has a double root at $\wp(w)$, but we already established above that $p$ has $3$ distinct roots. 

Therefore since $\ell(\wp(w))\neq0$, from $2\ell(\wp(w))\ell'(\wp(w))-p'(\wp(w))=0$, we can conclude that 
$$\ell'(\wp(w))=0\iff p'(\wp(w))=0.$$
Furthermore, differentiating the identity $2\wp''(z)=p'(\wp(z))$ we get $$2\wp'''(z)=\wp'(z)p''(\wp(z)).$$
Also since $\ell$ is linear, $\ell''=0$. Hence
\begin{align*}
    &\text{$w$ is at least a triple root of $h$}\\
    \iff &h''(w)&=0\\
    \iff &\wp''(w)\ell'(\wp(w))+\wp'(w)^2\ell''(\wp(w))-\wp'''(w)&=0\\
    \iff &p'(\wp(w))\ell'(\wp(w))-\wp'(w)p''(\wp(w))&=0\\
    \iff &p'(\wp(w))\ell'(\wp(w))-\ell(\wp(w))p''(\wp(w))&=0
\end{align*}

We also have that
\begin{align*}
    &\text{$\wp(w)$ is at least a triple root of $\ell^2-p$}\\
    \iff &(\ell^2-p)''(\wp(w))&=0\\
    \iff &(2{\ell'}^2+2\ell\ell''-p'')(\wp(w))&=0\\
    \iff &(2{\ell'}^2-p'')(\wp(w))&=0\\
    \iff&2\ell'(\wp(w))^2-p''(\wp(w))&=0
\end{align*}
In the case that $\ell'(\wp(w))=p'(\wp(w))=0$, then 
\begin{align*}
&2\ell'(\wp(w))^2-p''(\wp(w))&=0\\
\iff&p''(\wp(w))&=0\\
\iff&\ell(\wp(w))p''(\wp(w))&=0\\
\iff&p'(\wp(w))\ell'(\wp(w))-\ell(\wp(w))p''(\wp(w))&=0
\end{align*}
In the case that $\ell'(\wp(w))\neq0$ and $p'(\wp(w))\neq0$, then
\begin{align*}
&2\ell'(\wp(w))^2-p''(\wp(w))&=0\\
\iff&2\ell'(\wp(w))^2p'(\wp(w))-p''(\wp(w))p'(\wp(w))&=0\\
\iff&\ell'(\wp(w))^2p'(\wp(w))-p''(\wp(w))\ell(\wp(w))\ell'(\wp(w))&=0\\
\iff&\ell'(\wp(w))p'(\wp(w))-p''(\wp(w))\ell(\wp(w))&=0\\
\end{align*}
so either way, 
$$w\text{ is at least a triple root of }h \iff \wp(w)\text{ is at least a triple root of }\ell^2-p$$
$\ell^2-p$ has degree $3$ so roots can be at most triple roots. As for $h$, looking at the definition, $h$ has poles of multiplicity $3$ at each lattice point (that is from the $-\wp'$ term), so applying a similar argument that we did to $\wp$ above, except $g(z)=z^3h(z)$ this time, we get that $h$ has $3$ zeroes counting multiplicity on a period of the lattice. Hence roots of $h$ also can have multiplicity at most $3$. Therefore,  the multiplicity of the root $w$ of $h$ is equivalent to the multiplicity of the root $\wp(w)$ of $\ell^2-p$ where $\wp(w)$ is the $x$-coordinate of the intersection, i.e. the multiplicity of the intersection of the line with the elliptic curve.
\newline\newline
Now this means that mod $\mathcal L$, $z_1$ and $z_2$ are two of the roots of $h$ counting multiplicity, and the third root $z_3$ counting multiplicity corresponds to the third intersection of the line with the elliptic curve counting multiplicity. Now consider once again the anticlockwise border of a parallelogram of the lattice $\gamma=\gamma_1+\gamma_2+\gamma_3+\gamma_4$ where $\gamma_1,\gamma_2,\gamma_3,\gamma_4$ are its sides, $\gamma_3$ is $\gamma_1$ translated by $\omega_1$ and reversed, $\gamma_4$ is $\gamma_2$ translated by $\omega_2$ and reversed, and $\gamma$ doesn't pass through any poles or zeroes of $h$. Then consider
$$\int_\gamma \frac{zh'(z)}{h(z)}dz$$
which turns out from Theorem \ref{extended-argument-principle} to equal $2\pi i$ times the difference between the sum of the roots $z_1+z_2+z_3$ and the sum of poles of $h$ inside the parallelogram, counting multiplicity. We now show this integral evaluates to $2\pi i$ times an element of $\mathcal{L}$:
{\tiny
\begin{align*}
\int_\gamma \frac{zh'(z)}{h(z)}dz&=\int_{\gamma_1} \frac{zh'(z)}{h(z)}dz+\int_{\gamma_2} \frac{zh'(z)}{h(z)}dz+\int_{\gamma_3} \frac{zh'(z)}{h(z)}dz+\int_{\gamma_4} \frac{zh'(z)}{h(z)}dz\\
&=\int_{\gamma_1} \frac{zh'(z)}{h(z)}dz+\int_{\gamma_2} \frac{zh'(z)}{h(z)}dz-\int_{\gamma_1} \frac{(z+\omega_1)h'(z+\omega_1)}{h(z+\omega_1)}dz-\int_{\gamma_2} \frac{(z+\omega_2)h'(z+\omega_2)}{h(z+\omega_2)}dz\\
&=\int_{\gamma_1} \frac{zh'(z)}{h(z)}dz+\int_{\gamma_2} \frac{zh'(z)}{h(z)}dz-\int_{\gamma_1} \frac{(z+\omega_1)h'(z)}{h(z)}dz-\int_{\gamma_2} \frac{(z+\omega_2)h'(z)}{h(z)}dz\\
&=-\int_{\gamma_1} \frac{\omega_1h'(z)}{h(z)}dz-\int_{\gamma_2} \frac{\omega_2h'(z)}{h(z)}dz\\
&=-\omega_1\int_{\gamma_1} \frac{h'(z)}{h(z)}dz-\omega_2\int_{\gamma_2} \frac{h'(z)}{h(z)}dz.
\end{align*}
}
From here it suffices to show the remaining integrals are integers. These are two integrals of the logarithmic derivative of $h$ starting and ending at points which, due to the $\mathcal L$-periodicity of $h$, have equal values of $h$. Therefore each integral evaluates to an integer multiple of $2\pi i$, making the above evaluate to $2\pi i$ times a lattice point. Also, the only pole inside the parallelogram is a triple pole at a lattice point. Therefore the sum of the roots of $h$ in the parallelogram is itself actually part of the lattice, i.e. $z_1+z_2+z_3\in\mathcal L$ and hence $z_1+z_2+z_3\equiv0\mod\mathcal L$. Then since $z_3\equiv-z_1-z_2\mod\mathcal L$ and maps to the third intersection of the line with the elliptic curve, i.e. the negative of the sum of the other two points on the elliptic curve, then $-z_3\equiv z_1+z_2\mod\mathcal L$ maps to the actual sum of the other two points on the elliptic curve. Hence the map is a group isomorphism.
\end{proof}

Now the proof in the previous section applies to show the m torsion points are a direct sum of two cyclic groups. One thing we are missing here is an argument that every elliptic curve is covered here. That is, how do we know we get every cubic in $x$ on the right hand side of the equation.

\subsection{Level $N$ Modular Forms}
Keywords: Principal congruence subgroups, $\Gamma_0(N)$

We now complete the definition of modular forms
\begin{definition} [Principal congruence subgroup]
    The principal congruence subgroup of level $N$, denoted $\Gamma(N)$, is the subgroup of $\SLZ$ formed by those matrices whose entries are equivalent to those of the identity matrix $\mod N$.
\end{definition}

\begin{proof}
    Define a function $f:\SLZ\rightarrow \operatorname{SL}_2(\mathbb{Z}/N\mathbb Z)$ where you take each of the entries of the matrix $\mod N$. This is clearly a group homomorphism as the operations used in computing the product of matrices are addition and multiplication, which are compatible with taking $\mod N$, meaning $f(a)f(b)=f(ab)$. Furthermore, it is clear that $\Gamma(N)$ is the kernel of $f$. Thus $\Gamma(N)$ forms a group, in particular a normal subgroup of $\SLZ$.
\end{proof}

\begin{lemma}
    For any principal congruence subgroup $\Gamma(N)$, $\SLZ/\Gamma(N)$ is finite.
\end{lemma}
\begin{proof}
    With the same homomorphism as above $f:\SLZ\rightarrow \operatorname{SL}_2(\mathbb{Z}/N\mathbb Z)$, by the first isomorphism theorem, $\SLZ/\Gamma(N)$ is isomorphic to the image of $f$. However, evidently $\operatorname{SL}_2(\mathbb{Z}/N\mathbb Z)$ is finite since there are a finite number of matrices with entries $\mod N$, so the image of $f$, being a subset of that, is finite. Therefore $\SLZ/\Gamma(N)$ is finite.
\end{proof}

\begin{definition} [Congruence subgroup]
    A congruence subgroup is a subgroup of $\SLZ$ that contains a principal congruence subgroup. Then the smallest $N$ for which $\Gamma(N)$ is a subgroup of this congruence subgroup is called its "level".
\end{definition}

An example of a congruence subgroup which is used in the proof of Fermat's Last Theorem is
\begin{definition}
    For a positive integer $N$, $\Gamma_0(N)$ is the congruence subgroup of level $N$ formed by matrices in $\SLZ$ whose bottom left entry is $0\mod N$.
\end{definition}
\begin{proof}
    The product of two upper triangular matrices (square matrices for which all entries below the main diagonal is zero) is upper triangular. In the case of $2\times 2$ matrices in $\operatorname{SL}_2(\mathbb Z/N\mathbb Z)$, using the formula for inverses, then the inverse of an upper triangular matrix is also upper triangular. The identity matrix is upper triangular. Hence the set of upper triangular matrices in $\operatorname{SL}_2(\mathbb{Z}/N\mathbb Z)$ which we will call $T$ form a subgroup. It is then clear that $\Gamma_0(N)$ is the pre-image of $T$ under the map $f:\SLZ\rightarrow \operatorname{SL}_2(\mathbb{Z}/N\mathbb Z)$ that takes the entries of the matrix $\mod N$. Hence $\Gamma_0(N)$ forms a subgroup of $\SLZ$. Also, ${I}\subset T$ in $\operatorname{SL}_2(\mathbb Z/N\mathbb Z)$, so taking preimages of each, $\Gamma(N)\subset\Gamma_0(N)$, so $\Gamma_0(N)$ is a congruence subgroup of level $N$.
\end{proof}

\begin{lemma}
    For any congruence subgroup $\Gamma$, $\SLZ/\Gamma$ is finite.
\end{lemma}
\begin{proof}
    Since $\Gamma$ is a congruence subgroup, then there is some $N$ where $\Gamma(N)\subset\Gamma\subset\SLZ$. Then since $\SLZ/\Gamma(N)$ is finite, $\SLZ/\Gamma$ is finite, as it is a less fine partitioning of $\SLZ$.
\end{proof}

Now we are better equipped to define what a modular form is.

\begin{definition} [Weakly Modular Form]

A holomorphic function $f:\mathcal{H}\rightarrow\mathbb{C}$ is a weakly modular form of weight $k$ with respect to congruence subgroup $\Gamma$ if
$$f\left(\Mabcd z\right)=(cz+d)^kf(z)\text{ for all }z\in \mathbb{C}, \Mabcd\in\Gamma$$
\end{definition}

When the equation in the above definition holds, we say that "$f$ is weight-$k$ invariant under $\Gamma$".


\subsubsection{Periodic}

Here $\Gamma$ is a congruence subgroup so it contains some $\Gamma(N)$, and therefore contains $\begin{pmatrix} 1&N\\0&1 \end{pmatrix}$. Thus define $h$ to be the smallest positive integer such that $\begin{pmatrix} 1&h\\0&1 \end{pmatrix}\in\Gamma$. If we take a weakly modular form $f$ with respect to $\Gamma$ and sub into the functional equation $\Mabcd = \begin{pmatrix} 1&h\\0&1 \end{pmatrix}$ where $N$ is the level of $\Gamma$, then we get $f(z+h)=f(z)$, so $f$ is actually periodic with period $h$. This actually means that, by making the substitution $q=e^{2\pi iz/h}$, and letting $D$ be the unit disc of the complex plane, there exists a unique function $g:D\setminus\{0\}\rightarrow\mathbb{C}$ such that $g(q)=f(z)$. $g$ is then also holomorphic on $D\setminus\{0\}$.

\subsubsection{Holomorphic at infinity}

Apparently it is very common to study this $\hat{g}$ with the $q$ sub, such as taking the power series of $\hat{g}$ at $0$. Also, as long as $g$ is bounded on an open neighbourhood of $0$, then by Riemann's Theorem on removable singularities, it is indeed possible to fill in $g$ at $0$. Therefore as long as $f(z)$ is bounded for $\operatorname{Im}(z)>a$ for some $a$, then $f$ being a weakly modular form implies it is a modular form.
\newline\newline
Because the disc $D$ is missing $0$, $g$ may not necessarily be nicely behaved around $0$, based on the behaviour of $f$ as the imaginary part approaches infinity. We say that a weakly modular form $f$ is "holomorphic at infinity" if its corresponding $g$ can be continued to $0$ such that $g$ ends up being holomorphic at $0$.
\newline\newline
Remark about holomorphicness of function of $q$: A function of $q$ can be expanded as a Laurent series, which is a power series where the powers of $q$ are allowed to be negative. There are general situations for functions having a Laurent series, while also, in this case, it amounts to a periodic function having a Fourier series. One practical way to defining holomorphic here is just to say that there are no negative powers of $q$ in this expansion, so that taking $q=0$ should be a valid move on this expression once the point of infinity is added.
\newline\newline
An interesting fact is that by Riemann's theorem on removable singularities, if $g$ is bounded on some neighbourhood of $0$, then it is indeed possible to holomorphically continue $g$ to $0$. The converse is true by continuity. Therefore, $f$ being holomorphic at infinity is equivalent to $f$ being bounded past some value for its imaginary part.
\newline\newline
If we adjoin $\infty$ and $\mathbb{Q}$ to a congruence subgroup $\Gamma$, call the equivalence class of points in $\{\infty\}\cup \mathbb{Q}$ a \emph{cusp} of $\Gamma$. 

\begin{lemma}
    $\{\infty\}\cup \mathbb{Q}$ is closed under $SL_2(\mathbb{Z)}$
\end{lemma}

\begin{proof}
   First, as $z \to \infty$, $$\frac{az + b}{cz + d} \to \frac{a}{c}\in\Q.$$ 
   If $z$ is a rational $\neq -d/c$, then $\frac{az + b}{cz + d}\in\Q$. Finally, when $z = -d/c$ where $c\neq0$, the denominator goes to $0$ while the $a(-d/c) +b= \frac{bc-ad}{c}=-1/c \in \mathbb{Q}$ (as $c\neq0$), taking the value to $\infty$. Thus,  $\{\infty\}\cup \mathbb{Q}$ is closed under $SL_2(\mathbb{Z)}$. 
\end{proof}

\begin{theorem}
    If $\Gamma$ is a congruence subgroup of level $N$, then $\alpha^{-1}\Gamma\alpha$ is also a congruence subgroup with level $N$, where $\alpha \in SL_2(\mathbb{Z})$.
\end{theorem}

\begin{proof}
  We will first show that $\alpha^{-1}\Gamma\alpha$ is a subgroup. If we choose the identity element of $\Gamma$ to be $I$, $\alpha^{-1}I\alpha= \alpha^{-1}\alpha$ which becomes its own identity element $I'$, i.e. $\alpha^{-1}\Gamma\alpha$ contains an identity element. Next, we will show that every element has an inverse. Suppose we have $\gamma$ such that $\gamma \in \Gamma$. Then, $$(\alpha^{-1}\gamma^{-1}\alpha) \cdot (\alpha^{-1}\gamma\alpha) = (\alpha^{-1}\gamma^{-1}) \cdot (\gamma\alpha)=(\alpha^{-1}) \cdot (\alpha)$$ giving the identity element. As $\gamma^{-1} \in \Gamma$, $(\alpha^{-1}\gamma^{-1}\alpha)$ is the inverse of $(\alpha^{-1}\gamma\alpha)$. Finally, for $\gamma \in \Gamma(N)\subset\Gamma$ we have $\gamma \equiv I$ (mod $N$) for each element in the matrix entry. Then, $\alpha^{-1}\gamma\alpha\in\alpha^{-1}\Gamma\alpha$, and $\alpha^{-1}\gamma\alpha \equiv \alpha^{-1}I\alpha \equiv \alpha^{-1}\alpha \equiv I$ mod $N$. Hence $\Gamma(N)\subset\alpha^{-1}\Gamma\alpha$, so $\alpha^{-1}\Gamma\alpha$ is a congruence subgroup with level $N$.
\end{proof}

\begin{definition}
   Define $j(\alpha,z)=cz+d $ where $\alpha= \Mabcd$ which can be any $2\times2$ matrix.
\end{definition}

\begin{definition}
For a function $f:\mathcal{H}\rightarrow\mathbb{C}$, an integer $k$, and a $2\times2$ matrix $\alpha$ with complex entries, define the function $f[\alpha]_k$ such that
$$f[\alpha]_k(z)=j(\alpha,z)^{-k}f(\alpha(z))$$
for all $z\in\mathcal{H}$.
\end{definition}
    
\begin{lemma} \label{j func result}
$j(\beta,\gamma(z))j(\gamma,z)=j(\beta\gamma,z)$.
\end{lemma}

\begin{proof}
   Let $\beta=\Mabcd$ and $\gamma=\begin{pmatrix}
e & f \\
g & h
\end{pmatrix}$.

\begin{align*}
\text{LHS} &= (c\gamma(z) + d)(gz + h) \\
          &= c\gamma(z)(gz + h) + dgz + dh \\
          &= \frac{c(ez + f)}{gz + h}(gz + h) + dgz + dh \\
          &= cez + dgz + cf + dh.
\end{align*}
Evaluating $\beta\gamma$, \\ \[\begin{pmatrix}
a & b \\
c & d
\end{pmatrix}
\begin{pmatrix}
e & f \\
g & h
\end{pmatrix}
=
\begin{pmatrix}
ae + bg & af + bh \\
ce + dg & cf + dh
\end{pmatrix}.
\]

Hence, $LHS=RHS= cez+dgz+cf+dh$.

\end{proof}
\begin{definition}
    For a congruence subgroup $\Gamma$ in $\SL_2\mathbb(\Z)$ and some integer $k$, we define a function $f$ to be a \emph{modular form of weight $k$ with respect to $\Gamma$} if 
    \begin{enumerate}
        \item $f$ is weakly modular of weight $k$ with respect to $\Gamma$
        \item $f[\alpha]_k$ is holomorphic at $\infty$ for all $\alpha \in \SL_2(\mathbb{Z})$
    \end{enumerate}
    \end{definition}
    
But what does holomorphic at $\infty$ mean? First, it serves to show that 

\begin{prop}
    $f[\alpha]_k$ is weakly modular under $\alpha^{-1}\Gamma\alpha$. 

\end{prop}
\begin{proof}
    We know from our definition of a weakly modular form that $$f(\gamma\alpha(z))=j(\gamma,\alpha(z))^kf(\alpha(z))$$  Then by Lemma \ref{j func result}, 
    $$f((\gamma\alpha)(z))= j(\gamma\alpha,z)^kj(\alpha,z)^{-k}f(\alpha(z))$$ and again, $$j(\alpha,(\alpha^{-1}\gamma\alpha)(z))^{-k}f((\gamma\alpha) (z))=j(\alpha^{-1}\gamma\alpha,z)^kj(\alpha,z)^{-k}f(\alpha(z)).$$
    \\ This implies $$f[\alpha]_k((\alpha^{-1}\gamma\alpha)(z))=j(\alpha^{-1}\gamma\alpha,z)^kf[\alpha]_k(z)$$ as desired.
\end{proof}
    Thus, as $f[\alpha]_k$ is holomorphic on $\mathcal{H}$ and weakly modular on $\alpha^{-1}\Gamma\alpha$, a congruence subgroup of $\SL_2\mathbb(\Z)$, its holomorphy at $\infty$ is justified.

\subsection{Hecke Operators}
Keywords: Newform, Eigenform, Hecke Operator. \newline

Modular forms of a fixed weight and congruence subgroup form a vector space over $\mathbb{C}$, as  adding two modular forms gives another, and scaling by a complex number preserves modularity.
\newline\newline
It would be nicer to work with modular forms if there were operations you could do to them to get other modular forms. One particular operator of interest is the Hecke operator $T_n$ defined for each positive integer $n$. We will define it after introducing some prerequisites.

\begin{definition}
    $M_n$ denotes the set of $2\times2$ matrices with integer entries and determinant $n$.
\end{definition}

If you multiply a matrix from $M_n$ with a matrix in $\SLZ$ on either the left or right, the resulting matrix will still have integer entries and determinant $n$, so it's still in $M_n$. You can thus verify that left multiplication by things in $\SLZ$ defines a group action on $M_n$. Then we can write the set of orbits under this group action as $\SLZ \backslash M_n$. A similar kind of notation can be used for any subgroup of $\SLZ$.

\begin{lemma}
    For any positive integer $n$, the set of orbits $\SLZ \backslash M_n$ is finite.
\end{lemma}

\begin{proof}
    Let $m\in M_n$ be a matrix $\Mabcd$. Left multiplying it by $\begin{pmatrix}
        1 & k\\
        0 & 1
    \end{pmatrix}\in \SLZ$
    does the row operation of adding $k$ times the bottom row to the top row ($k\in\mathbb{Z}$). Similarly, left multiplying it by $\begin{pmatrix}
        1 & 0\\
        k & 1
    \end{pmatrix}$
    does the row operation of adding $k$ times the top row to the bottom row. Using these two operations you can perform the Euclidean algorithm on the entries $a$ and $c$ to make $c=0$. Then in the resulting matrix, express $b=qd+r$ via division so we can perform a row operation to make the new $b$ equal to $r$, and thus $0\leq b<d$, while preserving $c=0$. In the final matrix, since $c=0$, the determinant is $ad=n$, so there are a finite number of possible values for $a,d$ (namely the divisors of $n$). For each of these pairs of values, there is also a finite number of possible values for $b$, since the size of $b$ is bounded by $d$. Therefore the matrix we end up with is among a finite set of possible matrices. Therefore each matrix in $M_n$ shares an orbit with one of a finite set of matrices so there are a finite number of orbits.
\end{proof}

\begin{lemma}
    For any positive integer $n$ and congruence subgroup $\Gamma$, the set of orbits $\Gamma \backslash M_n$ is finite.
\end{lemma}
\begin{proof}
    Because the set of orbits $\SLZ \backslash M_n$ is finite, and $\Gamma$ has finite index in $\SLZ$ by being a congruence subgroup, then the set of orbits $\Gamma \backslash M_n$ is finite.
\end{proof}

The interesting thing about matrices that are in the same orbit in this case is
\begin{lemma}
    If $f$ is a modular form of weight $k$ over congruence subgroup $\Gamma$, and $\mu_1,\mu_2\in M_n$ are in the same orbit in $\Gamma \backslash M_n$, then
    $$f[\mu_1]_k=f[\mu_2]_k.$$
\end{lemma}
\begin{proof}
    If $\mu_1$ and $\mu_2$ are in the same orbit in $\Gamma \backslash M_n$, then there exists $\gamma\in\Gamma$ such that $\mu_2=\gamma\mu_1$. Since $f$ is a modular form of weight $k$ over $\Gamma$ we know that by substituting $\mu_1(\tau)$ into the invariance condition we get
    \begin{align*}
    f(\gamma(\mu_1(\tau)))&=j(\gamma,\mu_1(\tau))^kf(\mu_1(\tau))\\
    j(\mu_1,\tau)^{-k}f(\mu_1(\tau))&=(j(\gamma,\mu_1(\tau))j(\mu_1,\tau))^{-k}f(\gamma(\mu_1(\tau)))
    \\
    f[\mu_1]_k(\tau)&=(j(\gamma,\mu_1(\tau))j(\mu_1,\tau))^{-k}f(\gamma\mu_1(\tau))
    \end{align*}
    Using Lemma \ref{j func result}, this equals
$$j(\gamma\mu_1,\tau)^{-k}f(\gamma\mu_1(\tau))=j(\mu_2,\tau)^{-k}f(\mu_2(\tau))=f[\mu_2]_k(\tau)$$
and we are done.
\end{proof}

Hence, $f[\mu]_k$ is the same for all $\mu$ in a particular orbit. Now we can define the Hecke operator as this sum:

\begin{definition}[Hecke Operator]
    Consider the modular forms $f$ of weight $k$ over $\Gamma$. For each positive integer $n$, there is a Hecke operator $T_n$ defined on these modular forms in the following way: Let $\mu_1,\mu_2,\dots,\mu_\ell$  be a set of representatives for each of the orbits in $\Gamma\backslash M_n$. Then define the Hecke operator $T_n$ as
    $$T_n(f)=n^{k-1}\sum_{i=1}^l f[\mu_i]_k.$$
\end{definition}

The sum is finite since there are a finite number of orbits, and from the above lemma, this definition is independent of the choice of representative from each orbit.
\newline\newline
It is easy to see that this operator is linear. However, does it always output a modular form? We now verify that $T_n$ indeed outputs another modular form of weight $k$ over $\Gamma$.
\newline\newline
Firstly we have to check the invariance condition. Consider the map $\mu\mapsto\mu\gamma$ where $\mu\in M_n, \gamma\in\Gamma$. Clearly the map is $M_n\rightarrow M_n$. If this map takes matrices $\mu_i$ and $\mu_j$ to members of the same orbit, then there exists $\gamma_1\in\Gamma$ such that $\mu_i\gamma=\gamma_1\mu_j\gamma$, so $\mu_i=\gamma_1\mu_j$ and $\mu_i,\mu_j$ are also in the same orbit. Hence this map takes the orbit representatives to distinct orbits, and with there being a finite number of orbits, this map also takes the orbit representatives to all orbits, and each orbit is covered exactly once. Hence for $\gamma\in\Gamma$, $\mu_1\gamma,\mu_2\gamma,\dots,\mu_\ell\gamma$ is also a valid set of orbit representatives, so considering that the choice of orbit representatives doesn't matter, we can actually express the sum in the definition of $T_n$ in two different ways:
$$T_n(f)=n^{k-1}\sum_if[\mu_i]_k=n^{k-1}\sum_if[\mu_i\gamma]_k$$
which evaluates to
$$T_n(f)(\tau)=n^{k-1}\sum_ij(\mu_i\gamma,\tau)^{-k}f(\mu_i\gamma(\tau))$$
Using Lemma \ref{j func result}, this means that
\begin{align*}
T_n(f)(\tau)&=n^{k-1}\sum_ij(\mu_i,\gamma(\tau))^{-k}j(\gamma,\tau)^{-k}f(\mu_i\gamma(\tau))\\
&=n^{k-1}j(\gamma,\tau)^{-k}\sum_ij(\mu_i,\gamma(\tau))^{-k}f(\mu_i\gamma(\tau))\\
j(\gamma,\tau)^kT_n(f)(\tau)&=n^{k-1}\sum_ij(\mu_i,\gamma(\tau))^{-k}f(\mu_i(\gamma(\tau)))\\
&=T_n(f)(\gamma(\tau))
\end{align*}
Thus the invariant condition holds for all $\gamma\in\Gamma$ as required. Furthermore, due to the finiteness of the sum in the definition of $T_n$, it is clear that $T_n(f)$ is holomorphic. Hence, $T_n(f)$ is indeed weakly modular with weight $k$ over $\Gamma$.
\newline\newline
Now we have to verify the "holomorphic at infinity" condition. Each individual $f[\mu_i]_k$, by definition of $f$ being a modular form, is holomorphic at infinity with respect to each of their individual periods. Therefore they are bounded for large imaginary part of the input. The sum is finite, so the sum of them ($T_n(f)$) is also bounded for large imaginary part of the input. Then in its $q$-expansion, the singularity at $q=0$ is removable, so indeed $T_n(f)$ is holomorphic at infinity.
\newline\newline
Now that we have defined the Hecke operators, it is also interesting to note that the Taniyama-Shimura-Weil conjecture states that for the modular form $f$ associated to a rational elliptic curve, then for each $n\in\mathbb Z^+$, there is a $\lambda_n\in\mathbb C$ such that $T_n(f)=\lambda_n f$. In linear algebra terms, this says that $f$ is an eigenvector for each of the Hecke operators $T_n$, and the $\lambda_n$ is then the eigenvalue of this eigenvector. Since $f$ is a modular form, the convention is to call $f$ an "eigenform" of the Hecke operators. What is even more amazing is that for each prime number $p$, the eigenvalue of $f$ with respect to $T_p$ is $a_p$, i.e. $T_p(f)=a_pf$.


\subsection{No Non-Zero Weight $2$, Level $2$, $\Gamma_0(2)$ Cusp Forms}
Here we try to adapt the proof of what happens in level $1$ to the case of the congruence subgroup $\Gamma_0(2)$.

\subsubsection{Argument principle modified by $\SL_2(\Z)$}
\begin{lemma}
    If $\alpha\in\SLZ$, then $\alpha'(z)=j(\alpha,z)^{-2}$.
\end{lemma}
\begin{proof}
    Let $\alpha=\Mabcd$, then $\det(\alpha)=ad-bc=1$ so
    $$\alpha'(z)=\left(\frac{az+b}{cz+d}\right)'=\frac{a(cz+d)-c(az+b)}{(cz+d)^2}=\frac{ad-bc}{(cz+d)^2}=\frac 1{(cz+d)^2}$$
    as desired.
\end{proof}

\begin{definition}
    For $\alpha\in\SLZ$, denote by $j'(\alpha)$ the bottom left entry of $\alpha$.
\end{definition}
The reason we called it that is because if you differentiate $j(\alpha,z)$ with respect to $z$, you get $j'(\alpha)$.
\newline\newline
First let's establish a result about what happens to the line integral on holomorphic function when you transform the path using a Mobius transformation.

\begin{definition}
    If $\gamma:[t_1,t_2]\rightarrow\mathbb C$ is a path and $\alpha\in\SLZ$, then denote by $\alpha(\gamma)$ the path formed by applying the Mobius transformation associated with $\alpha$ to $\gamma$, i.e. if $\alpha=\Mabcd$ then $\alpha(\gamma)(t)=\alpha(\gamma(t))=\frac{a\gamma(t)+b}{c\gamma(t)+d}$.
\end{definition}

\begin{lemma}\label{alphaarg}
If you have a holomorphic function $f : \mathcal{H} \rightarrow \mathbb{C}$ and in $\mathcal{H}$ is some piecewise continuously differentiable curve $\gamma : [t_1, t_2] \rightarrow \mathcal{H}$ not passing through any zero of $f$, and you have $\alpha \in \SLZ$, then for any non-negative integer $k$,
$$\int_{\alpha (\gamma)} \frac{f' (z)}{f (z)} \mathrm{d} z=\int_{\gamma} \frac{(f [\alpha]_k)' (z)}{f [\alpha]_k (z)} \mathrm{d} z + \int_{\gamma} \frac{kj' (\alpha)}{j (\alpha, z)} \mathrm{d} z$$
\end{lemma}

\begin{proof}
    First we do the case where $\gamma$ is by itself continuously differentiable. We compute that

\begin{align}
  \int_{\alpha (\gamma)} \frac{f' (z)}{f (z)} \mathrm{d} z
  &=  \int_{t_1}^{t_2} (\alpha (\gamma))' (t)  \frac{f' (\alpha (\gamma (t)))}{f
  (\alpha (\gamma (t)))} \mathrm{d} t \nonumber\\
  &=  \int_{t_1}^{t_2} \alpha' (\gamma (t)) \gamma' (t)  \frac{f' (\alpha (\gamma
  (t)))}{f (\alpha (\gamma (t)))} \mathrm{d} t \nonumber\\
  &= \int_{t_1}^{t_2} j (\alpha, \gamma (t))^{- 2} \gamma' (t)  \frac{f' (\alpha
  (\gamma (t)))}{f (\alpha (\gamma (t)))} \mathrm{d} t \nonumber
\end{align}

We know that since
\[ f [\alpha]_k (z) = j (\alpha, z)^{- k} f (\alpha (z)) \]
then

\begin{align}
  (f [\alpha]_k)' (z) &=  j (\alpha, z)^{- k} \alpha' (z) f' (\alpha (z)) -
  kj' (\alpha) j (\alpha, z)^{- k - 1} f (\alpha (z)) \nonumber\\
  &=  j (\alpha, z)^{- k} j (\alpha, z)^{- 2} f' (\alpha (z)) - kj' (\alpha) j
  (\alpha, z)^{- k - 1} f (\alpha (z)) \nonumber\\
  &=  j (\alpha, z)^{- k - 2} f' (\alpha (z)) - kj' (\alpha) j (\alpha, z)^{-
  k - 1} f (\alpha (z)).\nonumber
\end{align}

So
\[ f' (\alpha (z)) = j (\alpha, z)^{k + 2}  (f [\alpha]_k)' (z) + kj' (\alpha)
   j (\alpha, z) f (\alpha (z)). \]
Now

{\tiny{\begin{align}
  & \int_{\alpha (\gamma)} \frac{f' (z)}{f (z)} \mathrm{d} z \nonumber\\
  = & \int_{t_1}^{t_2} j (\alpha, \gamma (t))^{- 2} \gamma' (t)  \frac{f' (\alpha
  (\gamma (t)))}{f (\alpha (\gamma (t)))} \mathrm{d} t \nonumber\\
  = & \int_{t_1}^{t_2} j (\alpha, \gamma (t))^{- 2} \gamma' (t)  \frac{j (\alpha,
  \gamma (t))^{k + 2}  (f [\alpha]_k)' (\gamma (t)) + kj' (\alpha) j (\alpha,
  \gamma (t)) f (\alpha (\gamma (t)))}{f (\alpha (\gamma (t)))} \mathrm{d} t
  \nonumber\\
  = & \int_{t_1}^{t_2} \gamma' (t)  \frac{j (\alpha, \gamma (t))^{- 2} j (\alpha,
  \gamma (t))^{k + 2}  (f [\alpha]_k)' (\gamma (t))}{f (\alpha (\gamma (t)))}
  \mathrm{d} t + \int_{t_1}^{t_2} \gamma' (t) j (\alpha, \gamma (t))^{- 2}  \frac{kj'
  (\alpha) j (\alpha, \gamma (t)) f (\alpha (\gamma (t)))}{f (\alpha (\gamma
  (t)))} \mathrm{d} t \nonumber\\
  = & \int_{t_1}^{t_2} \gamma' (t)  \frac{(f [\alpha]_k)' (\gamma (t))}{j (\alpha,
  \gamma (t))^{- k} f (\alpha (\gamma (t)))} \mathrm{d} t + \int_{t_1}^{t_2} \gamma' (t) 
  \frac{kj' (\alpha) j (\alpha, \gamma (t))}{j (\alpha, \gamma (t))^2} \mathrm{d}
  t \nonumber\\
  = & \int_{t_1}^{t_2} \gamma' (t)  \frac{(f [\alpha]_k)' (\gamma (t))}{f [\alpha]_k
  (\gamma (t))} \mathrm{d} t + \int_{t_1}^{t_2} \gamma' (t)  \frac{kj' (\alpha)}{j
  (\alpha, \gamma (t))} \mathrm{d} t \nonumber\\
  = & \int_{\gamma} \frac{(f [\alpha]_k)' (z)}{f [\alpha]_k (z)} \mathrm{d} z +
  \int_{\gamma} \frac{kj' (\alpha)}{j (\alpha, z)} \mathrm{d} z. \nonumber
\end{align}}}

Then for the general case where $\gamma$ is piecewise continuously differentiable, simply split the integral up into the continuously differentiable pieces of $\gamma$ and apply the result to each of them, then add them back together to get the general result.
\end{proof}

\begin{cor}
    If in the above, $f$ is weakly modular of weight $k$ with congruence subgroup $\Gamma$, and $\alpha\in\Gamma$, then
    $$\int_{\alpha(\gamma)}\frac{f'(z)}{f(z)}\mathrm dz=\int_{\gamma} \frac{f' (z)}{f (z)} \mathrm{d} z + \int_{\gamma} \frac{kj' (\alpha)}{j (\alpha, z)} \mathrm{d} z.$$
\end{cor}
\begin{proof}
    Since $f$ is weight-$k$ invariant number $\alpha$, then
    $$f[\alpha]_k(z)=j(\alpha,z)^{-k}f(\alpha(z))=f(z),$$ 
     which we substitute into Lemma \ref{alphaarg}.
\end{proof}

\subsubsection{Proof of non-existence}

Now we are ready to prove there are no cusp forms as those described in the title of this section.

Assume on the contrary that $f : \mathcal{H} \rightarrow \mathbb{C}$ is a non-zero cusp form
of weight $2$ over $\Gamma_0 (2)$. Let $t = \begin{pmatrix}
  1 & 1\\
  0 & 1
\end{pmatrix}$ and $\alpha = \begin{pmatrix}
  - 1 & 0\\
  2 & - 1
\end{pmatrix}$ so $t, \alpha \in \Gamma_0 (2)$. Also let $s =
\begin{pmatrix}
  0 & - 1\\
  1 & 0
\end{pmatrix} \in \SLZ$. Then let $\gamma$ be the closed path composed of the paths shown below. That is, 
$$\gamma=\gamma_1+\gamma_5+\alpha(-\gamma_3)+s(\gamma_6)+\gamma_3+\gamma_4+t(-\gamma_1)+\gamma_2.$$
\begin{center}
\asyinclude{gamma02.asy}
\end{center}

We take advantage of the fact that zeroes of a holomorphic function are
isolated to make the path avoid zeroes: $\gamma_4$ and $\gamma_5$ we try to
make small enough arcs so that their circles do not enclose any zero of $f$
other than the ones potentially existing at $\frac{1 + i}{2}$ and $\frac{- 1 +
i}{2}$ respectively. This is possible because otherwise there is a sequence of zeroes converging on those points, which would imply that $f=0$.  We can also make $\gamma_6$'s imaginary value arbitrarily
high which allows us to avoid any zeroes since there are only a finite number
of zeroes in the compact region bounded by our path. Similarly, $\gamma_2$'s
imaginary value can be made as high as we want. Finally, if there are any
zeroes on $\gamma_1$ and $\gamma_3$, there are finitely many of them (since
the path is compact) so we can go around them using small detours that are
within the neighbourhood of those zeroes in which there are no other zeroes,
which due to the $\Gamma_0 (2)$-invariance of $f$ means that $t (- \gamma_1)$
and $\alpha (- \gamma_3)$ also do not pass through any zeroes. Thus using the
argument principle (modified by $\SL_2(\Z)$):
\begin{dmath*}
   \left( \text{\# roots inside $\gamma$} \right) =  \frac{1}{2 \pi i} 
  \int_{\gamma} \frac{f' (z)}{f (z)} \mathrm{d} z \\    
  =\frac{1}{2 \pi i}  \left( \int_{\gamma_1} \frac{f' (z)}{f (z)} \mathrm{d} z
  + \int_{t (- \gamma_1)} \frac{f' (z)}{f (z)} \mathrm{d} z + \int_{\gamma_2}
  \frac{f' (z)}{f (z)} \mathrm{d} z + \int_{\gamma_3} \frac{f' (z)}{f (z)} \mathrm{d}
  z + \int_{\alpha (- \gamma_3)} \frac{f' (z)}{f (z)} \mathrm{d} z +
  \int_{\gamma_4} \frac{f' (z)}{f (z)} \mathrm{d} z + \int_{\gamma_5} \frac{f'
  (z)}{f (z)} \mathrm{d} z + \int_{s (\gamma_6)} \frac{f' (z)}{f (z)} \mathrm{d} z
  \right)\\
  =  \frac{1}{2 \pi i}  \left( \int_{\gamma_1} \frac{f' (z)}{f (z)} \mathrm{d} z
  - \int_{\gamma_1} \frac{f' (z)}{f (z)} \mathrm{d} z - \int_{\gamma_1} \frac{2
  \cdot 0}{1} \mathrm{d} z + \int_{\gamma_2} \frac{f' (z)}{f (z)} \mathrm{d} z +
  \int_{\gamma_3} \frac{f' (z)}{f (z)} \mathrm{d} z - \int_{\gamma_3} \frac{f'
  (z)}{f (z)} \mathrm{d} z - \int_{\gamma_3} \frac{2 \cdot 2}{2 z - 1} \mathrm{d} z +
  \int_{\gamma_4} \frac{f' (z)}{f (z)} \mathrm{d} z + \int_{\gamma_5} \frac{f'
  (z)}{f (z)} \mathrm{d} z + \int_{\gamma_6} \frac{(f [s]_2)' (z)}{f[s]_2 (z)} \mathrm{d} z +
  \int_{\gamma_6} \frac{2 \times (- 1)}{- z} \mathrm{d} z \right)\\
  =  \frac{1}{2 \pi i}  \int_{\gamma_2} \frac{f' (z)}{f (z)} \mathrm{d} z + 2
  \times \frac{1}{2 \pi i}  \int_{- \gamma_3} \frac{1}{z - \frac{1}{2}} \mathrm{d}
  z - \frac{1}{2 \pi i}  \int_{- \gamma_4} \frac{f' (z)}{f (z)} \mathrm{d} z -
  \frac{1}{2 \pi i}  \int_{- \gamma_5} \frac{f' (z)}{f (z)} \mathrm{d} z -
  \frac{1}{2 \pi i}  \int_{- \gamma_6} \frac{(f [s]_2)' (z)}{f [s]_2 (z)}
  \mathrm{d} z + \frac{2}{2 \pi i} \int_{\gamma_6} \frac{1}{z} \mathrm{d} z. 
\end{dmath*}

Now we address each term individually.  

\begin{itemize}
    \item[1.] $$\frac{1}{2 \pi i}  \int_{\gamma_2}
\frac{f' (z)}{f (z)} \mathrm{d} z$$

Using a change of variables, this equals the multiplicity of $f$ at $0$ in its
$q$-expansion which we will denote by $\nu_{\infty} (f)$. (need to check this bruh)

\item[2.]
$$\frac{1}{2 \pi i}  \int_{- \gamma_3} \frac{1}{z - \frac{1}{2}} \mathrm{d} z$$
integrates anticlockwise around $\frac{1}{2}$, so if the angle of the arc is
$\theta$ radians, then this equals $\frac{\theta}{2 \pi i}$. As we make
$\gamma_4$ and $s (\gamma_6)$ have smaller and smaller radius (by making
$\gamma_6$ have larger imaginary part) then $\theta \rightarrow
\frac{\pi}{2}$. Therefore this integral approaches $\frac{1}{4}$.

$$\frac{1}{2 \pi i}  \int_{- \gamma_3} \frac{1}{z - \frac{1}{2}} \mathrm{d} z=\frac{1}{2 \pi i}  \int_{\theta_1}^{\theta_2} \frac{ie^{it}}{e^{it}} \mathrm{d} t=\frac{\theta_2-\theta_1}{2\pi}\to \frac{\pi/2}{2\pi}=\frac{1}{4}.$$

\item[3.] $$\frac{1}{2 \pi i}  \int_{- \gamma_4} \frac{f' (z)}{f (z)} \mathrm{d} z$$
integrates anticlockwise around $\frac{1 + i}{2}$, and as we decrease the
radius of $\gamma_4$, the angle that it integrates approaches $\frac{\pi}{2}$.
Therefore this approaches $\frac{1}{4}$ of the multiplicity of the zero at
$\frac{1 + i}{2}$.

\item[4.] $$\frac{1}{2 \pi i}  \int_{- \gamma_5} \frac{f' (z)}{f (z)} \mathrm{d} z$$ Similarly as we decrease the radius of $\gamma_5$,
$\frac{1}{2 \pi i}  \int_{- \gamma_5} \frac{f' (z)}{f (z)} \mathrm{d} z$
approaches $\frac{1}{4}$ the multiplicity of the zero at $\frac{- 1 + i}{2}$,
but since $f$ is periodic with period $1$, this is the same as the
multiplicity of the zero at $\frac{1 + i}{2}$. Therefore we will denote both
multiplicities by $\nu_{\frac{1 + i}{2}} (f)$.

\item[5.] $$\frac{1}{2 \pi i}  \int_{- \gamma_6} \frac{(f [s]_2)'
(z)}{f[s]_2 (z)} \mathrm{d} z$$

Due to the cusp condition of cusp forms, we know that $f [s]_2$ has a
$q$-expansion and has a zero at $0$ in its $q$-expansion. However, what is the
value of the minimal $h$ in $q = e^{2 \pi ihz}$ for $f [s]_2$? We know that $f
[s]_2$ is weakly modular with respect to the congruence subgroup $s^{- 1}
\Gamma_0 (2) s$, which is of level $2$ since $\Gamma_0 (2)$ is of level $2$.
Therefore $h \leqslant 2$. However, $s \begin{pmatrix}
  1 & 1\\
  0 & 1
\end{pmatrix} s^{- 1} = \begin{pmatrix}
  1 & 0\\
  - 1 & 1
\end{pmatrix} \not\in \Gamma_0 (2)$, so $\begin{pmatrix}
  1 & 1\\
  0 & 1
\end{pmatrix} \not\in s^{- 1} \Gamma_0 (2) s$, and hence $h \neq 1$.
Therefore $h = 2$, and $f [s]_2$ is periodic with period $2$ which is what
will be used for the $q$-expansion. If we set the imaginary part of $\gamma_6$
to be high enough, then $\frac{1}{2 \pi i}  \int_{- \gamma_6} \frac{(f [s]_2)'
(z)}{f[s]_2 (z)} \mathrm{d} z$ is equal to the multiplicity of the zero at $0$ in the
$q$-expansion of $f [s]_2$, which we will denote by $\nu_0 (f)$.

\item[6.] $$\lim_{\text{Im}\gamma_6\to0}\int_{\gamma_6} \frac{1}{z}
\mathrm{d}z =0.$$ 

As the imaginary part of $\gamma_6$ is increased, $\int_{\gamma_6} \frac{1}{z}
\mathrm{d} z$ tends to $0$ as it starts at real part $1$ and ends at real part $-
1$ with the same imaginary part, so the beginning and end of the path have the
same modulus, but the angle about the origin approaches $0$. 

Alternatively, let $y$ be the imaginary part of $\gamma_6$, the integral becomes
$$\int_{-1}^1 \frac{1}{t+yi}dt=\int_{-1}^1 \frac{t-iy}{t^2+y^2}dt=i\int_{-1}^1 \frac{y}{t^2+y^2}dt=2i\tan^{-1}(1/y).$$

\end{itemize}

Therefore, taking the limit of this equality as we make all the circles on
which $\gamma_4, \gamma_5$ reside as well as the detours tiny and the
imaginary part of $\gamma_6$ really high, we get
\[ \left( \text{\# roots inside $\gamma$ without detours} \right) = -
   \nu_{\infty} (f) + \frac{1}{4} - \frac{1}{4} \nu_{\frac{1 + i}{2}}
   (f) - \frac{1}{4} \nu_{\frac{1 + i}{2}} (f) - \nu_0 (f) \]
\[ \left( \text{\# roots inside $\gamma$ without detours} \right) +
   \nu_{\infty} (f) + \frac{1}{2} \nu_{\frac{1 + i}{2}} (f) + \nu_0 (f) =
   \frac{1}{4} \]
However, due to $f$ being a cusp form, $\nu_{\infty} (f) \geqslant 1$ which is
a contradiction. Therefore there are no non-zero cusp forms of weight $2$ over
$\Gamma_0 (2)$.



%% file: prelimGalois.tex
\section{Preliminaries III}

\subsection{Topology and continuity} 

Keywords: Open, Closed, Continuous, Discrete Topology, Subspace Topology, (infinite) Product Topology.

This is a key subject toward the end of undergraduate. It is also a core topic for graduate school. We can avoid alot of this here but it is still good to know what we are avoiding.

We have already seen open sets and closed sets defined in the complex plane, a topology is something that makes precise the notion of open and closed in other contexts as well. A topological space is a set with a set of subsets called open sets or equivalently closed sets. Intuitively, it is a set with a notion of closeness. We could use limits to express this but we can have a notion of closeness without necessarily a notion of distance.

\begin{definition}[Topological Space] A \emph{topological space} is a pair \( (X, \mathcal{T}) \) where:
\begin{itemize}
    \item \( X \) is a set;
    \item \( \mathcal{T} \subseteq \mathcal{P}(X) \) is a collection of subsets of \( X \), called a \emph{topology}, such that:
    \begin{enumerate}
        \item \( \emptyset \in \mathcal{T} \) and \( X \in \mathcal{T} \);
        \item The union of any collection of sets in \( \mathcal{T} \) is also in \( \mathcal{T} \), i.e., if \( \{U_i\}_{i \in I} \subseteq \mathcal{T} \), then \( \bigcup_{i \in I} U_i \in \mathcal{T} \);
        \item The intersection of any finite number of sets in \( \mathcal{T} \) is also in \( \mathcal{T} \), i.e., if \( U_1, \dots, U_n \in \mathcal{T} \), then \( \bigcap_{j=1}^n U_j \in \mathcal{T} \).
    \end{enumerate}
\end{itemize}

The elements of \( \mathcal{T} \) are called \emph{open sets}. If we swapped finite intersections with arbitrary intersections and arbitrary unions with finite unions, then we get the axioms for \emph{closed} sets. Their relationship is that 

$$U \,\text{open}\,\iff\, X\setminus U\,\,\text{closed}$$

\end{definition} 

Some trivial examples of a topology on $X$ can help formalise some intuitive concepts.

\begin{eg}[Discrete topology] The discrete topology of a set $X$ is the set of \emph{all} subsets of $X$.
\end{eg}

\begin{definition}[Subspace Topology]
Let $(X, \mathcal{T})$ be a topological space and let $Y \subseteq X$ be a subset. The \emph{subspace topology} on $Y$ is defined by
\[
\mathcal{T}_Y = \{ U \cap Y \mid U \in \mathcal{T} \}.
\]
That is, a set $V \subseteq Y$ is open in the subspace topology if and only if there exists an open set $U$ in $X$ such that $V = U \cap Y$.

The topological space $(Y, \mathcal{T}_Y)$ is called a \emph{subspace} of $(X, \mathcal{T})$.
\end{definition}

The open discs that helped define open sets in the complex plane correspond in this abstract setting t0 that of a base open set, defined as follows:

\begin{definition}[Basis for a Topology]
Let $X$ be a set. A collection $\mathcal{B}$ of subsets of $X$ is called a \emph{basis} for a topology on $X$ if:

\begin{enumerate}
    \item For each $x \in X$, there exists at least one $B \in \mathcal{B}$ such that $x \in B$.
    \item If $x \in B_1 \cap B_2$ for some $B_1, B_2 \in \mathcal{B}$, then there exists $B_3 \in \mathcal{B}$ such that $x \in B_3 \subseteq B_1 \cap B_2$.
\end{enumerate}

Given a basis $\mathcal{B}$, the topology $\mathcal{T}$ \emph{generated by} $\mathcal{B}$ is defined by declaring a subset $U \subseteq X$ to be open if for every $x \in U$, there exists a basis element $B \in \mathcal{B}$ such that $x \in B \subseteq U$.

Note every open set is consequently a union of sets in $\mathcal{B}$ and sets in $\mathcal{B}$ are open. 
\end{definition}

\begin{definition}[Product Topology]
Let $(X, \mathcal{T}_X)$ and $(Y, \mathcal{T}_Y)$ be topological spaces. The \emph{product topology} on the Cartesian product $X \times Y$ is the topology $\mathcal{T}_{X \times Y}$ generated by the basis
\[
\mathcal{B} = \{ U \times V \mid U \in \mathcal{T}_X,\ V \in \mathcal{T}_Y \}.
\]
In other words, a subset $W \subseteq X \times Y$ is open in the product topology if and only if for every point $(x, y) \in W$, there exist open sets $U \in \mathcal{T}_X$ and $V \in \mathcal{T}_Y$ such that
\[
(x, y) \in U \times V \subseteq W.
\]

The topology $\mathcal{T}_{X \times Y}$ is called the \emph{product topology}, and the space $(X \times Y, \mathcal{T}_{X \times Y})$ is called the \emph{product space}.
\end{definition}

The natural functions between topologial spaces are continuous functions defined below. They are a generalisation of one's familiar definition of continuous.

\begin{definition}[Continuous Function] Let \( (X, \mathcal{T}_X) \) and \( (Y, \mathcal{T}_Y) \) be topological spaces. A function
\[
f : X \to Y
\]
is said to be \emph{continuous} if for every open set \( V \in \mathcal{T}_Y \), the preimage
\[
f^{-1}(V) = \{ x \in X : f(x) \in V \}
\]
belongs to \( \mathcal{T}_X \). That is, the preimage of every open set in \( Y \) is an open set in \( X \).

\end{definition}

There will be many situations where one is forced to put a topology on a set, for example there is a saying : `when things become infinite, there is always a notion of closeness'

The discrete topology gives a way to put a topology on every set that makes every function continuous.

There is often a canonical choice for the topology in a certain situation based on how it is defined and it usually concerns making the appropriate functions continuous.

\begin{definition}[Product Topology on an Infinite Product]
Let $\{ (X_i, \mathcal{T}_i) \}_{i \in I}$ be a family of topological spaces indexed by a set $I$. Let $X = \prod_{i \in I} X_i$ denote the Cartesian product of the sets $X_i$.

The \emph{product topology} on $X$ is the topology generated by the basis consisting of all sets of the form
\[
\prod_{i \in I} U_i,
\]
where for each $i \in I$, $U_i \in \mathcal{T}_i$, and $U_i = X_i$ for all but finitely many indices $i$.

In other words, a basic open set in the product topology is a product of open sets where only finitely many factors differ from the whole space.

This topology is the coarsest topology on $X$ such that all the projection maps
\[
\pi_j : X \to X_j, \quad (x_i)_{i \in I} \mapsto x_j
\]
are continuous for every $j \in I$.

The resulting topological space $\left( \prod_{i \in I} X_i,\ \mathcal{T}_{\text{prod}} \right)$ is called the \emph{product space} with the \emph{product topology}.
\end{definition}

\subsubsection{Topological groups and rings}

\begin{definition}[Topological Group]
A \emph{topological group} is a group \( G \) equipped with a topology such that:
\begin{itemize}
    \item The group operation \( \mu : G \times G \to G \), defined by \( \mu(x, y) = xy \), is continuous with respect to the product topology on \( G \times G \);
    \item The inversion map \( \iota : G \to G \), defined by \( \iota(x) = x^{-1} \), is continuous.
\end{itemize}
That is, \( G \) is both a group and a topological space, where the group structure and the topology are compatible in the sense that multiplication and inversion are continuous functions.
\end{definition}

For a topological group, when discussing a base $\mathcal{B}$, it suffices to consider a base of neighbourhoods about the identity. This is because for every open set $U$ and $g\in U$, $g^{-1}(U)$ is an open set containing the identity and isomorphic to $U$.
\newline\newline
Topology comes up as an issue when we discuss infinite extensions in Galois theory next.
\newline\newline
Similar definitions can be made for topological rings.

\subsection{Galois Theory} 

This is often taught in third-year university, but can be a really good topic for Olympiad students in high school (see our online course where we try to teach Galois theory to high school students). It is also reviewed again as a core topic in graduate school algebra.
\newline\newline
Keywords: Field extension, Galois Extension, Normal, Separable, Galois Group, Fundamental Theorem of Galois Theory, Algebraic Closure

\begin{definition}[Field Extension]
Let \( F \) and \( K \) be fields. We say that \( K \) is a \emph{field extension} of \( F \), written \( K/F \), if \( F \subseteq K \) and the operations of \( F \) are those of \( K \) restricted to \( F \). That is, \( K \) is a field containing \( F \) as a subfield.
\end{definition}

$K$ is a $F$-vector space. We say the extension is finite if this vector space is finitely generated.

\begin{definition}[\( \operatorname{Aut}(K/F) \)]
Let \( K/F \) be a field extension. The group \( \operatorname{Aut}(K/F) \) is defined as the set of all field automorphisms of \( K \) that fix every element of \( F \). That is,
\[
\operatorname{Aut}(K/F) = \{ \sigma \in \operatorname{Aut}(K) \mid \sigma(f) = f \text{ for all } f \in F \}.
\]
This is a subgroup of the group of all automorphisms of \( K \), and its group operation is composition of functions. 
\end{definition}

We often see the definition of Galois group given as $$\operatorname{Gal}(K/F)=\operatorname{Aut}(K/F)$$ 
but is there really no difference?

The difference is, using $\Gal(K/F)$ only applies to Galois extensions:


\begin{definition}[Galois Extension]
Let $E/F$ be a field extension. We say that $E$ is a \emph{Galois extension} of $F$ if it is both \emph{normal} and \emph{separable}. That is:
\begin{itemize}
    \item $E$ is a \textbf{normal extension}, meaning every irreducible polynomial in $F[x]$ that has at least one root in $E$ factorises completely into linear factors over $E$.
    \item $E$ is a \textbf{separable extension}, meaning that for every $e\in E$, and $f\in F[x]$ of minimal degree satisfying $f(e)=0$ does not have a double root or higher at $e$.
\end{itemize}
The group of field automorphisms of $E$ that fix $F$ pointwise is called the \emph{Galois group} of $E$ over $F$, denoted $\mathrm{Gal}(E/F)$.
\end{definition}

Some practical remarks about this definition:
\begin{itemize}
    \item Separable is guaranteed for finite fields and fields of characteristic zero.
    \item Normal extensions $F\subset M\subset E$ between $E$ and $F$ can be characterised by the property that $\Gal(E/M)$ is a normal subgroup of $\Gal(E/F)$.
\end{itemize}

Galois extensions are the kinds of extensions that we can apply Galois theory to, that is, we have:

\begin{theorem}[Fundamental Theorem of Galois Theory]
Let $E/F$ be a finite Galois extension with Galois group $G = \mathrm{Gal}(E/F)$. Then there is a one-to-one, inclusion-reversing correspondence between the intermediate fields $K$ such that $F \subseteq K \subseteq E$ and the subgroups $H$ of $G$, given by:
\begin{align*}
    K \mapsto \mathrm{Gal}(E/K), \quad\text{and}\quad H \mapsto E^H
\end{align*}
where $E^H = \{ x \in E \mid \sigma(x) = x \text{ for all } \sigma \in H \}$ is the fixed field of $H$.

\begin{itemize}
    \item $K$ is a normal extension of $F$ if and only if the corresponding subgroup $H$ is a normal subgroup of $G$.
    \item In that case, $\mathrm{Gal}(K/F) \cong G/H$.
\end{itemize}
\end{theorem}

We now give an example of a Galois extension we really care about: (we also start departing from undergraduate Galois theory to more graduate level)

\begin{definition} [Algebraic Closure]
    Let $K$ be a field. Then the algebraic closure of $K$, denoted by $\overline{K},$ is the smallest possible field extension of $K$ which contains $K,$ contains the roots of all polynomials with coefficients in $K,$ and every polynomial with coefficients in $\overline{K}$ has all its roots in $\overline{K}.$
\end{definition}

\begin{prop}\label{closuregal}
    The algebraic closure of a field $K$ of characteristic zero is a Galois extension of $K$.
\end{prop}
\begin{proof}
    We get from characteristic zero that it is separable. Since every polynomial completely factorises in the algebraic closure, it is also normal.
\end{proof}

\subsubsection{Finite Fields}
Keywords: Frobenius Element
\newline\newline
The study of finite fields is a really nice case study in applying Galois theory. Every finite field is found to be classified uniquely by their size, which necessarily is a prime power. The Galois groups also have explicity descriptions.

\begin{definition}
    In a finite field $\mathbb F_{p^k}$, the Frobenius element of the Galois group $\operatorname{Gal}(\mathbb F_{p^k}/\mathbb F_p)$ is defined to be $\operatorname{Frob}_p:\mathbb F_{p^k}\rightarrow\mathbb F_{p^k}=x\mapsto x^p$.
\end{definition}

The finite field $\mathbb F_{p^k}$ has characteristic $p$, i.e. $p=\underbrace{1+1+\dots+1}_{p\text{ times}}=0$. Therefore, for any $a,b\in\mathbb F_{p^k}$,
$$\operatorname{Frob}_p(a+b)=(a+b)^p=\sum_{i=0}^p\binom{p}{i}a^ib^{p-i}$$
and since for all $i=1,2,3,\dots,p-1$, $p\mid\binom{p}{i}$, then those terms disappear and this equals
$$\binom{p}{p}a^p+\binom{p}{0}=b^p=a^p+b^p=\operatorname{Frob}_p(a)+\operatorname{Frob}_p(b)$$

Furthermore,
$$\operatorname{Frob}_p(ab)=(ab)^p=a^pb^p=\operatorname{Frob}_p(a)\operatorname{Frob}_p(b)$$
$$\operatorname{Frob}_p(1)=1^p=1$$

Thus, $\operatorname{Frob}_p$ is indeed an automorphism of $\mathbb F_{p^k}$. Furthermore, by Fermat's Little Theorem we know that $\operatorname{Frob}_p$ fixes $\mathbb F_p$, so it is part of the Galois group $\operatorname{Gal}(\mathbb F_{p^k}/\mathbb F_p)$. Indeed, this is closely related to the Frob that appears in various statements for modular Galois representations.

\begin{theorem}[Classification of Extensions of Finite Fields]
Let \( \mathbb{F}_q \) be a finite field with \( q = p^r \) elements, where \( p \) is prime. Then:

\begin{itemize}
    \item For each integer \( n \geq 1 \), there exists a unique (up to isomorphism) finite extension field \( \mathbb{F}_{q^n} \) of \( \mathbb{F}_q \) with \( q^n \) elements.
    
    \item The extension \( \mathbb{F}_{q^n} / \mathbb{F}_q \) is Galois of degree \( n \), and its Galois group is cyclic:
    \[
    \operatorname{Gal}(\mathbb{F}_{q^n} / \mathbb{F}_q) \cong \mathbb{Z}/n\mathbb{Z},
    \]
    generated by the Frobenius automorphism \( \operatorname{Frob}_q : x \mapsto x^q \).
    
    \item The lattice of intermediate fields between \( \mathbb{F}_q \) and \( \mathbb{F}_{q^n} \) corresponds bijectively to the divisors of \( n \) (note each subgroup of $\Z/n\Z$ is $\Z/d\Z$ for divisors $d$ of $n$). For each divisor \( d \mid n \), there is a unique subfield \( \mathbb{F}_{q^d} \subseteq \mathbb{F}_{q^n} \).
\end{itemize}
\end{theorem}

\subsection{Representation theory}

The abstract algebra objects often manifests concretely as functions on other objects. The notions that are associated with representations include: group actions, group representations (linear actions), and a big example for us will be Galois representations.


\begin{definition}[Group Representation]
Let \( G \) be a group and let \( V \) be a vector space over a field \( F \). A \emph{representation} of \( G \) on \( V \) is a group homomorphism
\[
\rho : G \to \operatorname{GL}(V),
\]
where \( \operatorname{GL}(V) \) is the group of invertible linear transformations of \( V \).

If \( \dim V = 2 \) and a basis of \( V \) is chosen, then \( \operatorname{GL}(V) \cong \operatorname{GL}_2(F) \), the group of invertible \( 2 \times 2 \) matrices over \( F \). In this case, the representation \( \rho \) can be viewed as a homomorphism
\[
\rho : G \to \operatorname{GL}_2(F),
\]
and for each \( g \in G \), \( \rho(g) \) is a \( 2 \times 2 \) matrix representing the linear action of \( g \) on \( V \) with respect to the chosen basis.
\end{definition}

\begin{definition}[Irreducible Representation]
Let \( G \) be a group and let \( V \) be a vector space over a field \( F \). A \emph{representation} of \( G \) is a group homomorphism
\[
\rho : G \longrightarrow \mathrm{GL}(V).
\]
The representation \( \rho \) (or \( V \), viewed as a \( G \)-module) is said to be \emph{irreducible} if the only \( G \)-invariant subspaces of \( V \) are \( \{0\} \) and \( V \) itself.

In other words, \( V \) has no proper, nontrivial subspaces \( W \subset V \) such that \( \rho(g)(W) \subseteq W \) for all \( g \in G \).
\end{definition}

%% file: galois.tex
\section{Galois Representations}

One of the main goals of this Section is to establish enough definitions and background to state the assumptions and conclusions of Serre's modularity conjecture.

Keywords: $l$-adic/mod $l$ Galois Representation, Odd, Irreducible, Two-Dimensional, Absolutely Irreducible, Unramified, Trace of Frobenius

\subsection{Definition of Galois Representation}

Keywords: $l$-adic/mod $l$ Galois Representation, Continuous, Two-Dimensional
\newline\newline
At first glance one might think a Galois representation is a group representation of a Galois group but this isn't quite enough, one realises that the Galois group isn't just a group but a topological group and one needs the representation to also respect the topology, that is, be continuous.

\subsubsection{Absolute Galois Group as a topological group}

\begin{definition}[Absolute Galois Group]
Let $\overline{\Q}$ be the algebraic closure of $\Q$.

By Proposition $\ref{closuregal}$, $\overline{ \mathbb{Q}}$ is a Galois extension of $\Q$. 

We define the absolute Galois group as

    $$G_\Q=\Gal(\overline{ \mathbb{Q}}/\mathbb{Q})=\Aut(\overline{ \mathbb{Q}}).$$
\end{definition}
    

So we can apply our Galois theory to this right? Unfortunately, we cannot since this is an infinite extension! What happens if the extension is infinite? This is where the issue of topology arises! We need to first define a topology of the Galois group of an infinite extension and then the Galois correspondence can be salvaged as between \emph{closed} subgroups rather than all of the subgroups.
\newline\newline
How to define a topology on the Galois group?
\newline\newline
Let's start with finite Galois groups from which we are familiar. If $K$ is a finite Galois extension of $\Q$, the Galois group $\Gal(K/\Q)$ is finite and we give it the discrete topology.
\newline\newline
We now look more closely at how these finite groups interact (this is already part of the Fundamental theorem)

\begin{theorem}[Restriction Maps Between Galois Groups]
Let $\Q \subseteq L \subseteq M$ be Galois extensions, in particular $\Q\subseteq L$ is a finite extension. Then the restriction map
\[
\operatorname{res}_{M/L} : \operatorname{Gal}(M/\Q) \longrightarrow \operatorname{Gal}(L/\Q),
\quad \sigma \mapsto \sigma|_L
\]
is a group homomorphism.

Moreover, the subgroup
\[
\operatorname{Gal}(M/L) = \{ \sigma \in \operatorname{Gal}(M/\Q) \mid \sigma|_L = \operatorname{id}_L \}.
\]
are exactly those that are mapped to the identity.
\end{theorem}

\begin{proof}
    We prove here the restriction map indeed lands in $\Gal(L/\Q)$. For $\sigma\in\Gal(M/\Q)$ and $\alpha\in L$ we have $\sigma(\alpha)\in M$ but we want to show $\sigma(\alpha)\in L$. Since $L$ is a finite extension of $\Q$, then $\alpha$ is algebraic over $\Q$, so let $p\in \Q[x]$ be an irreducible polynomial with coefficients in $\Q$ such that $p(\alpha)=0$. Then $$p(\sigma(\alpha))=\sigma(p(\alpha))=\sigma(0)=0.$$
    So $\sigma(\alpha)$ is a root of $p$. Now as $L$ is a normal extension, the roots of $p$ are in $L$, which proves the claim.
\end{proof}

In the above we are interested in $M=G_\Q$ and otherwise $M$ and $L$  finite. 

Recall the definition of inverse limit:

\begin{definition}[Inverse System and Inverse Limit]
\leavevmode
\begin{enumerate}
    \item Let $I$ be a set with a partial order $\leq$. An \emph{inverse system} (also called a projective system) indexed by $I$ is a collection of sets (or groups or rings or topological spaces) $\{A_i\}_{i \in I}$ together with maps (of sets, groups, rings, or topological spaces)
    \[
    \varphi_{ij} : A_i \to A_j \quad \text{for all } j \leq i
    \]
    such that $\varphi_{ii} = \mathrm{id}_{A_i}$ for all $i \in I$ and $\varphi_{jk} \circ \varphi_{ij} = \varphi_{ik}$ for all $k \leq j \leq i$. The $\varphi_{ij}$ are called \emph{transition maps}.

    \item The \emph{inverse limit} of the system $\left( \{A_i\}, \{\varphi_{ij}\} \right)$ is the set/group/ring/topological space
    \[
    \varprojlim_{i} A_i = \left\{ (a_i)_{i \in I} \in \prod_{i \in I} A_i : \varphi_{ij}(a_i) = a_j \text{ for all } j \leq i \right\}
    \]
    of compatible systems of elements of the inverse system. In the case of topological spaces, we put the subspace topology on $\varprojlim A_i$.
\end{enumerate}
\end{definition}

The set of finite Galois groups form an inverse system with the restriction maps as the transition maps.

\begin{thm}
    
$$G_\Q=\varprojlim \Gal(K/\Q) $$
\end{thm}

\begin{definition}[Krull topology of $G_\Q$]
    Define discrete topology for finite Galois groups and define topology on inverse limit as subspace topology of the product topology. Give this topology to $G_\Q$ from the inverse limit. Now $G_\Q$ is a topological group!
\end{definition}
\begin{proposition}
The collection
\[
\left\{ \operatorname{Gal}(\overline{\Q}/K) \mid K/\mathbb{Q} \text{ finite Galois} \right\}
\]
forms a basis of open neighbourhoods of the identity element in \( \operatorname{Gal}(\overline{\Q}/\mathbb{Q}) \).
\end{proposition}

We are now ready to give the definition of a Galois representation.

\begin{definition}[Galois Representation of the Absolute Galois Group of \(\mathbb{Q}\)]
A \emph{Galois representation} of the absolute Galois group of \( \mathbb{Q} \) is a continuous group homomorphism
\[
\rho : \operatorname{Gal}(\overline{\mathbb{Q}}/\mathbb{Q}) \longrightarrow \mathrm{GL}_n(R),
\]
where:
\begin{itemize}
    \item \( \overline{\mathbb{Q}} \) is the algebraic closure of \( \mathbb{Q} \),
    \item \( R \) is a topological ring (commonly \( \mathbb{Q}_\ell, \mathbb{Z}_\ell, \mathcal{O}_F \), or a finite field),
    \item \( \mathrm{GL}_n(R) \) is given the topology induced by \( R \),
    \item \( \operatorname{Gal}(\overline{\mathbb{Q}}/\mathbb{Q}) \) has the Krull topology,
    \item and the homomorphism \( \rho \) is continuous with respect to these topologies.
    \item the $n$ is the dimension of the representation, for us $n=2$.
\end{itemize}
\end{definition}

We define $\mathcal{O}_F$ later, but for the other choices of $R$ we have $R=\F_\ell$ gives mod $l$ Galois representations:

\begin{definition}[Mod \(\ell\) Galois Representation]
Let \(\ell\) be a prime number. A \emph{mod \(\ell\) Galois representation} is a continuous group homomorphism
\[
\bar{\rho} : \operatorname{Gal}(\overline{\mathbb{Q}}/\mathbb{Q}) \longrightarrow \mathrm{GL}_n(\mathbb{F}_\ell),
\]
where:
\begin{itemize}
    \item \( \mathbb{F}_\ell \) is the finite field with \( \ell \) elements,
    \item \( \operatorname{Gal}(\overline{\mathbb{Q}}/\mathbb{Q}) \) has the Krull topology,
    \item \( \mathrm{GL}_n(\mathbb{F}_\ell) \) is given the discrete topology,
    \item and \( \bar{\rho} \) is continuous with respect to these topologies.
\end{itemize}
\end{definition}

and if $R$ is a finite extension of $\Q_\ell$ we get an $\ell$-adic Galois representation. Here we will mainly be concerned with $R=\Q_\ell$:

\begin{definition}[\(\ell\)-adic Galois Representation]
Let \(\ell\) be a prime number. An \emph{\(\ell\)-adic Galois representation} is a continuous group homomorphism
\[
\rho : \operatorname{Gal}(\overline{\mathbb{Q}}/\mathbb{Q}) \longrightarrow \mathrm{GL}_n(\Q_\ell),
\]
in the literature $\Q_\ell$ could also be replaced by a finite extension of $\Q_\ell$.
\end{definition}

These can be related to each other in similar ways to the construction of the Tate module. Given compatible representations on finite ring $\Z/\ell^n\Z$, an inverse limit can be taken to get $\Z_\ell$, which can then be extended to $\Q_\ell$. From a finite field, a quotient can also be taken to get back to a mod $\ell$ representation.

Next we discuss continuity 

\subsubsection{Continuity of Galois representations}

When the groups are finite with the discrete topology, continuity is essentially vacuous and we are back to the familiar situation without continuity.

Something like this can be achieved for finite fields or more generally finite rings:

\begin{proposition}
Let \( \rho : G_{\mathbb{Q}} \to \mathrm{GL}_2(\mathbb{F}_q) \) be a Galois representation over a finite field. Then \( \rho \) factors as
\[
G_{\mathbb{Q}} \twoheadrightarrow \operatorname{Gal}(K/\mathbb{Q}) \to \mathrm{GL}_2(\mathbb{F}_q)
\]
for some finite Galois extension \( K/\mathbb{Q} \). Moreover. $\F_q$ can be replaced by $\Z/\ell^n\Z$ to get the same conclusion.
\end{proposition}

\begin{proof} We write the argument for $\F_q$ but it is the same for any discrete topology on $\gl_2$.
Since \( \mathrm{GL}_2(\mathbb{F}_q) \) is a finite group with the discrete topology,  $V=\{\mathrm{id}\}\subset \mathrm{GL}_2(\mathbb{F}_q)$ is open. Then \( \rho^{-1}(V) \subseteq G_{\mathbb{Q}} \) is an open neighborhood of the identity in the Krull topology. Hence \( \rho^{-1}(V) \) contains some open normal subgroup \( \operatorname{Gal}(\overline{\mathbb{Q}}/K) \) for some finite Galois extension \( K/\mathbb{Q} \).
\newline\newline
But then \( \rho(\operatorname{Gal}(\overline{\mathbb{Q}}/K)) \subseteq V = \{\mathrm{id}\} \).
\newline\newline
Therefore, \( \rho \) factors through the finite quotient
\[
G_{\mathbb{Q}} / \operatorname{Gal}(\overline{\mathbb{Q}}/K) = \operatorname{Gal}(K/\mathbb{Q}),
\]
as desired.
\end{proof}

We might be careful here to conclude that a continuous representation is equivalent to a finite Galois group representation because this may have something to do with the choice of $K$.
\newline\newline
If we have compatible representations of finite Galois groups, we can also rebuild a continuous Galois representation of the absolute Galois group on $\Z_\ell$

\subsubsection{Example of Galois representation: Cyclotomic character}

To construct the $p$-adic cyclotomic character of $G_\Q$, first consider the $p^k$ roots of unity for any $k\in\Z^+_0$, which are the roots of the rational polynomial $x^{p^k}-1$. Therefore, for any $\sigma\in G_\Q$, $\sigma$ permutes the roots. Furthermore, the roots of unity form a cyclic group of size $p^k$ under multiplication, and since $\sigma$ is multiplicative, it thus induces a group automorphism of that cyclic group. If we take a primitive root of unity to be $\zeta$, so then the roots of unity are $1,\zeta,\zeta^2,\dots,\zeta^{p^k-1}$, we can create a bijective map $f:\{1,\zeta,\zeta^2,\dots,\zeta^{p^k-1}\}\to\Z/p^k\Z$ were $f(\zeta^i)=i\in\Z/p^k\Z$, so that $f(\zeta^a\zeta^b)=a+b=f(\zeta^a)+f(\zeta^b)$. Hence, this is a group isomorphism between the roots of unity under multiplication and $\Z/p^k\Z$ under addition. Therefore $g_\sigma=f\circ \sigma\circ f^{-1}$ is a group automorphism of $\Z/p^k\Z$ under addition. However, by being an automorphism under addition, then for any $\lambda\in\Z/p^k\Z$, $g_\sigma(\lambda x)=\lambda g_\sigma(x)$, since multiplication by $\lambda$ is simply repeated addition. Hence, $g_\sigma$ forms a $\Z/p^k\Z$-linear map from $\Z/p^k\Z$ to itself. 

Furthermore, for any $\sigma_1,\sigma_2\in G_\Q$, $g_{\sigma_1}\circ g_{\sigma_2}=f\circ \sigma_1\circ f^{-1}\circ f\circ \sigma_2\circ f^{-1} = f\circ \sigma_1\circ\sigma_2\circ f^{-1}=g_{\sigma_1\circ\sigma_2}$. Also the identity element of $G_\Q$ fixes the roots of unity so it maps to the identity linear map. Therefore for all $\sigma\in G_\Q$, $g_{\sigma^{-1}}\circ g_\sigma$ is the identity map, so $g_\sigma$ is invertible, i.e. $g_\sigma\in\operatorname{GL}_1(\Z/p^k\Z)$.
\begin{definition}
    $\chi:G_\Q\to\operatorname{GL}_1(\Z/p^k\Z)=\sigma\mapsto g_\sigma$ is the mod-$p^k$ cyclotomic character of $G_\Q$. It is a one dimensional representation.
\end{definition}

A worry that one may have is that this definition depends on our choice of primitive root $\zeta$. Since this is a one-dimensional representation, elements of $\operatorname{GL}_1(\Z/p^k\Z)$ compose like $\Z/p^k\Z$ under multiplication, so composition is commutative. Therefore we can actually show that this definition is independent of the choice of the primitive root. Consider two primitive roots $\zeta_1,\zeta_2$, which will then have bijective group isomorphisms $f_1,f_2$ between the $p^k$ roots of unity under multiplication and $\Z/p^k\Z$ under addition as above. Then let the corresponding mod $p^k$ cyclotomic characters be $\chi_1,\chi_2$ respectively. Then for any $\sigma\in G_\Q$,

\begin{align*}
\chi_2(\sigma)&=f_2\circ\sigma\circ f_2^{-1}=f_2\circ f_1^{-1}\circ f_1\circ\sigma\circ f_1^{-1}\circ f_1\circ f_2^{-1}\\&=\left(f_1\circ f_2^{-1}\right)^{-1}\circ\left(f_1\circ\sigma\circ f_1^{-1}\right)\circ\left(f_1\circ f_2^{-1}\right)\\&=\left(f_1\circ f_2^{-1}\right)^{-1}\circ\chi_1(\sigma)\circ\left(f_1\circ f_2^{-1}\right)
\end{align*}

We know that $f_1\circ f_2^{-1}$ is a group automorphism of $\Z/p^k\Z$, so once again it must be an invertible linear map, and obviously $\chi_1(\sigma)$ is an invertible linear map. Therefore we can commute these to get

$$\left(f_1\circ f_2^{-1}\right)^{-1}\circ\left(f_1\circ f_2^{-1}\right)\circ\chi_1(\sigma)=\chi_1(\sigma)$$

and hence we have shown that $\chi_1=\chi_2$.

We can now combine these into a $p$-adic representation! Note that while it's true that our representation does not depend on the choice of the primitive root of unity, when combining them and showing that they are compatible, we must choose "compatible" roots of unity. An easy way to do this is simply to set each $p^k$ primitive root of unity to the principal one, i.e. $\zeta_{p^k}=e^{2\pi i/p^k}$.
\newline\newline
We view the $p$-adic integers as the inverse limit of the rings $\Z/p^k\Z$ under the transition map between $\Z/p^k\Z\to\Z/p^{k-1}\Z$ where we take the number mod $p^{k-1}$. For each $k\in\Z^+_0$, we let the mod $p^k$ cyclotomic character be $\chi_{p^k}$.
\begin{definition}[$p$-adic cyclotomic character]
   Define the $p$-adic cyclotomic character $\chi$ to be
   $$\chi(\sigma)=(x_0,x_1,x_2,\dots)\mapsto(\chi_{p^0}(\sigma)(x_0),\chi_{p^1}(\sigma)(x_1),\chi_{p^2}(\sigma)(x_2),\dots)$$
   (Here, $(x_0,x_1,x_2,\dots)\in \Z_p=\varprojlim_n \Z/p^n\Z\subset \Z/p^0\Z\times\Z/p^1\Z\times\Z/p^2\Z\times\cdots$)
\end{definition}
\begin{proof}
Now we prove that this is indeed a character. There are a number of things that we need to check. Firstly, we should show that $\chi(\sigma)$ actually outputs something in $\Z_p$. Thus we need to check that for each $k\in\Z^+$, $\chi_{p^k}(\sigma)(x_k)\equiv\chi_{p^{k-1}}(\sigma)(x_{k-1})\mod{p^{k-1}}$.
\newline\newline
$\zeta_{p^k}^p=\zeta_{p^{k-1}}$ (this is what I meant when I said that our choice of primitive roots were compatible). Therefore, $\left(\zeta_{p^k}^a\right)^p=\zeta_{p^{k-1}}^b\iff \zeta_{p^{k-1}}^a=\zeta_{p^{k-1}}^b\iff a\equiv b\mod{p^{k-1}}$. Hence
\begin{align*}
   x_k&\equiv x_{k-1 }&\mod{p^{k-1}}\\
   (\zeta_{p^k}^{x_k})^p&=\zeta_{p^{k-1}}^{x_{k-1}}\\
   \sigma((\zeta_{p^k}^{x_k})^p)&=\sigma(\zeta_{p^{k-1}}^{x_{k-1}})\\
   \sigma((\zeta_{p^k}^{x_k}))^p&=\sigma(\zeta_{p^{k-1}}^{x_{k-1}})\\
   \chi_{p^k}(\sigma)(x_k)&\equiv\chi_{p^{k-1}}(\sigma)(x_{k-1})&\mod{p^{k-1}}
\end{align*}

Then we need to show that $\chi(\sigma)$ is an invertible linear map. This is clear as each of the components is an invertible linear map, so it's a linear map, and an inverse function can be constructed by inverting each of the components. So $\chi_{p^k}(\sigma)\in\operatorname{GL}_1(\Z_p)$.
\newline\newline
Finally we need to show that this is a representation. However, this is also obvious, as since the individual $\chi_{p^k}$'s are representations, and $\chi$ simply applies them to the components, then the whole thing preserves group operations and hence is a representation.
\end{proof}

\subsubsection{Hypotheses of Serre's conjecture}
Serre's conjecture aims to find minimal hypotheses on a two dimensional mod $\ell$ Galois representation in order to conclude that it is modular. What exactly modular means, we will ignore for now and focus on what conditions appear in this minimal list.

\begin{definition}[Odd]
   Let $c\in\mathrm{Gal}(\overline{\mathbb{Q}}/\mathbb{Q})$ denotes complex conjugation over $\mathbb{C}$, that is 
   $$c(x+iy)=x-iy$$
   for $i^2=-1$ and all $x,y\in\R$. Restricting this makes sense as we can take $\overline{\Q}\subset\C$ and for any $f\in\Z[x]$ with root $\alpha$ we have 
   $$f(c(\alpha))=c(f(\alpha))=0,$$
which is because the coefficients of $f$ are real and $c|_\R=\mathrm{id}_\R$.
\newline\newline
We say a two dimensional mod $\ell$ Galois representation
\( \rho : G_{\mathbb{Q}} \to \mathrm{GL}_2(\mathbb{F}_\ell) \)
is \emph{odd} if
  $$\text{det}\rho(c)=-1.$$ 

  \end{definition}
  
  This is a necessary condition for representations to arise from modular forms but it is not completely clear that it is needed in the hypotheses to conclude modularity. 
\newline\newline
   Note that $\det\rho(c)=\pm1$ since applying $c$ twice is the identity and $\det\rho$ is multiplicative:
   $$1=\det\rho(c^2)=\det(\rho(c)^2)=(\det\rho(c))^2.$$
We might then call those representations with $+1$ even.
\newline\newline
We now discuss irreducibility. The general definition of irreducible representation in Preliminaries III is particularly concrete for two dimensional representations as the only smaller representations are one or zero dimensional. 
\begin{proposition}[Irreducible two-dimensional mod $\ell$ Galois representation] We say 
   \( \rho : G_{\mathbb{Q}} \to \mathrm{GL}_2(\mathbb{F}_\ell) \)
   is irreducible if there are no one dimensional invariant subspaces. Equivalently we have

1. There does not exist $v\in \F_\ell^2$ that is an eigenvector simultaneously of $\rho(\sigma)$ for all $\sigma\in G_\Q$.

   2. There is no $B\in \GL_2(\F_\ell)$ such that $B\rho(\sigma)B^{-1}$ has bottom left entry zero for all $\sigma\in G_\Q$.
   
\end{proposition} 

Irreducibility is also a necessary condition and Serre \cite{Serre} makes a case that it cannot be removed as a hypothesis in his conjecture. He shows this by showing there are reducible representations present satisfying other collections of hypotheses.
\newline\newline
We can now give a statement of Serre's conjecture emphasising the hypotheses:

\begin{thm}[Serre's conjecture (assumptions version)] 
    Any odd, irreducible Galois representation 
    \( \rho : G_{\mathbb{Q}} \to \mathrm{GL}_2(\mathbb{F}_\ell) \)
    is modular.
\end{thm}

\subsection{Galois Representations from an Elliptic Curve}

Keywords: mod $\ell$ Galois Representation, $\ell$-adic Galois Representation, $\ell$-adic Numbers

\begin{definition} [Galois Group acting on Torsion Points]
    Let $\sigma\in\Gal(\overline{ \mathbb{Q}}/\mathbb{Q}).$ Let $P \in E[n]$ and $P=(x,y)$ where $x,y \in \overline \Q.$ Then define $\sigma(P)=(\sigma(x), \sigma(y))$ and $\sigma(\mathcal{O})=\mathcal{O}.$
\end{definition}

Note the $n$-torsion points' coordinates are all algebraic numbers. This is because the elliptic curve has coefficients in $\mathbb{Q}$, then finding the $P$ on the elliptic curve which are $n$-torsion points amounts to solving two simultaneous equations: the first one is the elliptic curve itself, and the second is $nP=\mathcal{O}$, which together is a single variable polynomial equation in $x$. We can then get $y$ via another equation. 
\newline\newline
We show that $\sigma(P)\in E[n]$ and $\sigma$ is additive on $E[n]$.

\begin{lemma} If $\sigma\in$Gal$(\overline{ \mathbb{Q}}/\mathbb{Q})$, then $\sigma(-x)=-\sigma(x)$ for all $x\in \overline{\Q}$.
\end{lemma}

\begin{proof}
    Since a field automorphism must preserve the identity elements and additivity, we have $\sigma(0)=0$ and $\sigma(x+y)=\sigma(x)+\sigma(y)$. Then, $$\sigma(x)+\sigma(-x)=\sigma(x-x)=\sigma(0)=0$$ and hence $$\sigma(-x)=-\sigma(x).$$
\end{proof}

\begin{lemma}
    If $\sigma\in\operatorname{Gal}(\overline\Q/\Q)$ and $Q\in\Q[x_1,x_2,\dots,x_n]$, and $P\in\overline\Q^n$, then $Q(\sigma(P))=\sigma(Q(P))$.
\end{lemma}
\begin{proof}
    Let $P=(p_1,p_2\dots,p_n)$. Since $\sigma$ preserves addition, then $\sigma(Q(P))$ equals the sums of the $\sigma$s of the monomials of $Q$ applied to $P$. Let such a monomial be $qx_1^{\alpha_1}x_2^{\alpha_2}\dots x_n^{\alpha_n}$, where $q\in\Q$. Then since $\sigma$ is multiplicative, 
    $$\sigma(qp_1^{\alpha_1}p_2^{\alpha_2}\dots p_n^{\alpha_n})=\sigma(q)\sigma(p_1)^{\alpha_1}\sigma(p_2)^{\alpha_2}\dots\sigma(p_n)^{\alpha_n}$$
    and since $q\in\Q$ is fixed under $\sigma$, this evaluates to $q\sigma(p_1)^{\alpha_1}\sigma(p_2)^{\alpha_2}\dots\sigma(p_n)^{\alpha_n}$. Adding up the results of the monomials we indeed get $Q(\sigma(P))$.
\end{proof}

\begin{lemma}\label{permuteroots}
    If $\sigma\in\operatorname{Gal}(\overline\Q/\Q)$, $Q(x,y)$ is a polynomial of rational coefficients with degree $n>1$, and $\ell(x,y)$ is a linear polynomial (so its solution set forms a line), and the points $P_1(x_1,y_1), P_2(x_2,y_2),..., P_n(x_n,y_n)\in\overline\Q^2$ form the complete set of distinct common zeros of $Q$ and $\ell$, then $$\{\sigma(P_1), \sigma(P_2),..., \sigma(P_n)\} = \{ P_1, P_2,..., P_n\}.$$
\end{lemma}

\begin{proof}
  For $i\in\mathbb{N}:1\leq i\leq n$, $Q(\sigma(P_i))= \sigma(Q(P)) = \sigma(0) = 0$. Similarly, $\ell(\sigma(P_i))=\sigma(\ell(P_i))=\sigma(0)=0$  Thus, $\sigma$ takes the common zeroes of $Q,\ell$ to other common zeroes. Furthermore, since $P_1,P_2,\dots,P_n$ are distinct points on a line, then they either have distinct $x$-coordinates or distinct $y$-coordinates. If they have distinct $x$ coordinates, so then since $\sigma$ is injective, they all get taken to distinct $x$-coordinates, and a similar thing happens when they have distinct $y$-coordinates, so distinct points get taken to distinct points. Since there are a finite number of common zeroes, then $\sigma$ permutes them.  
\end{proof}

Note that in the above, by Bezout's Theorem, in an algebraically closed field such as $\overline\Q$, there are always going to be $1\times n=n$ common zeroes of $Q$ and $\ell$ counting "multiplicity". However, we only stated this lemma for distinct roots, not multiple roots. Unfortunately stating the exact definition of multiplicity requires the algebraic geometry, so we use the same working definition of multiplicity as we did in the uniformisation section, which is that in the case that $Q$ is an elliptic curve, the multiplicity of the intersection of an elliptic curve $y^2=p(x)$ with a line $\ell(x)=mx+b$ is equal to the multiplicity of the $x$-coordinate of the intersection as a root of the polynomial $\ell^2-p$. We will then state the analogous lemma without proof:

\begin{lemma}
    If $\sigma\in\operatorname{Gal}(\overline\Q/\Q)$, $Q(x,y)$ is a polynomial of rational coefficients with degree $n>1$, and $\ell(x,y)$ is a linear polynomial (so its solution set forms a line), and the points $P_1(x_1,y_1), P_2(x_2,y_2),..., P_n(x_n,y_n)\in\overline\Q^2$ form the complete set of common zeros of $Q$ and $\ell$ counting multiplicity, then $\sigma$ permutes the common zeroes of each multiplicity.
\end{lemma}

\begin{lemma}
    $\sigma$ is additive on the torsion points of a rational elliptic curve, i.e. $\sigma(P+Q)=\sigma(P)+\sigma(Q)$.
\end{lemma}

\begin{proof} 
    Suppose that the three points $P(x_1,y_1), Q(x_2,y_2),$ and $R(x_3,y_3)$ are collinear and lie on the Elliptic curve. We cover the case where the three points are distinct.

Then we have $(\sigma(x_1),\sigma(y_1)), (\sigma(x_2),\sigma(y_2)),$ and $(\sigma(x_3),\sigma(y_3))$ are equal to $P(x_1,y_1), Q(x_2,y_2),$ and $R(x_3,y_3)$  in some order by Lemma \ref{permuteroots}. Hence we then have $$\sigma(P)+\sigma(Q)=(\sigma(x_1),\sigma(y_1)) + (\sigma(x_2),\sigma(y_2))=(\sigma(x_3),-\sigma(y_3))$$ as the three points are collinear on the Elliptic curve. But $$(\sigma(x_3),-\sigma(y_3))=(\sigma(x_3),\sigma(-y_3))=\sigma(x_3,-y_3)=\sigma(P+Q).$$ Hence, $$\sigma(P+Q)=\sigma(P)+\sigma(Q).$$
\end{proof}

 \begin{lemma}
     The function $\sigma$ permutes the set of $n$-torsion points. 
 \end{lemma}

 \begin{proof}
    Since $P$ is an $n$-torsion point, we have $nP=\mathcal{O}.$ Now by Lemma $5.3$ we have get that $$n\sigma(P)=\sigma(nP)=\sigma(\mathcal{O})=\mathcal{O}.$$ Hence $\sigma(P)\in E(n)$. Also, $\sigma$ is additive and the only input which goes to $0$ is $0$ itself hence it is injective because if $\sigma(x)=\sigma(y)$ this means $$\sigma(x-y)=\sigma(x)-\sigma(y)=0,$$ so $x=y.$ This implies that since the set of $n$-torsion points is finite the function $\sigma$ actually permutes the set of $n$-torsion points. 
    \end{proof}

    \begin{definition} [$l$-adic Galois Representation of $E$]
 From the additive action of $G_\Q$ on each $E[\ell^n]$, the multiplication by $\ell$ transition maps are compatible with the action and hence we get an action of $G_\Q$ on the Tate module of $E$ over $\Z_\ell$, from which we can consider as a representation over $\Q_\ell$. Call this representation $\rho_{E,\ell}$
  $$\rho_{E, l}: \text{Gal}(\overline{\Q},\Q)\mapsto \text{Aut}(\tate[l])\cong GL_2(\mathbb{\Z}_\ell)\subset GL_2(\mathbb{\Q}_\ell).$$


\end{definition}

This can be found on page 88 \cite{Silverman}. We have tried to discover it for ourselves here without looking there.
\newline\newline
If instead of including into the $\ell$-adic numbers, we quotient to $\F_\ell$ we get the mod $\ell$ representation:

\begin{definition}[mod $\ell$ Galois Representation of an Elliptic Curve]
 $$\overline{\rho_{E, l}}: \text{Gal}(\overline{\Q},\Q)\mapsto \text{Aut}(\tate[l])\cong GL_2(\mathbb{\Z}_l)\to GL_2(\mathbb{F}_l).$$

This is the same as $\galq$ acting on $E[\ell]$, which we call the $\ell$-division points of $E$ when considered as an $\F_\ell$-vector space.

\end{definition}

\subsubsection{Galois reps from Elliptic curves and Serre's conjecture}

\begin{proposition}[Proposition 6 \cite{Serre}]

Let $P\geq5$ be a prime and let $E$ be the Frey curve for 
$$x^P+y^P=z^P.$$

Then the mod $P$ Galois representation $\overline{\rho_{E, P}}$ for the Frey curve is irreducible.
    
\end{proposition}

\begin{proposition}[Lemma 4 \cite{Serre}] If $\rho$ is irreducible, then
    
$\det\rho=\chi$

\end{proposition}
    
\begin{proposition}[\cite{Serre}]
    
If $\det\rho=\chi$ then $\rho$ odd
\end{proposition}

\subsubsection{Wiles' Use of Galois Representations}
Recall in our story that it was the $3$-adic representation Wiles proved was modular to conclude FLT, while it was the mod 3 representation that he already knew was modular from existing results of Langlands and Tunnel. It was the mod $P$ representation of the Frey curve that was used to prove it was not modular.
\newline\newline
Langlands-Tunnel means mod 3 rep of Frey irreducible implies it is modular. It is known that Elliptic curve modular if and only if the 3-adic rep is modular.
\newline\newline
Wiles uses some mod 5 reps to get around the irreducibility condition (3-5 trick?). He then concludes mod 3 is modular. He uses this to start an induction for high powers of 3 to get that 3-adic is a module. He then concludes E is modular.

\subsection{Modular Galois Representations}

Here we focus on the conclusions of Serre's conjecture. The main idea is that given a modular form, one can attach a Galois representation, while Serre's conjecture allows us to conclude when a given Galois representation is of this form. We unpackage the necessary conditions associated with this.

Keywords: Frobenius, Finite Flat, Unramified
\newline\newline
Questions:
\begin{itemize}
    \item How to get from modular form to Galois representation?
    \item What is a modular Galois representation and its level?
    \item Why is level of modular Galois representation = level of modular form = conductor of E?
\end{itemize}

Here we explore various other properties of Galois representations and their definitions.
\newline\newline
Our main goal is to understand the definitions involved in the statement of Serre's conjecture:

\begin{thm}[Serre's Conjecture]

Let 

$$\rho: G\to GL_2(\overline{\F}_p)$$

be a two dimensional, irreducible, odd, Galois representation. 

Attach to $\rho$ three quantities $N,k,\in\N$ and $\epsilon:\Z/N\Z\to\F_p\setminus\{0\} $ as described in \cite{Serre}.

Then there exists a cusp form

$$f=\sum_n A_nq^n$$

of weight $k$ and level $N$ such that for all primes $\ell$ coprime to $Np$, $\rho$ is unramified at $\ell$ and we have

$$\mathrm{Tr}(\mathrm{Frob}_{\ell, \rho}) = A_\ell \quad \text{and} \quad \det(\mathrm{Frob}_{\ell, \rho}) = \varepsilon(\ell) \ell^{k-1},
$$
    
\end{thm}

If we get to it, then we would also like to show that Frey's curve gives a Galois representation satisfying the hypotheses of this conjecture.

We examine the conclusion in several parts:

1. Weight, Level, $k$, $N$, $\epsilon$.

2. Unramified, finite, $\mathrm{Tr}(\mathrm{Frob}_{\ell, \rho})$ and $\mathrm{det}(\mathrm{Frob}_{\ell, \rho})$.







\subsubsection{Conductor, level and Weight, $k, N, \epsilon$}

Serre explicitly gives a way to get $N$ and $k$ from the Galois representation in his 1987 paper \cite{Serre}. These correspond to the level and weight of the modular form. We will be interested in how these apply to the Frey curve, which will ultimately have $k=2$, $\epsilon=1$ and $N$ lowered to $2$. In particular, he describes in detail the condition for $k=2$:

\begin{proposition}[\cite{Serre}]
    For Frey's curve, $k=2$ from $\nu_P(\Delta)=0\mod P$.
\end{proposition}

The conductor $N$, which corresponds to the level, is obtained from the Galois representation as the \emph{Artin conductor}. We did not explore this further.
\newline\newline
Once $N$ is defined and a modular form exists of level $N$, Ribet's level lowering argument can be applied to get a form of level $2$.
\newline\newline
We will revisit some of this proof in the next subsection.
\newline\newline
Lastly, we have $\epsilon= 1$ since $\det\rho=\chi$, the cyclotomic character.

\subsubsection{Unramified, Finite, trace and determinant of Frobenius}

Here we give a sketch/summary of the roles of each of these definitions and treat them in more detail in the following section.

\begin{itemize}
    \item Unramified at a prime $p$ is a property of a Galois representation of $G_\Q$ that allows us to define an element $\Frp\in G_\Q$ up to conjugacy in the image of the representation.
\item 
$\Frp$ is closely related to the Frobenius automorphism of finite fields of characteristic $p$ and generates a large part of $G_\Q$. Knowing how this element act will determine a large chunk of the overall representation.
\item Trace and determinant is taken of the image of $\Frp$ in the representation. They are well-defined because $\Frp$ is defined up to conjugacy and trace/determinant are invariant under conjugacy. In a two dimensional representation where each matrix has 4 entries, knowing the trace and determinant gives alot of information. It essentially gives two out of the four entries and also allows the eigenvalues of the matrix to be determined.
salso generates a large part of G and the trace and determinant give two out of 4 entries worth of information.

\item The Trace of the image of Frobenius allows one to recover the modular form via its Fourier coefficients $A_p$. These are eigenvalues for the Hecke operators whose eigenvectors are the forms themselves, such modular forms are called \emph{eigenforms}.

\item Finite (or finite flat) is defined in \cite{Serre} and we do not understand the definition in terms of group schemes nor explore it further, other than to note some practical remarks: it is claimed that for primes not equal to $p$, finite is equivalent to unramified.

At $p$, being finite is as good as one gets to being unramified and leads to many similar desirable consequences.

Finite at $p$ also leads to $k=2$. (\cite{Serre}).

\item It is known that for an elliptic curve E, its
mod $p$ Galois representation is unramified outside $pN$ and finite at $p$.

\end{itemize}

\subsubsection{Ribet's Theorem revisited}

If we look at Ribet's theorem (see for example, screenshot in example \ref{Ribet}) again, we see that the only condition we had not addressed earlier was that every time a prime was removed from the level, the representation needed to be finite/finite flat at that removed prime.
\newline\newline
We sketch again how Ribet's proof of FLT goes from Ribet's Theorem, showing Frey's curve is not modular.

\begin{itemize}
    \item Take the Frey curve $E$ from $x^P+y^P=z^P$ for prime $P\geq 5$. Take its mod $P$ Galois representation. It is known generally that from Elliptic curves, the representation is unramified outside $NP$ and finite at $P$. It is also known that specifically for Frey's curve that its representation is unramified outside of $P$ and $2$. (This is where $\Delta$ being a $P$ power is used, since the condition is $P\mid \nu_\ell(\Delta)$ for unramified at $\ell$.)

    Hence the representation is finite/ finite flat at all primes except $2$.

    \item There are two parts in the statement of Ribet's Theorem $(i),(ii)$: $(i)$ for the prime factors of $N$ (noting it is square free) and $(ii)$ for the prime $P$ dividing $N$. 

    We can remove these from the level provided the representation is finite at these primes.

    Hence we are left with level 2.

\end{itemize}

%% file: algnt.tex
\subsection{Unramified at a prime and Trace of Frobenius}

\subsubsection{Galois Theory Extended aka Algebraic Number Theory}

I will start moving this stuff to main section 

\begin{definition}[Frobenius Element]
    In a finite field $\mathbb F_{p^k}$, the Frobenius element of the Galois group $\operatorname{Gal}(\mathbb F_{p^k}/\mathbb F_p)$ is defined to be $\operatorname{Frob}_p:\mathbb F_{p^k}\rightarrow\mathbb F_{p^k}=x\mapsto x^p$.
\end{definition}

The finite field $\mathbb F_{p^k}$ has characteristic $p$, i.e. $p=\underbrace{1+1+\dots+1}_{p\text{ times}}=0$. Therefore, for any $a,b\in\mathbb F_{p^k}$,
$$\operatorname{Frob}_p(a+b)=(a+b)^p=\sum_{i=0}^p\binom{p}{i}a^ib^{p-i}$$
and since for all $i=1,2,3,\dots,p-1$, $p\mid\binom{p}{i}$, then those terms disappear and this equals
$$\binom{p}{p}a^p+\binom{p}{0}b^p=a^p+b^p=\operatorname{Frob}(a)+\operatorname{Frob}(b)$$

Furthermore,
$$\operatorname{Frob}(ab)=(ab)^p=a^pb^p=\operatorname{Frob}(a)\operatorname{Frob}(b)$$
$$\operatorname{Frob}(1)=1^p=1$$

Thus, $\operatorname{Frob}$ is indeed an automorphism of $\mathbb F_{p^k}$. Furthermore, by Fermat's Little Theorem we know that $\operatorname{Frob}$ fixes $\mathbb F_p$, so it is part of the Galois group $\operatorname{Gal}(\mathbb F_{p^k}/\mathbb F_p)$.

\begin{definition}[Algebraic Integers]
    The algebraic integers $\mathbb A$ are the set of complex numbers which are roots of monic polynomials with integer coefficients.
\end{definition}
The algebraic integers form a ring, in particular the sum, product, and difference of two algebraic integers are also algebraic integers. All roots of monic polynomials with algebraic integers as coefficients are also algebraic integers. Furthermore, $\mathbb A\cap\mathbb Q=\mathbb Z$.

\begin{definition}[Number Field]
    A number field is a finite degree extension of $\mathbb Q$, i.e. it forms a $\mathbb Q$-vector space of finite dimension.
\end{definition}

Any finite degree field extension is necessarily algebraic, meaning every number in it is algebraic (every element of it is a root of some polynomial with coefficients in the base field), so a number field is always a subfield of $\overline{\mathbb Q}$.

\begin{definition}
    For a number field $K$, its ring of integers $\mathcal O_K$ is $K\cap\mathbb A$.
\end{definition}

Being the intersection of two rings, the ring of integers of any number field $K$ is a ring. It must contain $\mathbb Z$.

\begin{lemma}\label{scaletoint}
Any algebraic number can be scaled by some non-zero integer to become an algebraic integer.
\end{lemma}
\begin{proof}
    Let the algebraic number be $\alpha$, such that it is a root of a polynomial with integer coefficients:
    $$a_n\alpha^n+a_{n-1}\alpha^{n-1}+\dots+a_0=0$$
    Then multiply the equation by $a_n^{n-1}$
    $$a_n^n\alpha^n+a_{n-1}a_n^{n-1}\alpha^{n-1}+a_{n-2}a_n^{n-1}\alpha^{n-2}+\dots+a_0a_n^{n-1}$$
    $$(a_n\alpha)^n+a_{n-1}(a_n\alpha)^{n-1}+a_{n-2}a_n(a_n\alpha)^{n-2}+\dots+a_0a_n^{n-1}$$
    so indeed, $a_n\alpha$ is a root of a monic polynomial with integer coefficients and is hence an algebraic integer.
\end{proof}

\begin{lemma}
    The minimal polynomial over $\mathbb Q$ of an algebraic integer is monic with integer coefficients.
\end{lemma}
\begin{proof}
    If $f\in\mathbb Z[x]$ were a minimal monic polynomial with the algebraic integer as a root, and $g\in\mathbb Q[x]$ were the minimal polynomial, then $g\mid f$. If $\deg g<\deg f$, then that means $f$ is reducible over the rationals, so by Gauss' Lemma it is reducible over the integers. This means $f$ can be factored into two lower degree monic polynomials with integer coefficients, which is a contradiction to minimality. Hence $\deg g=\deg f$, so $g=f$ up to multiplication by a constant factor so $f$ is the minimal polynomial.
\end{proof}

\begin{definition}
    Let $K$ be a number field, which is a degree $n$ extension of $\mathbb Q$, so $K\cong\mathbb Q^n$ as $\mathbb Q$-vector spaces. For a number $s\in K$, define $\operatorname{Tr}_{K/\mathbb Q}(s)\in\mathbb Q$ to be the trace of the linear map $x\mapsto sx:K\rightarrow K$ when considered as an $n\times n$ matrix.
\end{definition}
Note that this trace is independent of which basis you choose to convert the linear map into a matrix, as different bases are related by a linear map and trace is invariant under conjugacy.

\begin{lemma}
    If $K$ is a number field of degree $d$ and $s\in K$, where the minimal polynomial of $s$ over $\mathbb Q$ is $a_nx^n+a_{n-1}x^{n-1}+\dots+a_0$ where $a_i\in\mathbb Z$ by multiplying out the denominators, then $\operatorname{Tr}_{K/\mathbb Q}(s)=-\frac{a_{n-1}}{a_n}\frac{d}{n}$.
\end{lemma}
\begin{proof}
    We have the minimal polynomial of degree $n$, meaning that $1,s,s^2,\dots,s^{n-1}$ forms a basis of the field $\mathbb Q[s]$ over $\mathbb Q$, which is also a subfield of the whole $K$. Then let $b_1,b_2,\dots,b_m$ be a basis of $K$ viewed as a $\mathbb Q[s]$ vector space. Using the tower law for the tower
    $$\mathbb Q\subset\mathbb Q[s]\subset K$$
    we get that the $b_is^j$'s form a basis for $K$ as a $\mathbb Q$-vector space, where $i=1,\dots,m$ and $j=0,\dots,n-1$. (This also means that $mn=d$.) We will evaluate the trace of $x\mapsto sx$ with respect to this basis.

    $b_is^j\mapsto b_is^{j+1}$, which for $j=0,\dots,n-2$ means that the term on the diagonal of the matrix corresponding to $b_is^j$ is $0$. When $j=n-1$, then applying the minimal polynomial,
    $$b_is^{n-1}\mapsto b_is^n=-\frac{b_i}{a_n}(a_{n-1}s^{n-1}+a_{n-2}s^{n-2}+\dots+a_0)=\sum_{j=0}^{n-1}-\frac{a_j}{a_n}b_is^j$$
    and therefore taking the coefficient of $b_is^{n-1}$ we get that the diagonal term is $-\frac{a_{n-1}}{a_n}$. There are $m$ of these elements of the basis with $j=n-1$, so in total the trace is $-\frac{a_{n-1}}{a_n}\times m=-\frac{a_{n-1}}{a_n}\frac{d}{n}$. Note that $\frac{d}{n}$ is then an integer.
\end{proof}
\begin{cor}\label{traceint}
    If in the above, $s\in\mathcal O_K$, i.e. $s$ is an algebraic integer, then the minimal polynomial for $s$ with integer coefficients actually is monic, so $\operatorname{Tr}_{K/\mathbb{Q}}(s)=-a_{n-1}\frac{d}{n}\in\mathbb Z$.
\end{cor}

\begin{theorem}
\label{ring-of-integers-free-module}
    The ring of integers $\mathcal O_K$ of a number field $K$ is a finitely generated $\mathbb Z$-module, in particular it is actually isomorphic to $\mathbb{Z}^n$ where $n$ is the degree of the extension $K/\mathbb Q$, and a basis for $\mathcal O_K$ as a $\mathbb Z$-module forms a basis for $K$ as a $\mathbb Q$-vector space.
\end{theorem}

\begin{proof}
Let a basis for $K$ as a $\mathbb Q$-vector space be $x_1,x_2,\dots,x_n$. By scaling these by integers (Lemma \ref{scaletoint}), WLOG the elements of the basis are algebraic integers and hence in $\mathcal O_K$. Then consider the map from $K\rightarrow\mathbb Z^n$
$$x\mapsto(\operatorname{Tr}_{K/\mathbb Q}(xx_1),\operatorname{Tr}_{K/\mathbb Q}(xx_2),\dots,\operatorname{Tr}_{K/\mathbb Q}(xx_n))$$
This is in fact a linear map: if $x,y\in K$ and $z\in\mathbb Q$, then 

{\tiny
\begin{align*}
x+zy&\mapsto(\operatorname{Tr}_{K/\mathbb Q}((x+zy)x_1),\operatorname{Tr}_{K/\mathbb Q}((x+zy)x_2),\dots,\operatorname{Tr}_{K/\mathbb Q}((x+zy)x_n))\\
&=(\operatorname{Tr}_{K/\mathbb Q}(xx_1+zyx_1),\operatorname{Tr}_{K/\mathbb Q}(xx_2+zyx_2),\dots,\operatorname{Tr}_{K/\mathbb Q}(xx_n+zyx_n))\\
&=(\operatorname{Tr}_{K/\mathbb Q}(xx_1)+z\operatorname{Tr}_{K/\mathbb Q}(yx_1),\operatorname{Tr}_{K/\mathbb Q}(xx_2)+z\operatorname{Tr}_{K/\mathbb Q}(yx_2),\dots,\operatorname{Tr}_{K/\mathbb Q}(xx_n)+z\operatorname{Tr}_{K/\mathbb Q}(yx_n))\\
&=(\operatorname{Tr}_{K/\mathbb Q}(xx_1),\operatorname{Tr}_{K/\mathbb Q}(xx_2),\dots,\operatorname{Tr}_{K/\mathbb Q}(xx_n))+z(\operatorname{Tr}_{K/\mathbb Q}(yx_1),\operatorname{Tr}_{K/\mathbb Q}(yx_2),\dots,\operatorname{Tr}_{K/\mathbb Q}(yx_n))
\end{align*}
}

If there were some non-zero number $w\in K$ that gets mapped to the zero vector, then we know that $\operatorname{Tr}_{K/\mathbb Q}(wx_i)=0$ for each $i$. Then using the basis to express $w^{-1}=\sum_{i=1}^na_ix_i$ where $a_i\in\mathbb Q$,
\begin{align*}
\operatorname{Tr}_{K/\mathbb Q}(ww^{-1})=\operatorname{Tr}_{K/\mathbb Q}(1)=n&=\operatorname{Tr}_{K/\mathbb Q}\left(w\sum_{i=1}^na_ix_i\right)\\&=\operatorname{Tr}_{K/\mathbb Q}\left(\sum_{i=1}^nwa_ix_i\right)\\&=\sum_{i=1}^n\operatorname{Tr}_{K/\mathbb Q}(wa_ix_i)\\&=\sum_{i=1}^na_i\operatorname{Tr}_{K/\mathbb Q}(wx_i)=0
\end{align*}
so $n=0$ which is a contradiction. Therefore this map doesn't send non-zero to zero, so it's an injective linear map. Furthermore, when $x\in \mathcal O_K$, then $xx_i\in\mathcal O_K$ for all $i$, meaning that $\operatorname{Tr}_{K/\mathbb Q}(xx_i)\in\mathbb Z$ (Corollary \ref{traceint}), so the map actually bijects $\mathcal O_K$ with the image of $\mathcal O_K$ under the map, which forms a submodule of $\mathbb Z^n$. Any submodule of $\mathbb Z^n$ is isomorphic to $\mathbb Z^k$ for some $k\leq n$ (I need to prove this but I'm too lazy so I'll do it later), so indeed $\mathcal O_K$ is isomorphic to $\mathbb Z^k$.
\newline\newline
Now to prove that $k$ is actually the degree of $K/\mathbb Q$. We show that the set of generators $g_1,g_2,\dots,g_k\in\mathcal O_K\cong\mathbb Z^k$ forms a basis for the $\mathbb Q$-vector space $K$. For any number $r\in K$, we can scale it by some non-zero integer $t$ to get $tr\in\mathcal O_K$ (Lemma \ref{scaletoint}), which can then be expressed in terms of the generators
$$tr=\sum_{i=0}^kb_ig_i$$
where $b_i\in\mathbb Z$. Then
$$r=\sum_{i=0}^k\frac{b_i}{t}g_i$$
where $\frac{b_i}{t}\in\mathbb Q$. Therefore $g_1,g_2,\dots,g_k$ span $K$, so $k\geq n$. However we already have $k\leq n$, so $k=n$ and $g_1,g_2,\dots,g_n$ form a basis of $K$.
\end{proof}

Consider prime ideals $\mathfrak{p}\subset\mathcal{O}_K$. Intersecting $\mathfrak p\cap\mathbb Z$ produces a prime ideal of $\mathbb Z$. This is either the zero ideal or set of integer multiples of some prime number $p$. However, if $\mathfrak p$ is non-zero, i.e. it contains some $s\in\mathcal O_K\subset K,s\neq0$, since $K$ is a field we know that $s^{-1}\in K$, so now we can scale it by some integer $t$ to get an algebraic integer $ts^{-1}\in\mathcal O_K$, and hence $ts^{-1}s=t\in\mathfrak p\cap\mathbb Z$. Therefore $\mathfrak p\cap\mathbb Z$ is non-zero and indeed it is the set of multiples of $p$. We then say that $\mathfrak p$ "lies over $p$".
\newline\newline
Now this means that if we take the quotient $\mathcal O_K/\mathfrak p$, then the quotient map $\mathcal O_K\rightarrow\mathcal O_K/\mathfrak p$ takes $\mathbb Z$ to a subring of $\mathcal O_K/\mathfrak p$ isomorphic to $\mathbb F_p$ (since $\mathfrak p\cap\mathbb Z=p\mathbb Z$ and $\mathbb Z/p\mathbb Z=\mathbb F_p$), so we can say that $\mathbb F_p\subset\mathcal O_K/\mathfrak p$.

\begin{lemma}\label{Opfield}
    If we have a number field $K$, and a prime ideal $\mathfrak p$ of $\mathcal O_K$ lying over prime number $p$, then $\mathcal O_K/\mathfrak p$ is isomorphic to $\mathbb F_{p^k}$ for some positive integer $k$.
\end{lemma}

\begin{proof}
By \ref{ring-of-integers-free-module}, let $x_1,x_2,\dots,x_n$ be a basis for $\mathcal O_K$ as a $\mathbb Z$-module. Let $\varphi:\mathcal O_K\rightarrow\mathcal O_K/\mathfrak p$ be the quotient map (which is a surjective ring homomorphism). Then any element in $\mathcal O_K/\mathfrak p$ can be expressed as $\varphi(h)$ for some $h\in\mathcal O_K$. Then using the basis, for some $m_1,m_2,\dots,m_n\in\mathbb Z$,
$$\varphi(h)=\varphi(\sum_{i=1}^nm_ix_i)=\sum_{i=1}^n\varphi(m_ix_i)=\sum_{i=1}^n\varphi(m_i)\varphi(x_i)$$
and since $\varphi(m_i)\in\mathbb F_p$, then $\varphi(x_1),\varphi(x_2),\dots,\varphi(x_n)$ form a finite spanning set of $\mathcal O_K/\mathfrak p$ as a $\mathbb F_p$-vector space, so $\mathcal O_K/\mathfrak p$ has finite dimension. Then for any non-zero element $y\in\mathcal O_K/\mathfrak p$, if you consider the set ${1,y,y^2,\dots,y^{\operatorname{dim}_{\mathbb F_p}\mathcal O_K/\mathfrak p+1}}$, this set must be linearly dependent, which forms a polynomial in $\mathbb F_p[y]$ that evaluates to zero. Hence $y$ is algebraic over $\mathbb F_p$, so let $f\in\mathbb F_p[x]$ be the minimal polynomial of $y$ over $\mathbb F_p$. So $f$ has no factor of $x$, since otherwise you could divide it by $x$ and get a lower degree polynomial with $y\neq0$ as a root. Using Bezout's lemma on the polynomials $x,f(x)$, there exists polynomials $a,b$ with coefficients in $\mathbb F_p$ such that
$$a(x)f(x)+b(x)x=\gcd(f(x),x)=1$$
$$a(y)f(y)+b(y)y=1$$
since $f(y)=0$, then
$$b(y)y=1$$
so $b(y)$ is the multiplicative inverse of $y$. Hence $\mathcal O_K/\mathfrak p$ is a field, and since it's a finite-dimensional vector space over $\mathbb F_p$, then it must be $\mathbb F_{p^k}$ for some $k\in\mathbb Z^+$.
\end{proof}
This also shows the prime ideals are maximal.

Now let $G_{\Q}$ be the absolute Galois group $\operatorname{Gal}(\overline{\mathbb{Q}}/\mathbb{Q})$.

\begin{theorem}
    For any number field $K$, the ring of integers $\mathcal O_K$ is invariant under $\gal(K,\Q)$.
\end{theorem}

\begin{proof} $\Gal(K,\Q)K\subset K$ by definition, so it suffices to show $\gal(K,\Q)\mathbb{A}\subset\mathbb{A}$. Let $\alpha\in\A$, the $f(\alpha)=0$ for some monic $f\in\Z[x]$. If $\sigma\in\gal(K,\Q)$, then 
$$0=\sigma(0)=\sigma(f(\alpha))=f(\sigma\alpha).$$
Hence $\sigma\alpha\in\A$.
\end{proof}

$\Gal(K,\Q)$ also acts on the prime ideals of $\mathcal{O}_K$: if $ab\in\sigma(\mathfrak{p})$, then $\sigma^{-1}(ab)\in\mathfrak{p}$. Since $\sigma{-1}(ab)=\sigma^{-1}(a)\sigma^{-1}(b)\in\mathfrak{p}$, we have $\sigma^{-1}(a)\in\mathfrak{p}$ or $\sigma^{-1}(b)\in\mathfrak{p}$. Hence $a\in\sigma{\mathfrak p}$ or $b\in \sigma\mathfrak{p}$. Hence $\sigma{\mathfrak{p}}$ is a prime ideal. 
Generally the preimage of a prime ideal under a ring homormorphism is a prime ideal. Proof that this action is transitive is omitted in current version.

We want $\Gal(K,\Q)$ to act not just on the ring of integers but also on the quotient by $\mathfrak{p}$, which we know is a finite field by Lemma \ref{Opfield}. This naturally restricts us to the following subgroup:

\begin{definition} The decomposition group $D_{\mathfrak{p}}$ is the stabiliser of $\mathfrak{p}$ in the action of $\gal(K,\Q)$ on the prime ideals of $\mathcal{O}_K$. (spec $O_K$)

$$D_{\mathfrak{p}}=\{\sigma\in \Gal(K,\Q)\,|\, \sigma(\mathfrak{p})=\mathfrak{p}\}$$
    
\end{definition}
\begin{definition}
    
Define the inertia group as
$$I_{\mathfrak{p}}=\{\sigma\in \Gal(K,\Q)\,|\, \sigma(\alpha)=\alpha \mod \mathfrak{p} \forall \alpha\in\mathcal{O}_K\}\subset D_{\mathfrak{p}}.$$

\end{definition}

Restricting to $D$ gives us what we want:

\begin{thm}There is a (natural) group isomorphism
$$\varphi: D_{\mathfrak{p}}/I_{\mathfrak{p}}\to \Gal(\mathcal{O}_K/\mathfrak{p},\F_p)$$
   
\end{thm}

\begin{proof}
We take the action of $\Gal(K,\Q)$ on $\mathcal{O}_K$ and restrict it to $D_{\mathfrak{p}}$. This then gives an action on the quotient $\mathcal{O}_K/\mathfrak{p}$ by field automorphisms. Hence we have a group homomorphism
$$D_{\mathfrak{p}}\to \Gal(\mathcal{O}_K/\mathfrak{p},\F_p),$$
whose kernel (preimage of identity) is $I_{\mathfrak{p}}$. Hence we can factor through the quotient by the kernel to get an injective group homomorphism
$$\varphi: D_{\mathfrak{p}}/I_{\mathfrak{p}}\to \Gal(\mathcal{O}_K/\mathfrak{p},\F_p).$$
Surjectivity omitted in current version.
\end{proof}

Since $\Frob_p\in \Gal(\mathcal{O}_K/\mathfrak{p},\F_p)$, we can pull it back to define $\Frob_{\mathfrak{p}}\in D_{\mathfrak{p}}/I_{\mathfrak{p}}$:

\begin{definition}Define $\Frob_{\mathfrak{p}}\in D_{\mathfrak{p}}/I_{\mathfrak{p}}$ by
    $$\Frob_{\mathfrak{p}}=\varphi^{-1}(\Frob_p)$$
\end{definition}

To ultimately get something in the representation of the absolute Galois group, there are many hurdles of well-defined-ness to overcome. Here is a major one:

\begin{definition} A representation $\rho$ of $\Gal(K,\Q)$ is unramified at $\mathfrak{p}$ if $\rho(I_{\mathfrak{p}})=0$.
\end{definition}
This means that as far as $\rho$ is concerned, $\Frob_{\mathfrak{p}}$ resides in $\Gal(K,\Q)$.

In this current version, it remains to be checked that the Frobenius element is well-defined for $p$ rather than $\mathfrak{p}$ up to conjugacy and that having these for $K$ will lead to something in the absolute Galois group.

If you have read this far and want us to post the revised version of this bit, please contact us to speed us up.

%% file: n234ABC.tex
\section{Appendix 0: Outlines}

 \subsection{Base Cases for Fermat's Last Theorem}

\subsubsection{Pythagorean Triples: $n=2$} The problem of finding all integer Pythagorean triples is not usually considered as part of Fermat's Last Theorem but it is used to prove $n=4$ and is interesting on its own.

\begin{theorem}[Pythagorean Triples]\label{pythagdio}
     For all Pythagorean triples $(x,y,z)$ satisfying $x^2+y^2=z^2$ where $\gcd(x,y)=1$, there exists $a,b\in\mathbb{Z}$ with $\gcd(a,b)=1$ and $a,b$ of opposite parity such that $x=2ab, y=a^2-b^2, z=a^2+b^2.$ 
\end{theorem}

If $x$ and $y$ share greatest common divisor $d$, then $d$ divides $z$. So we can reduce the equation by dividing $x,y,z$ by $d$ to make all terms coprime. If $x,y$ are both odd, then  $x^2+y^2=z^2\equiv2 $ (mod $4$), which is impossible. So $x$ and $y$ have opposite parity.

Video Proof Link:

\subsubsection{$n=4$}

  Based on exposition in \cite{HW}.  
\begin{thm}\label{n4}
    There are no non-zero integer solutions to $x^4+y^4=z^4$. 

\end{thm}
We show this as a corollary of:

\begin{thm}
    There are no integer solutions to $x^4+y^4=Z^2$.

\end{thm}
\begin{proof}
    
Suppose $u$ is the smallest positive integer such that $x^4+y^4=u^2$ and $x,y>0$.  Then $\gcd(x,y)=1$, otherwise there is $d>1$ such that $d\mid x$ and $d\mid y$, so $d^4\mid x^4+y^4=u^2$ and hence $d^2\mid u$, so then $\left(\frac{x}{d}\right)^4+\left(\frac{y}{d}\right)^4=\left(\frac{u}{d^2}\right)^2$ where $\frac xd,\frac yd,\frac u{d^2}$ are integers and $\frac u{d^2}<u$, so the minimality of $u$ would be contradicted.
\newline\newline
Then we rewrite the equation as $(x^2)^2+(y^2)^2=u^2$ where $\gcd(x^2,y^2)=1$. We may now apply Theorem \ref{pythagdio}, that is our knowledge of the Pythagorean triples, to get that WLOG assuming $x^2$ is the even and $y^2$ is odd, there exists positive integers $a,b$ with $\gcd(a,b)=1$ such that $x^2=2ab$,  $y^2=a^2-b^2$ and $u=a^2+b^2$.
\newline\newline
As $y^2$ is odd, if $a$ is even, then $b$ would be odd, so $y^2=a^2-b^2\equiv0-1 \equiv -1$ (mod $4$), which is impossible. Hence, $a$ is odd and $b$ is even. Letting $b=2c$, we get $(x/2)^2=ac$, where since $x$ is even,  $x/2$ is an integer. As $gcd(a,c)=1$, we see that $a$ and $c$ are perfect squares. Letting $a=d^2$ and $c=f^2$ with $d,f>0$, we have $y^2=a^2-b^2=d^4-4f^4$ i.e. $y^2+(2f^2)^2=(d^2)^2$. As $2f^2=b$ which is coprime with $y$ (otherwise $a$ and $b$ would not be coprime), we may again apply Pythagoras' general result, obtaining $d^2=l^2+m^2$ and $f^2=lm$ with $l,m>0$ and $gcd(l,m)=1$. 
\newline\newline
From $y^2=lm$ and $gcd(l,m)=1$, we get as before that $l=s^2$ and $m=t^2$ for some positive integers $s,t$. Then from the first equation, $s^4+t^4=d^2$. But $0<d \leq a^2<a^2+b^2=u$ contradicts the assumption that $u$ is the minimal value whose square is the sum of two fourth powers. Thus, there are no non-zero solutions for $n=4$.

\end{proof}

\subsubsection{$n=3$}\label{n3} Based on Landau's proof as presented in \cite{HW}.
We are trying to find non-trivial integral solutions to the equation $x^3+y^3=z^3$. Let $\rho=e^{2\pi i/3}$ and consider the ring of integers adjoined with $\rho$, $k(\rho)$. Then, $x^3+y^3=(x+y)(x+\rho y)(x+\rho^2y)=z^3$. Also, define $\lambda=1-\rho$. Before we begin the proof, we have some definitions and lemmas to cover first.

\begin{definition} [Rational Prime]
    A rational prime is a prime over $\mathbb{Z}$. 
\end{definition}

\begin{definition} [Norm]
    The norm of $\alpha$ denoted by $N(\alpha)$ satisfies $N(\alpha)=\alpha\bar{\alpha}$ where $\bar{\alpha}$ is the complex conjugate of $\alpha$.
\end{definition}

\begin{lemma}
    The norm of $\alpha\in k(\rho)$ is always an integer.
\end{lemma}
\begin{proof}
Since $\alpha\in k(\rho)$, then there exists a polynomial $f\in\mathbb Z[x]$ such that $f(\rho)=\alpha$. $\rho^2+\rho+1=0$, so by monic polynomial division if we express $f(x)=(x^2+x+1)g(x)+bx+a$, we can express $\alpha=f(\rho)=(\rho^2+\rho+1)g(\rho)+b\rho+a=a+b\rho$ where $a,b\in\mathbb Z$. Then $N(\alpha)=\alpha\overline{\alpha}=(a+b\rho)\overline{(a+b\rho)}=(a+b\rho)(a+b\overline\rho)=a^2+b^2\rho\overline\rho+ab(\rho+\overline(\rho)=a^2+b^2-ab\in\mathbb Z$.
\end{proof}

\begin{lemma}
    If $\alpha\in k(\rho)$, then $\overline \alpha\in k(\rho)$.
\end{lemma}
\begin{proof}
    Firstly, $\overline\rho=\rho^2\in k(\rho)$. Then if $\alpha\in k(\rho)$ then there is some polynomial $f\in\mathbb Z[x]$ such that $\alpha =f(\rho)$. Therefore $\overline\alpha=\overline{f(\rho)}=f(\overline\rho)\in k(\rho)$.
\end{proof}

\begin{definition} [Unity]
    A unity in $k(\rho)$ is a number with norm $1$.
\end{definition}

Note that all unities have reciprocals in $k(\rho)$, namely if $u$ is a unity then $u^{-1}=\overline{u}\in k(\rho)$.
    
\begin{definition} [Associate]
    Two numbers $\alpha,\beta\in k(\rho)$ are associates if $\alpha=u\beta$ for some unity $u$.
\end{definition}
Note that this definition is symmetric, i.e. if $\alpha=u\beta$, then $\beta=u^{-1}\alpha$, where $u^{-1}$ is also a unity.

\begin{lemma} \label{Norm Prime}
    A number in $k(\rho)$ is prime if its norm is a rational prime.
\end{lemma}

\begin{proof}
    Suppose for the sake of contradiction that we have $\alpha,\beta, \gamma \in k(\rho)$ satisfying $\alpha=\beta\gamma$ where neither $\beta$ nor $\gamma$ are unities, and $N(\alpha)=p$ where $p$ is a rational prime. By the multiplicative condition of the norm, $N(\alpha)=N(\beta)N(\gamma)=p$. But $N(\beta)$ and $N(\gamma)$ are positive integers greater than 1, contradicting the primacy of $p$. 
\end{proof}

\begin{lemma}
    $\lambda$ is prime. 
\end{lemma}

\begin{proof}
    $N(\lambda)=N(1-\rho)=(1-\rho)(1-\rho^2)= 1+\rho^3-(\rho+\rho^2)=3$ which is a rational prime. By \ref{Norm Prime}, $\lambda$ is prime.
\end{proof}

\begin{lemma}\label{rho congruent}
    The integers of $k(\rho)$ are congruent to $0, \pm1$ (mod $\lambda)$
\end{lemma}

\begin{proof}
    For any integer $\omega \in k(\rho)$, we have $a,b\in \mathbb{Z}$ such that $\omega=a+b\rho=a+b-b\lambda\equiv a+b$ (mod $\lambda$). Since $\lambda\mid3$, and $a+b \equiv0,\pm1$ (mod $3$), $\omega\equiv0, \pm1$ (mod $\lambda)$ as required.
\end{proof}
\begin{lemma}\label{omega divide}
    For  $\omega\in k(\rho)$, if $\lambda$ does not divide $\omega$ then $\omega^3\equiv\pm1$ (mod $\lambda^4)$
\end{lemma}

\begin{proof}
    By \ref{rho congruent}, $\omega\equiv0,\pm1$ (mod $\lambda$), and since $\lambda$ does not divide $\omega$, $\omega\equiv\pm1$ (mod $\lambda$). We may therefore choose $\alpha=\pm\omega$ so that $\alpha \equiv1$ (mod $\lambda$), i.e $\alpha=1+\beta\lambda$. Then, 
    \begin{align*}
\pm(\omega^3\mp1)=\alpha^3-1&=(\alpha-1)(\alpha-\rho)(\alpha-\rho^2)\\
&=\beta\lambda(\beta\lambda+1-\rho)(\beta\lambda+1-\rho^2)\\
&=\beta\lambda(\beta\lambda+\lambda)(\beta\lambda+(1-\rho)(1+\rho))\\
&=\beta\lambda(\beta\lambda+\lambda)(\beta\lambda+\lambda(1+\rho))\\
&=\lambda^3\beta(\beta+1)(\beta+1+\rho)\\
&=\lambda^3\beta(\beta+1)(\beta-\rho^2)
    \end{align*}
    as $-\rho^2=\rho+1$. We know that $\rho^2=\rho^2-1+1=(\rho-1)(\rho+1)+1=-\lambda(\rho+1)+1\equiv1$ (mod $\lambda$), which implies that $\beta(\beta+1)(\beta-\rho^2)\equiv \beta(\beta+1)(\beta-1)$ (mod $\lambda$). One of these factors must be divisible by $\lambda$ by \ref{rho congruent}, and so $\pm(\omega^3\mp1)\equiv0$ (mod $\lambda^4$) and hence $\omega^3\equiv\pm1$ (mod $\lambda^4$).
\end{proof}

\begin{lemma}\label{three assoc}
    $3$ is an associate of $\lambda^2$
\end{lemma}

\begin{proof}
    $\lambda^2=1-2\rho+\rho^2=-3\rho$
\end{proof}
\begin{flushleft}
   Now suppose there exists $a,b,c \in k(\rho)$ with $abc\neq0$ such that $a^3+b^3+c^3=0$.  
\end{flushleft}

\begin{lemma}
   If $a^3+b^3+c^3=0$ then one of $a,b,c$ is divisible by $\lambda$
\end{lemma}

\begin{proof}
    Suppose the proposition is false. Then $0=a^3+b^3+c^3\equiv\pm1\pm1\pm1$ (mod $\lambda^4$) by \ref{omega divide}, and so $\pm1\equiv0$ or $\pm3\equiv0$ (mod $\lambda^4$). The first is impossible as $\lambda$ is not a unity, and the second impossible because 3 is an associate of $\lambda^2$ (\ref{three assoc}) and hence $\lambda^4$ cannot divide $3$. This contradiction yields the result.
    \end{proof}

    Therefore, we may assume without loss of generality that $\lambda\mid c$ and that $c=\lambda^n\gamma$ for $n\geq1$ and $\lambda$ does not divide $\gamma$. As gcd$(a,b,c)=1$, $\lambda$ does not divide $ab$. Our equation has now become $a^3+b^3+\lambda^{3n}\gamma^3=0$. Indeed, if we have some unity $\epsilon$, it is more convenient to prove $a^3+b^3+\epsilon\lambda^{3n}\gamma^3=0$ has no solutions.

    \begin{lemma}
        If $a,b,\gamma$ satisfy the above conditions, then $n\geq2$
    \end{lemma}

    \begin{proof}
        $-\epsilon\lambda^{3n}\gamma^3=a^3+b^3\equiv \pm1\pm1$ (mod $\lambda^4$). The signs cannot be the same, for $-\epsilon\lambda^{3n}\gamma^3\equiv \pm2$ (mod $\lambda^4$) is impossible as $\lambda$ does not divide $2$. Then $-\epsilon\lambda^{3n}\gamma^3\equiv 0$ (mod $\lambda^4$) implies $n\geq2$ as gcd$(\epsilon\gamma^3, \lambda)=1$.
    \end{proof}

    \begin{lemma}\label{final n=3}
        If $a^3+b^3+\lambda^{3n}\gamma^3=0$ holds for some $n=m>1$, then it holds for $n=m-1$.
    \end{lemma}

    The implication of this Lemma is that if our equation holds for any $n$, by the method of descent it will hold for $n=1$ which is impossible by Lemma 8.7. This will give the desired contradiction for the proof.

    \begin{proof}
        Substituting $n=m$, we get $$-\epsilon\lambda^{3m}\gamma^3=(a+b)(a+\rho b)(a+\rho^2b).$$ The differences of the factors on the right are $b\lambda, \rho b\lambda,\rho^2b\lambda$ which are each associates of $b\lambda$ and maximally divisible by $\lambda$ once (as $\lambda$ does not divide $b$). As $m>1$, $3m>3$, therefore exactly one of the factors must be divisible by $\lambda^2$, while the other two must be divisible by $\lambda$ but not $\lambda^2$. Assume without loss of generality that this factor is $a+b$ for if it were another factor, we could replace $b$ with one of its associates. 
       \newline\newline 
        We may now write $$a+b=\lambda^{3m-2}\kappa_1, a+b\rho=\lambda\kappa_2, a+b\rho^2=\lambda\kappa_3$$ where $\lambda$ does not divide $\kappa_i$. If $\delta\mid \kappa_2$ and $\delta\mid \kappa_3$, then $\delta$ divides $\kappa_2-\kappa_3=b\rho$ and $\rho\kappa_3-\rho^2\kappa_2=\rho a$, contradicting the co-prime condition of $a$ and $b$. Similar logic may be applied to $\lambda^{3m-3}\kappa_1, \kappa_2$ and $\lambda^{3m-3}\kappa_1, \kappa_3$ to obtain that gcd$(\kappa_1,\kappa_2,\kappa_3)=1$. Substituting into the equation, we get $-\epsilon\gamma^3=\kappa_1\kappa_2\kappa_3$ which implies that each of $\kappa_1,\kappa_2,\kappa_3$ is an associate of a cube.
        \newline\newline
        We now write $$a+b=\lambda^{3m-2}\kappa_1=\epsilon_1\lambda^{3m-2}\theta^3, a+b\rho=\epsilon_2\lambda\phi^3, a+b\rho^2=\epsilon_3\lambda\psi^3$$ where gcd$(\theta,\phi,\psi)=1$ and $\epsilon_1,\epsilon_2,\epsilon_3$ are unities. 
        \newline\newline
        Then, \begin{align*}
0 &= (1 + \rho + \rho^2)(a + b) \\
  &= a + b + \rho(a + b\rho) + \rho^2(a + b\rho^2) \\
  &= \epsilon_1 \lambda^{3m - 2} \theta^3 + \epsilon_2 \lambda \phi^3 + \epsilon_3 \lambda \psi^3
\end{align*} which we will simplify to $$\phi^3+\epsilon_4\psi^3+\epsilon_5\lambda^{3m-3}\theta^3$$ where $\epsilon_4=\epsilon_3\rho/\epsilon_2$ and $\epsilon_5=\epsilon_1/\rho\epsilon_2$ and hence they are both unities. 
\newline\newline
As $m\geq2$, $$
\phi^3 + \epsilon_4 \psi^3 \equiv 0 \pmod{\lambda^2}.
$$
But neither term is divisible by $\lambda$, so by \ref{omega divide}, $\phi^3\equiv\pm1$ (mod $\lambda^2$) and $\psi^3\equiv\pm1$ (mod $\lambda^2$). 
\newline\newline
Hence $$
\pm 1 \pm \epsilon_4 \equiv 0 \pmod{\lambda^2}.
$$
 But $\epsilon_4$ is $\pm1,\pm\rho,\pm\rho^2$. However, neither $\pm1\pm\rho$ or $\pm1\pm\rho^2$ is divisible by $\lambda^2$ since each is an associate of 1 or $\lambda$. Hence, $\epsilon_4=\pm1$. If $\epsilon_4=1$, then we have proved \ref{final n=3}. If $\epsilon_4=-1$, we may replace $\psi$ with $-\psi$. This result concludes the proof.
    \end{proof}

This approach was explored generally by Kummer but it was found that the kind of unique factorisation needed, which $n=3$ possessed, was not present for all primes. Those primes for which there was unique factorisation were called regular primes and Kummer was able to prove the theorem for these primes. It is an open problem whether or not there are infinitely many regular primes.

\subsubsection{WLOG $n$ is a Prime at Least 5} 
\begin{theorem}
    If Fermat's Last Theorem is proved true when $n$ is an odd prime or $4$ and $a,b,c$ are pairwise coprime, it must also be true in general.
\end{theorem}
\begin{proof}
    We first show that no $n>2$ solution exists. We will do two cases: $n$ is a power of $2$ and $n$ is not a power of $2$: 
    \newline
   \begin{enumerate}
       \item Suppose that $n=2^k$ satisfies Fermat's equation for some $k\in \Z_{\geq2}$. We have $$a^{2^k}+b^{2^k}=c^{2^k}$$ $$\iff (a^{2^{k-2}})^4+(b^{2^{k-2}})^4=(c^{2^{k-2}})^4.$$ But the case $n=4$ has no solutions, a proof of which can be found in the appendix?.
       \item Now suppose that $n=pr$ satisfies Fermat's equation for some $r\in \Z^+$ and odd prime $p$. Then $$a^{pr}+b^{pr}=c^{pr}$$ $$\iff (a^r)^p+(b^r)^p=(c^r)^p$$ which is only true if $n=p$ satisfies Fermat's equation. 
   \end{enumerate} 

    Now we will show that if no pairwise coprime solutions exist then no non-zero solutions exist in $(a,b,c)$. Fix the value of $n>2$ and assume for the sake of contradiction that a non-zero solution $(a,b,c)$ exists while no pairwise coprime solution exists. Firstly, the following holds as if any number divides 2 of $a,b,c$ it must divide the third: $$gcd(a,b,c)=gcd(a,b)=gcd(a,c)=gcd(b,c).$$ 
    Hence dividing both sides of the equation by $gcd(a,b,c)$ yields a new solution $${(\frac{a}{gcd(a,b,c)})}^n+{(\frac{b}{gcd(a,b,c)})}^n={(\frac{c}{gcd(a,b,c)})}^n$$ where $({\frac{a}{gcd(a,b,c)}},{\frac{b}{gcd(a,b,c)}},{\frac{c}{gcd(a,b,c)}})$ are pairwise coprime. But since we have assumed that no pairwise coprime solutions exist this obtains the desired contradiction.

\end{proof}


\subsection{The ABC Conjecture}\label{abc}
Define $\rad(n)$ to be the product of the distinct prime divisors of $n$.

\begin{conj}[The ABC Conjecture] For every $\epsilon>0$, there exists $K_{\epsilon}$ such that for any integers $a,b,c$ satisfying $a+b=c$, we have

$$c\leq K_{\epsilon}\rad(abc)^{1+\epsilon}$$
\end{conj}

If this were true, then Fermat's Last Theorem is true for sufficiently large $n$.

Proof: Suppose $x^n+y^n=z^n$ for non-zero coprime integers $x,y,z$. Then there exists $K_{0.5}$ such that

$$z^n\leq K_{0.5}\rad(x^ny^nz^n)^{1.5}=K_{0.5}\rad(xyz)^{1.5}\leq K_{0.5}(xyz)^{1.5}\leq K_{0.5}z^{4.5}$$
$$z^{n-4.5}\leq K_{\epsilon}$$

There are further conjectures giving an explicit bound on $K$ but without knowing $K$, it would not be known how many lower cases need to be considered. However if $K$ were also to be known to be $1$ or $2$ say then Fermat's Last Theorem would be a direct corollary.

%% file: appendElliptic.tex
\section{Appendix I: Elliptic Curves}

\subsubsection{$j$-invariant}
Unfortunately the $j$-invariant is zero for fields of characteristic $2$ or $3$ so in those cases it may be better to define it as simply $$\frac{c_4^3}{\Delta}.$$

\begin{theorem}
    When considering minimal Weierstrass equations, an elliptic curve $E$ has good reduction modulo $p$ if and only if the discriminant is non-zero modulo $p$. Otherwise, $E$ has bad reduction, having multiplicative reduction if $p$ does not divide $c_4$ and additive reduction if $p\mid c_4$.
\end{theorem}

\begin{proof}
    When $E$ is non-singular after being reduced modulo $p$ then it has good reduction at $p$. There are no repeated roots when $E$ is non-singular, so the discriminant becomes nonzero. For the other direction, if $\Delta\neq 0$, there are no repeated roots and so the curve is non-singular and has good reduction.
\newline\newline
    When $p\mid \Delta$, $E$ has bed reduction at $p$, and there is a root of multiplicity at least $2$. We use the change of variables to turn the minimal Weierstrass equation into the form $y^2=x^3+Ax+B$, as denoted above where $A=-4c_4$, when the characteristic is not $2$ or $3$. If $p\mid c_4$, then $p\mid A$ as $p\neq 2,3$. 
    But $$\Delta=-16(4A^3+27B^2)$$ so if $p\mid A,\Delta$ and $p\neq 2,3$, then $p\mid B$. But then we can reduce $E$ modulo $p$ again to obtain that $y^2=x^3$ which has a triple root and implies additive reduction. This logic is \textit{if and only if}, so if $p$ does not divide $c_4$, it does not have a triple root, and therefore has a double root. This implies multiplicative reduction.
\end{proof}

\subsection{Elliptic Curve as an Abelian Group}

\begin{proposition}
    The addition defined on an elliptic curve is associative. (From Lemma 3.1)  that is $(P+Q)+R= P+(Q+R)$ for all $P$, $Q$, $R$ on the elliptic curve.
\end{proposition}
 \begin{proof}This can be done just using Cartesian coordinates, but the algebra is very long so we will do another proof.
\newline\newline
    Case 1 ($P$ is $\mathcal{O}$): 
     $(\mathcal{O} +Q)+R=Q+R=\mathcal{O}+(Q+R).$  A similar proof works if any of $P,Q,R$ are $\mathcal{O}.$
    \newline\newline
    Case 2 ($PQ$ is a vertical line):  This case implies that $P+Q=\mathcal{O}$. It's required to prove now that $\mathcal{O} +R= P+(Q+R)$ if $P+Q=\mathcal{O}.$ 
    We already know $P+Q=\mathcal{O} \implies -P=Q$. Let the line $QR$ be $l$ and note the elliptic curve is called $E$. Then $l  \cap E = (Q,R, -(Q+R)).$ Reflect $l$ in the $x$-axis to $l'$, $l' \cap E=(-Q, -R, Q+R)$. But $-P=Q$ so we have $l' \cap E=(P, -R, Q+R)$. This means that $R=P+(Q+R)$ as required.
    \newline\newline
    Case 3: Let $A,B,C$ be three points on $E$ such that no three of them are the point at infinity and $A+B, B+C$ are not infinity (not vertical lines). We will now show the associativity condition for this case.
    \begin{proof}
     Define $AB \cap E =D, BC \cap E = Q$, and $D$ and $Q$ reflected over the $x$-axis to $G$ and $R$ respectively. Define $L=GC \cap AR$. To prove associativity it suffices to prove that $L$ lies on $E$.
    
    \includegraphics[scale=0.3]{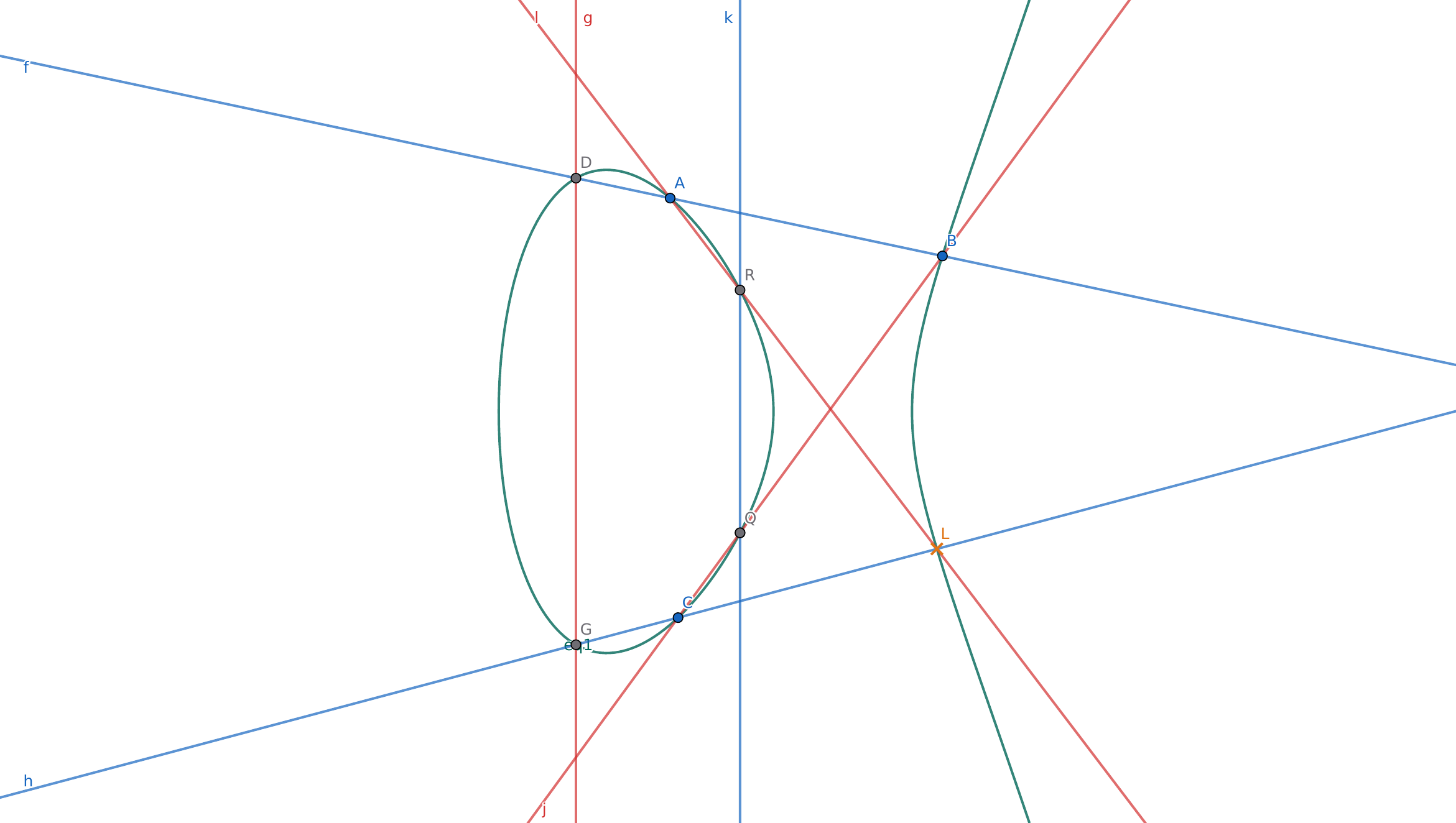}

    Now consider the cubic defined by the product of the equations of the three blue lines, let this be the cubic $C_1$, and the cubic defined by the product of the equations of the three red lines, let this cubic be $C_2$. Both these cubics pass through the 8 points on the above diagram (if some pair of the points are the same then multiplicity is counted) including the point at infinity (intersection point of $DG$ and $RQ$). If any of these points coincide then they contribute to the multiplicity of the intersections of the cubics at those points. Now by the Cayley-Bacharach theorem since $E$ passes through the seven points $D,G,A,R,Q,G,B$ and the point at infinity it also must pass through the $9$th point $L$ as it passing through $8$ of the intersection points of $C_1$ and $C_2$ means it must also pass through the $9$th. This proves associativity.
\end{proof}
    \end{proof}

\subsection{Generalised Definition of the Conductor $N_E$}

\subsubsection{Cusp and Node Singularities}
Before defining the conductor more generally, we take a moment to look at the graph below to see what singularities might look like. This part is not relevant to our discussion, but may help the reader get a more intuitive understanding of the reduction.
\newline\newline
     In the diagram, we call the point where the blue curve intersects itself a node, and the origin a cusp of the red curve. More explicitly, a point is a node if there are two tangents to the curve at that point, where the tangents are of different gradients.
    \newline\newline
A cusp on an algebraic curve is a singular point where the curve sharply changes direction and is not smooth. Near a cusp, the curve behaves like the graph of $y^2=x^3$, approaching a single tangent line from both sides with opposite orientation. The limit of the curve’s direction as one approaches the cusp from either side exists but points in opposite directions, meaning the curve bends back on itself. At this point, the curve is not differentiable and has no well-defined tangent.

\includegraphics[width=\textwidth]{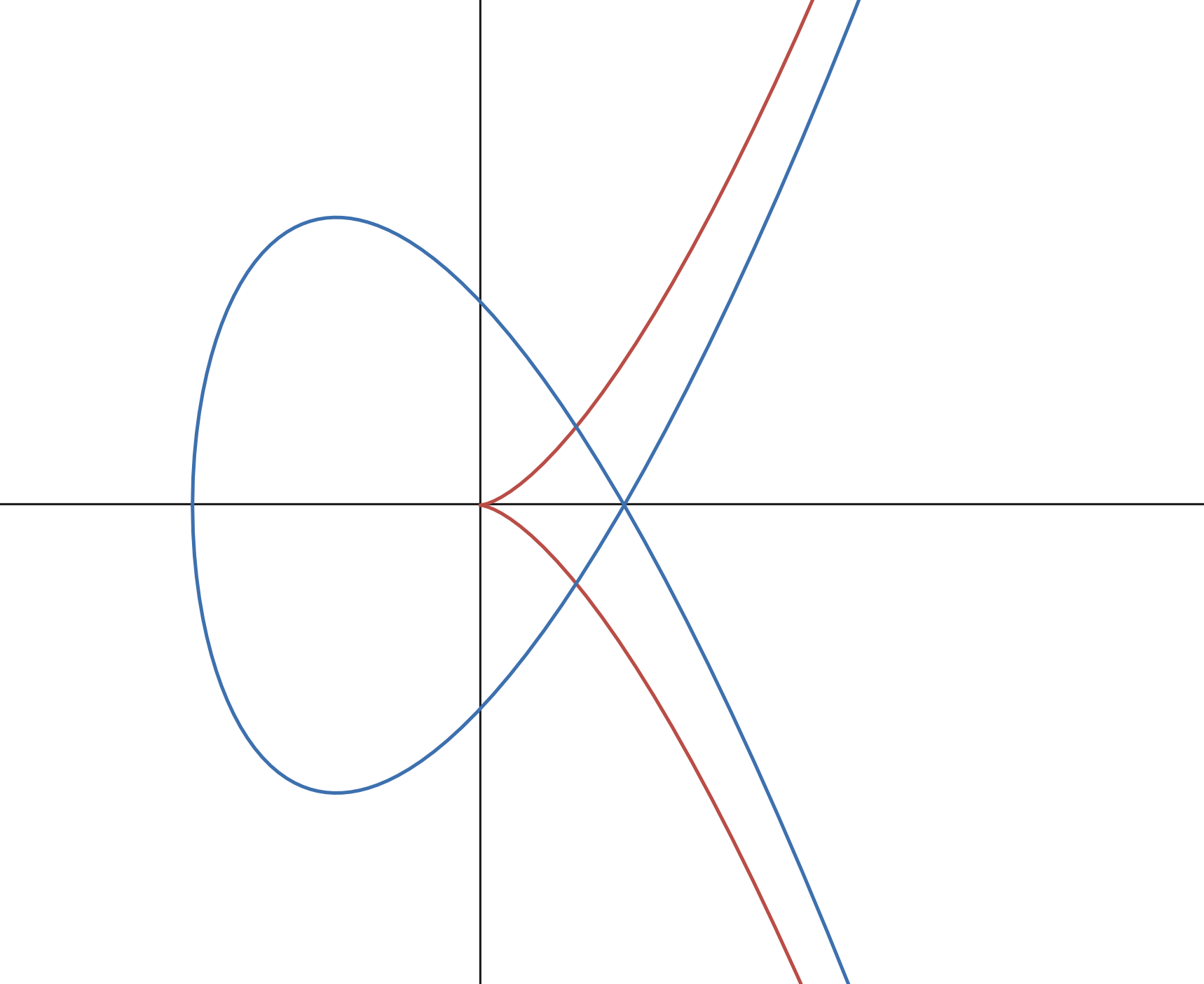}




The terminology of reduction, good, bad, additive and multiplicative are defined in the outline Definition \ref{reduction}.

\begin{itemize}
    \item For  multiplicative reduction, we get the analogue of nodal singularities.
    \item For additive reduction, we get analogue of  cusps.
    \item For good reduction, we get the analogue of smooth or non-singular.
\end{itemize}

\begin{definition}
    For each prime $p\in \mathbb{Z}$ we define $f_p$ as: 

  $f_p=$  
    $\begin{cases}
$0, if $E$ has good reduction at $p\\
$1,  if $E$ has multiplicative reduction at $p\\
$2, if $E$ has additive reduction at $p\neq2,3\\
$2+$\delta_p$, if $E$ has additive reduction at $p=2,3$
   $\end{cases}$
\newline\newline
   The definition of $\delta_p$ is quite complicated, and because of the semistable condition, all $f_p$ are either $0$ or $1$. So we will omit the definition here.  
\end{definition}

\begin{definition} [Conductor]\label{conductor}
We define the conductor 
$$N_{E/\mathbb{Q}} = \prod_{p\,\text{prime}} p^{f_p}.$$

\end{definition}
If the elliptic curve is semistable, all $f_p$ are either $0$ or $1$. Therefore, $N_{E/\mathbb{Q}}$ is a product of distinct primes and we recover the definition of conductor for semistable elliptic curves in the outline.


%% file: appendModular.tex
\section{Appendix II: Modular Forms}

\begin{definition} [Special Linear Group]
The special linear group of degree $n$ on a ring $R$, notated $\SL_n(R)$, is the group formed by the $n\times n$ matrices with entries in $R$ of determinant $1$ under the group operation of matrix multiplication.
\end{definition}

For a short proof that $\SL_n(R)$ is indeed a group refer to the appendix. For our discussions, $R$ will often be $\mathbb{Z}$ the ring of integers and $n=2$.

\begin{definition} [Projective Space]
    For a field $F$, $\boldsymbol{P}_n(F)$ denotes $(F^{n+1}\setminus\{0\})/\sim$ where $\sim$ is the equivalence relation defined by $v_1\sim v_2\iff v_1=\lambda v_2$.
\end{definition}
Generally people refer to a point in the projective space using the coordinates of a representative element of the equivalence class. 

\begin{definition} [Riemann Sphere]
The Riemann sphere, denoted $\hat{\mathbb{C}}$ is $\boldsymbol{P}_1(\mathbb{C})$
\end{definition}

\begin{lemma}
SL$_n(R)$ forms a group under matrix multiplication.
\end{lemma}

\begin{proof}
It is clear that multiplying $n\times n$ matrices of determinant $1$ will produce itself an $n\times n$ matrix of determinant $1$. This proves closure under the operation. Also, the fact that matrix multiplication is associative proves associativity. The identity matrix given by $1$'s on the main diagonal and all other entries being $0$ is an element of SL$_n(R)$ (and SL$_2(\mathbb{Z})$). Finally, the inverse of a matrix $A$ of determinant $1$ can be found as $B$ where
$$b_{ij}=(-1)^{i+j}\det(A_{ji})$$
where $A_{ji}$ is $A$ with the $j$th row and the $i$th column removed, so then the entry in the $i$-th row and $j$-th entry of $AB$ is given by
$$\sum_{k=1}^na_{ik}b_{kj}=\sum_{k=1}^n(-1)^{i+j}a_{ik}\det(A_{jk})$$
When $i=j$, this evaluates to the determinant of the original matrix expanded along row $j$, so it equals $1$, and when $i\neq j$, this evaluates to the determinant of the original matrix expanded along row $j$ except row $j$ has been replaced with row $i$, so now two of the rows are equal and hence the determinant is $0$. So $AB$ is indeed the identity matrix. A similar argument can be made to show $BA$ is the identity matrix. $B$ has entries in $R$ and $\det(B)=\det(A)\det(B)=\det(AB)=\det(I)=1$. Thus $B\in\operatorname{SL}_n(R)$ and so $B$ is indeed an inverse of $A$.
\end{proof}

\begin{proposition}
    The Mobius transformation defines a group action of $\SLZ$ on $\mathcal{H}$.
\end{proposition}
\begin{proof}
First we make sure that the Mobius transformations given by elements of $\SLZ$ are actually closed under $\mathcal{H}$. In the case that $c\neq0$, $-\frac{d}{c}$ is real so it is not in $\mathcal{H}$. Therefore in all cases,
\begin{align}
  \Mabcd(x+yi)&=\frac{a(x+yi)+b}{c(x+yi)+d}\\
&=\frac{(ax+b)+ayi}{(cx+d)+cyi}\\  
&=\frac{((ax+b)+ayi)((cx+d)-cyi)}{((cx+d)+cyi)((cx+d)-cyi)}\\
&=\frac{(acx^2+adx+bcx+bd+acy^2)+(ac+ad-ac-bc)i}{(cx+d)^2+c^2y^2}\\
&=\frac{(acx^2+adx+bcx+bd+acy^2)+(ad-bc)i}{(cx+d)^2+c^2y^2}
\end{align}
The imaginary part of this is $$\frac{ad-bc}{(cx+d)^2+c^2y^2}=\frac{1}{(cx+d)^2+c^2y^2}$$ since the numerator is just the determinant of the matrix, and all matrices in $\SLZ$ have determinant $1$. The denominator is positive, since it is equal to $|(cx+d)+cyi|^2$. which we know is non-zero due to $-\frac{d}{c}\notin\mathcal{H}$. Thus, the imaginary part is positive and the result is in the upper half-plane.
\newline\newline
To prove that this is a group action, notice that since the identity matrix fixes $\mathbb{C}^2$, it fixes $\hat{\mathbb{C}}$, so the Mobius transformation associated with the identity matrix is the identity function on $\mathcal{H}$. Also, since matrix multiplication is associative, if $M,N\in\SLZ$ and $Z\in\mathbb{C}^2$, then $M(NZ)=(MN)Z$ treating $Z$ as a $1\times2$ matrix, meaning that after taking the quotient on $\mathbb{C}^2$ to get $\hat{\mathbb{C}}$ and further restricting the domain to $\mathcal{H}$, it is still true that $M(Nz)=(MN)z$. This completes the proof.

\end{proof}

\begin{prop}$\wp$ is periodic with respect to $\mathcal L$
\end{prop}
\begin{proof}
    From the initial remarks we see that periodic essentially boils down to reordering the summands, which we know still converges. Another proof is in the Appendix.
\end{proof}

\begin{theorem} \label{extended-argument-principle}
    If $U$ is a simply connected open subset of $\mathbb C$, $f$ is a holomorphic function on $U\setminus\{z_1,z_2,\dots,z_n\}$ where $z_1,z_2,\dots,z_n$ are its poles, and $\gamma$ is a simple rectifiable closed curve with winding number $1$ around every point of its interior, and there are a finite number of roots inside the curve, then the integral
    $$\int_\gamma \frac{zf'(z)}{f(z)}dz$$
    is equal to $2\pi i$ times the difference between the sum of the roots and the sum of the poles inside the curve, counting multiplicity.
\end{theorem}
\begin{proof}
    For any root $z_0$ of $f$ inside the curve, if it has multiplicity $k$ then $f(z)=(z-z_0)^kg(z)$ where $g(z_0)\neq 0$. Then $f'(z)=k(z-z_0)^{k-1}g(z)+(z-z_0)^kg'(z)$. If $z_0$ is instead a pole inside the curve, the same is true except now $g$ is defined at the pole unlike $f$, and $k$ is the negative of the multiplicity of the pole. Then
    $$\frac{zf'(z)}{f(z)}=\frac{z(k(z-z_0)^{k-1}g(z)+(z-z_0)^kg'(z))}{(z-z_0)^kg(z)}=\frac{kz_0}{z-z_0}+k+\frac{zg'(z)}{g(z)}$$
    Since the zeroes of $f$ are isolated, there exists a small disk around $z_0$ where $g(z)\neq 0$, so $k+\frac{zg'(z)}{g(z)}$ is holomorphic in that disk. Hence, integrating around $z_0$ in that disk, we get the residue of $g$ at $z_0$ is $kz_0$. Applying the residue theorem gives the result.
\end{proof}

\subsubsection{Level 1 Forms}
eg Eisenstein series, comes from Weierstrass $\wp$
put these examples at the right places.

$$G_4=1+240\sum_{n=0}^\infty \sigma_3(n)q^n$$
weight 4
$$G_6=1-504\sum_{n=0}^\infty \sigma_5(n)q^n$$
weight 6

\begin{theorem} The space of modular forms is the graded ring
    $\C[G_4,G_6]$
\end{theorem}

See that the weights are all even and none of weight two

eg weight 12

$$\Delta=\frac{E_4^3-E_6^2}{1728}=q\prod(1-q^{-i})^{24}=\sum \tau(n)q^n$$

$\tau$ is known as the Ramanujan tau function

eg modular function

$$j=\frac{E_4^3}{\Delta}$$

this looks like discriminant and j invariant

\subsection{No Weight 2 Level 2 Cusp Forms}

Our original approach was just to mimic the proof in the level 1 case but we ran out of time to complete this.

\subsubsection{Fundamental Domain}

\begin{definition}
    A fundamental domain for the $\SLZ$-set $\mathcal{H}$ is a subset $D$ such that every point in $\mathcal{H}$ corresponds to a unique point in $D$ under the $\SLZ$ action.
\end{definition}

\begin{proposition}
    Let $D\subset\C$ be defined by the set of $\tau\in\mathcal{H}$ with real part between $1/2$ and $-1/2$, not inside the unit circle, with points on the boundary identified as follows: $z,z+1$ on the vertical lines, and $z\mapsto -1/z$ on the unit circle. Then this is a fundamental domain of $\mathcal{H}$ under $\SLZ$, i.e. you can reach all parts of $\mathcal{H}$ from $D$ using the group action of $\SLZ$.
\end{proposition}

\begin{proof}
First, we repeatedly apply the translation operation given by the matrix $\begin{pmatrix}
1 & \pm1 \\
0 & 1
\end{pmatrix}$ so that $\tau \mapsto \tau \pm1$ and replace this with $\tau$ until $|$Re$(\tau)|\leq1/2$ is satisfied. If $\tau$ is not in the unit circle, it is in the fundamental domain. Otherwise $|\tau|<1$, and Im$(-1/\tau))=$ Im $(-\bar{\tau}/|\tau|^2)>$ Im $(\tau)$. Hence, 
$\begin{pmatrix}
0 & -1 \\
1 & 0
\end{pmatrix}(\tau)=-1/\tau$ will strictly increase the imaginary component of $\tau$. As there are finitely many lattice points in the unit circle, this transformation will eventually take $\tau$ into the fundamental domain. 
\end{proof}
\subsubsection{Struggles With Showing That Something Does Not Exist}

The main issue we encountered in proving that no level 2 forms exist was a lack of clarity about what the precise statement was — that is, what exactly was supposed to *not exist*.
\newline\newline
Initially, we mistakenly thought the claim was that modular forms of weight 2 do not exist at all. But that, in fact, is a true theorem — namely, that non-zero modular forms of level 1 and weight 2 do not exist. However, this is not the statement we were trying to prove.
\newline\newline
The correct statement involves both weight 2 \textit{and} level 2. Furthermore, additional conditions such as being a cusp form or an eigenform created some confusion. We weren't always sure whether the non-existence claim applied just to level 2, weight 2 modular forms (i.e., those invariant under $\Gamma_0(2)$), or whether the hypothesis included the assumption that the form was a cusp form or even an eigenform.
\newline\newline
Our plan, ultimately, was to proceed in analogy with the level 1, weight 2 case — namely, to construct a fundamental domain and use integration techniques (e.g., computing the number of zeros) to derive a contradiction or to demonstrate the absence of such forms.
\newline\newline
Here is the sketch of the proof for level 1, we imagine level 2 is similar with a different fundamental domain.
\newline\newline
One question is how does cusp, eigenform,  coefficients in $\F_p$ affect this conclusion?

\begin{proposition}
    There are no weakly modular forms of weight $2$.
\end{proposition}
(For this, one method says we need to know how to get zeros using integrals.)
\begin{theorem}
    The number of zeros of a modular form of weight $k$ on a fundamental domain of $\mathcal{H}$ is $k/12$. Note we count zeros on the boundary using a fraction depending on the angle when thinking of the zero as a vertex.
\end{theorem}
\begin{proof}
    First see what happens where there are no zeros on the boundary.

    Now suppose there are zeros on the boundary.
\end{proof}

\begin{cor}
     There are no modular forms of weight $2$.
\end{cor}
\begin{proof}
If $k=2$, there are $1/6$ zeros in the fundamental domain.
    
\end{proof}

\subsection{No Non-Zero Weight 2 Level 2 $\Gamma_0(2)$ Cusp Form}

\subsubsection{Modular Forms for \( \Gamma_0(N) \)}

We found this formula for computing the dimension of $S_2(\Gamma_0(N))$ the space of cusp forms of weight 2 over $\Gamma_0(N)$ in \cite{wstein}. This book by W. Stein was particularly helpful in having hyperlinks that help navigate the text and logical dependencies.
\newline\newline
We then carry out a check of this in the case $N=2$.
\newline\newline
For any prime \( p \) and any positive integer \( N \), let \( v_p(N) \) be the maximal power of \( p \) that divides \( N \). Also, let (using Legendre symbols, with some abuse for $p=2$ evaluating to zero)
\[
\mu_0(N) = \prod_{p \mid N} \left(p^{v_p(N)} + p^{v_p(N)-1}\right),
\]
\[
\mu_{0,2}(N) = 
\begin{cases}
0 & \text{if } 4 \mid N, \\
\prod_{p \mid N} \left(1 + \left( \frac{-4}{p} \right) \right) & \text{otherwise},
\end{cases}
\]
\[
\mu_{0,3}(N) = 
\begin{cases}
0 & \text{if } 2 \mid N \text{ or } 9 \mid N, \\
\prod_{p \mid N} \left(1 + \left( \frac{-3}{p} \right) \right) & \text{otherwise},
\end{cases}
\]
\[
c_0(N) = \sum_{d \mid N} \varphi(\gcd(d, N/d)),
\]
\[
g_0(N) = 1 + \frac{\mu_0(N)}{12} - \frac{\mu_{0,2}(N)}{4} - \frac{\mu_{0,3}(N)}{3} - \frac{c_0(N)}{2}.
\]

Note that \( \mu_0(N) \) is the index of \( \Gamma_0(N) \) in \( \mathrm{SL}_2(\mathbb{Z}) \).

\medskip

\begin{prop}[Proposition 6.1\cite{wstein}]
    We have 
    $$\dim S_2(\Gamma_0(N)) = g_0(N)$$
\end{prop} 

\begin{thm} There are no non-zero cusp forms of weight 2 and level 2 over $\Gamma_0(2)$. That is, in the notation just used,
$$S_2(\Gamma_0(2))=\{0\}.$$

\end{thm}
\begin{proof}

We compute:
\[
g_0(2) = 1 + \frac{\mu_0(2)}{12} - \frac{\mu_{0,2}(2)}{4} - \frac{\mu_{0,3}(2)}{3} - \frac{c_0(2)}{2}.
\]

We evaluate each term:

\begin{align*}
\mu_0(2) &= 2^1 + 2^0 = 2 + 1 = 3, \\
\mu_{0,2}(2) &= \prod_{p \mid 2} \left(1 + \left( \frac{-4}{p} \right) \right) = 1 + \left( \frac{-4}{2} \right) = 1 + 0 = 1, \\
\mu_{0,3}(2) &= 0 \quad \text{(since } 2 \mid 2), \\
c_0(2) &= \varphi(\gcd(1,2)) + \varphi(\gcd(2,1)) = \varphi(1) + \varphi(1) = 1 + 1 = 2.
\end{align*}

Thus:
\[
g_0(2) = 1 + \frac{3}{12} - \frac{1}{4} - 0 - \frac{2}{2} = 1 + 0.25 - 0.25 - 1 = 0.
\]

Hence for \( k = 2 \), \( N = 2 \), the space \( S_2(\Gamma_0(2)) \) has dimension 0.

\end{proof}

\subsection{Modular Form for Frey's Curve}
It might be weird to talk about a curve that doesn't exist, but we had at some point tried to look at explicit ways to reach a contradiction.
\newline\newline
Here the students arrived at this formula using Legendre symbols, without realising this was also a general formula given in many text books.
\newline\newline
When RHS is $0$ modulo $p$, then $y=0$ (mod $p$) is the only solution. When RHS is a non-zero quadratic residue, then $y$ has two solutions. When RHS is a non-quadratic residue, then $y$ has no solutions. Thus, if we consider the Legendre symbol $\left( \frac{x(x-a^P)(x+b^P)}{p} \right) + 1$, this will give the number of solutions for a particular $x$. Then, \[
\sum_{x = 0}^{p - 1} [\left( \frac{x(x-a^P)(x+b^P)}{p} \right) +1]
\]  gives the number of solutions to the equation modulo $p$. We now have a  formula for the coefficients of the modular form, namely 
$$a_p = -\sum_{x = 0}^{p - 1} \left( \frac{x(x-a^P)(x+b^P)}{p}\right).$$

Richard Borcherds Modular forms course
\begin{verbatim}
https://youtu.be/f8pjt9ivjjc?si=UKD7bQyvuuff4dau
\end{verbatim}

%% file: fourier.tex
\subsection{Fourier Series}
\begin{definition}[Inner Product Space]
    An inner product space is a vector space $V$ over either $\mathbb{R}$ or $\mathbb{C}$ together with an inner product, a function taking 2 vectors to a scalar. The inner product must satisfy the following conditions:

1. The Inner Product is linear with the first vector in the pair: 
    
    $\langle c_1 u + c_2 v, w \rangle = c_1 \langle u, w \rangle + c_2 \langle v, w \rangle$

2. We have conjugate symmetry: $\langle u, v \rangle = \overline{\langle v, u \rangle}$.
Note that this reduces to symmetry over $\mathbb{R}$. (Note: this implies that $\langle v,v\rangle=\overline{\langle v,v\rangle}$ for all $v\in V$, meaning $\langle v,v\rangle\in\mathbb{R}$.

3. $\langle v,v \rangle > 0$ for $v\neq 0$

4. $\langle 0,0\rangle=0$

For square-integrable functions $f(x)$ and $g(x)$ in a domain [a,b], 

we define the inner product as
$\langle f, g \rangle = \int_{a}^{b} f(x) \overline{g(x)} \, dx$

\end{definition}

We can define a norm for the inner product space $|v|=\sqrt{\langle v,v\rangle}$. Using the Cauchy-Schwarz inequality, we can prove the triangle inequality for any inner product space with distance function $d(v,w)=|v-w|=\sqrt{\langle v-w, v-w\rangle}$, i.e. $|a-b|+|b-c|\geq|a-c|$. Equipped with this we can now define

\begin{definition}[Convergent Sequence]
    A convergent sequence is a sequence of vectors $\{v_i\}_{i=1}^\infty$ in the inner product space for which there exists a vector $\ell$ such that for all $\varepsilon>0$, there exists $N\in\mathbb{Z}^+$ such that for all $n>N$, $|v_n-\ell|<\varepsilon$. 
\end{definition}

Notice that if a sequence is convergent, then the $\ell$, known as the limit, is unique: If $\ell_1,\ell_2$ both work as $\ell$ for the same sequence, then for all $\varepsilon>0$, there exists $n,N$ with $n>N$ such that 
$$|\ell_1-\ell_2|\leq|\ell_1-v_n|+|v_n-\ell_2|=|\ell_1-v_n|+|\ell_2-v_n|<2\varepsilon$$
so $|\ell_1-\ell_2|=0$ and $\ell_1=\ell_2$.

\begin{definition} [Cauchy Sequence]
    A Cauchy sequence is a sequence of vectors $\{v_i\}_{i=1}^\infty$ in the inner product space where for all $\varepsilon>0$, there exists $N\in\mathbb{Z}^+$ such that for all $m,n>N$, $|v_m-v_n|<\varepsilon$. 
\end{definition}

\begin{lemma}
All convergent sequences are Cauchy.
\end{lemma}
\begin{proof}
  If a sequence $\{v_i\}_{i=1}^\infty$ is convergent, then it has a limit $\ell$. For all $\varepsilon>0$, let $\varepsilon_1=\frac{\varepsilon}{2}>0$, so then there exists some $N$ so that for all $m,n>N$, $|v_m-\ell|<\varepsilon_1$ and $|v_n-\ell|<\varepsilon_1$. Thus $|v_m-v_n|\leq|v_m-\ell|+|v_n-\ell|<2\varepsilon_1=\varepsilon$. Therefore it is Cauchy.
\end{proof}

Now we may finally define Hilbert space.
\begin{definition} [Hilbert Space]
    A Hilbert space is an inner product space that is "complete" in the sense that all Cauchy sequences are convergent.
\end{definition}

Thus in Hilbert spaces, a sequence is convergent if and only if it is Cauchy.

\begin{lemma}
The basis Fourier vectors are orthonormal.
\end{lemma}

\begin{proof}
The basis Fourier vectors are in the form $\{ e^{2\pi*inx} \mid n \in \mathbb{Z}$ \} on the interval [0,1] for $x$. It is clear that the norm is 1. Consider the inner product of two vectors $e^{2\pi*inx}$ and $e^{2\pi*imx}$.

We have $\langle e^{2\pi*inx},e^{2\pi*imx}\rangle = \int_{0}^{1} e^{2\pi*inx}\overline{e^{2\pi*imx}} \, dx.$ 

If $n=m$, this reduces to $\int_{0}^{1} 1 \, dx = 1.$

If $n \neq  m$, then we have $\int_{0}^{1} e^{2\pi*ix(n-m)} \, dx = 0$. 

Thus, all distinct Fourier vectors are orthogonal, and hence orthonormal.
\end{proof}

\begin{definition} [Square-Integrable Function]
 A function \( f: \mathbb{C} \to \mathbb{C} \) is said to be \textit{square-integrable} on an interval \( I \) if:

\[
\int_I |f(x)|^2 \, dx < \infty
\]

The set of all square-integrable functions on \( I \) is denoted by \( L^2(I) \):

\[
L^2(I) = \{ f : I \to \mathbb{C} \mid \int_I |f(x)|^2 \, dx < \infty \}
\]

https://planetmath.org/ProofThatLpSpacesAreComplete
\end{definition}

%% file: appendGalois.tex
\section{Appendix III: Galois representations}

There was another description that looked interesting and seems very similar but I am not sure exactly the difference or purpose:

Consider the finite field extension
$$\Q\subset \Q(E[\ell])$$
and take its Galois group

This is apparently known to be $\gl_2(\F_\ell)$ and then we have a natural map from $\galq$ to this Galois group giving us a $\F_\ell$ representation.

The determinant of Frobenius gives (idk) the eigenvalues of the so called diamond operators

%% file: ResearchProject.bbl
\begin{thebibliography}{10}



\bibitem{DDT}
{\sc H. Darmon and F. Diamond and R. Taylor}
(2000)
{\em An Introduction to Theory of Numbers 6th Edition}, \\
   Oxford at the Clarendon Press (1960).
\bibitem{HW}  
   {\sc G. H.~Hardy and E. M.~Wright}, 
   {\em An Introduction to Theory of Numbers 6th Edition}, \\
   Oxford at the Clarendon Press (1960).
\bibitem{Serre}  
   {\sc J.P.~Serre}, 
   {\em Sur les représentations modulaires de degré 2 de $Gal(\bar\Q/\Q)$}, \\
  Duke Math. J. 54(1): 179-230 (1987). \\
  DOI: 10.1215/S0012-7094-87-05413-5
\bibitem{Ribet}  
   {\sc K.~Ribet}, 
   {\em On modular representations of Gal(Q/Q) arising from modular forms}. Inventiones Mathematicae. 100 (2): 431–476.  (1990) doi:10.1007/BF01231195. MR 1047143. S2CID 120614740., \\
\bibitem{}  
   {\sc F. Diamond and J. Shurman}, 
   {\em A First Course in Modular Forms}, \\
   Springer (2005).
\bibitem{Wiles}  
   {\sc A. J. Wiles }, 
   {\em Modular elliptic curves and Fermat's Last Theorem}, \\
   Annals of Mathematics, 141 (1995).

 \bibitem{Frey} (1982), {\sc G.~Frey}, Rationale Punkte auf Fermatkurven und getwisteten Modulkurven" [Rational points on Fermat curves and twisted modular curves], J. Reine Angew. Math. (in German), 1982 (331): 185–191, doi:10.1515/crll.1982.331.185, MR 0647382, S2CID 118263144
   \bibitem{Junior} 
{\sc M. ~Sun}, 
{\em Maths Olympiad theory building and problem solving from Junior to Intermediate},
Dr Michael Sun's School of Maths (2023),
\bibitem{Inter}
{\sc M. ~Sun}, 
{\em Maths Olympiad theory building and problem solving from Intermediate to Senior},
Dr Michael Sun's School of Maths (2023).
\bibitem{Senior}
{\sc M. ~Sun}, 
{\em Maths Olympiad theory building and problem solving from Senior to Professional},
Dr Michael Sun's School of Maths (2023).
 
\bibitem{HSM}  
   {\sc M.~Sun}, 
   {\em Higher School Maths}, 
   Dr Michael Sun's School of Maths (2018).

\bibitem{Serre Arithmetic}
{\sc J.P.~Serre}, 
   {\em A Course in Arithmetic}, 
   Springer (2012).

\bibitem{Serre Letter}
{\sc J.P.~Serre} (1985),
{\em Lettre a J.-F. Mestre},
In K. Ribet (Ed.), {\em Current Trends in Arithmetical Algebraic Geometry} (pp. 263-268). American Mathematical Society. 


\bibitem{Silverman}
{\sc J.H.~Silverman}, 
   {\em The Arithmetic of Elliptic Curves}, 
   Springer (2013).

\bibitem{Intro to Theory of Numbers}  
   {\sc E.M. Wright and G.H. Hardy}, 
   {\em An Introduction to the Theory of Numbers}, 
   Oxford  (2008).


\bibitem{MF & FLT}  
   {\sc G. Cornell and J.H. Silverman and G. Stevens}, 
   {\em Modular Forms and Fermat's Last Theorem}, 
   Springer  (1997).

  \bibitem{BCDT}
  {\sc Breuil,  Conrad,  Diamond, Taylor} (2001), "On the modularity of elliptic curves over Q: wild 3-adic exercises", Journal of the American Mathematical Society, 14 (4): 843–939, doi:10.1090/S0894-0347-01-00370-8, ISSN 0894-0347, MR 1839918

\bibitem{SMC}
    Khare, Chandrashekhar (2006), "Serre's modularity conjecture: The level one case", Duke Mathematical Journal, 134 (3): 557–589, doi:10.1215/S0012-7094-06-13434-8.
 Khare, Chandrashekhar; Wintenberger, Jean-Pierre (2009), "Serre's modularity conjecture (I)", Inventiones Mathematicae, 178 (3): 485–504, Bibcode:2009InMat.178..485K, CiteSeerX 10.1.1.518.4611, doi:10.1007/s00222-009-0205-7 and Khare, Chandrashekhar; Wintenberger, Jean-Pierre (2009), "Serre's modularity conjecture (II)", Inventiones Mathematicae, 178 (3): 505–586, Bibcode:2009InMat.178..505K, CiteSeerX 10.1.1.228.8022, doi:10.1007/s00222-009-0206-6.
 Khare, Chandrashekhar; Wintenberger, Jean-Pierre (2009), "On Serre's reciprocity conjecture for 2-dimensional mod p representations of Gal(Q/Q)", Annals of Mathematics, 169 (1): 229–253, doi:10.4007/annals.2009.169.229.

 Khare, Chandrashekhar (2006), "Serre's modularity conjecture: The level one case", Duke Mathematical Journal, 134 (3): 557–589, doi:10.1215/S0012-7094-06-13434-8.
 Khare, Chandrashekhar; Wintenberger, Jean-Pierre (2009), "Serre's modularity conjecture (I)", Inventiones Mathematicae, 178 (3): 485–504, Bibcode:2009InMat.178..485K, CiteSeerX 10.1.1.518.4611, doi:10.1007/s00222-009-0205-7 and Khare, Chandrashekhar; Wintenberger, Jean-Pierre (2009), "Serre's modularity conjecture (II)", Inventiones Mathematicae, 178 (3): 505–586, Bibcode:2009InMat.178..505K, CiteSeerX 10.1.1.228.8022, doi:10.1007/s00222-009-0206-6

\bibitem{wstein}
{\sc William Stein}
Modular Forms, a Computational Approach
  University of Washington, Seattle, WA
\end{thebibliography}
